\documentclass[a4paper,11pt]{article}

\usepackage[table,dvipsnames,svgnames,x11names]{xcolor}
\usepackage[utf8]{inputenc}
\usepackage[T1]{fontenc}
\usepackage{lmodern}
\usepackage{microtype}
\usepackage[a4paper,margin=1in]{geometry}
\usepackage{amsmath,amssymb,amsthm}
\usepackage{mathtools}
\usepackage{array,booktabs,multirow}
\usepackage{graphicx}
\usepackage{float}
\usepackage{caption}
\usepackage{subcaption}
\usepackage{enumitem}
\usepackage{fancyvrb}
\usepackage{tikz}
\usepackage{tikz-cd}
\usepackage[colorlinks=true,allcolors=MidnightBlue]{hyperref}
\usepackage{placeins}

\usepackage{aliascnt}

\newtheorem{theorem}{Theorem}[section]

\newaliascnt{proposition}{theorem}
\newtheorem{proposition}[proposition]{Proposition}
\aliascntresetthe{proposition}

\newaliascnt{lemma}{theorem}
\newtheorem{lemma}[lemma]{Lemma}
\aliascntresetthe{lemma}

\newaliascnt{corollary}{theorem}
\newtheorem{corollary}[corollary]{Corollary}
\aliascntresetthe{corollary}

\newaliascnt{conjecture}{theorem}

\aliascntresetthe{conjecture}

\theoremstyle{definition}

\newaliascnt{definition}{theorem}
\newtheorem{definition}[definition]{Definition}
\aliascntresetthe{definition}

\newaliascnt{example}{theorem}
\newtheorem{example}[example]{Example}
\aliascntresetthe{example}

\newaliascnt{construction}{theorem}
\newtheorem{construction}[construction]{Construction}
\aliascntresetthe{construction}

\theoremstyle{remark}

\newaliascnt{remark}{theorem}
\newtheorem{remark}[remark]{Remark}
\aliascntresetthe{remark}

\newaliascnt{question}{theorem}

\aliascntresetthe{question}

\usepackage[nameinlink,noabbrev]{cleveref}

\crefname{theorem}{Theorem}{Theorems}
\Crefname{theorem}{Theorem}{Theorems}

\crefname{proposition}{Proposition}{Propositions}
\Crefname{proposition}{Proposition}{Propositions}

\crefname{lemma}{Lemma}{Lemmas}
\Crefname{lemma}{Lemma}{Lemmas}

\crefname{corollary}{Corollary}{Corollaries}
\Crefname{corollary}{Corollary}{Corollaries}

\crefname{conjecture}{Conjecture}{Conjectures}
\Crefname{conjecture}{Conjecture}{Conjectures}

\crefname{definition}{Definition}{Definitions}
\Crefname{definition}{Definition}{Definitions}

\crefname{example}{Example}{Examples}
\Crefname{example}{Example}{Examples}

\crefname{construction}{Construction}{Constructions}
\Crefname{construction}{Construction}{Constructions}

\crefname{remark}{Remark}{Remarks}
\Crefname{remark}{Remark}{Remarks}

\crefname{question}{Question}{Questions}
\Crefname{question}{Question}{Questions}

\DeclareMathOperator{\sd}{sd}

\DeclareMathOperator{\id}{id}

\title{Finite Topological Space Filtrations:\\
A Topological Framework for Data Analysis}
\author{Sel\c{c}uk Kayacan}
\date{}

\begin{document}

\maketitle

\begin{center}
Bah\c{c}e\c{s}ehir University, Faculty of Engineering and Natural Sciences,\\
Istanbul, Turkey\\
\emph{e-mail:} \href{mailto:selcuk.kayacan@bau.edu.tr}{selcuk.kayacan@bau.edu.tr}
\end{center}

\begin{abstract}
We introduce a data-analysis framework based on filtrations of finite topological spaces. Starting from a finite metric data set, we construct a sequence of coarsening topologies on the same set of points. These topologies give persistence modules and barcodes in the usual way, but they also retain information that is lost when the filtration is reduced to homology. At each level one can examine, for example, which points are topologically indistinguishable, how their minimal neighbourhoods overlap, how connected components merge, and how these features change from one level to the next.

We develop the basic theory of these filtrations, establish stability results under suitable hypotheses, and give practical constructions starting directly from a distance matrix. We then study what can be learned from the resulting finite topologies. On synthetic data with known clusters of different shapes, sizes, and densities, we examine how these regions appear among the finite-topological structures and how they merge as the topology coarsens. We also study what happens when points that become uncovered early in the construction are removed and the analysis is repeated. For one-dimensional homology, we use paths in the finite-topological structure to locate cycles and to examine how their appearance is related to the geometry of the data.

We finally apply these ideas to two real data sets with quite different structures. On the Paul15 single-cell data, we use the evolving finite topology to examine fine cellular states, their overlaps and relations, their assembly into larger groups, and the effect of removing points that connect these structures. On COIL20, where images of an object are sampled through a full rotation, we study how the cyclic organization of the images is reflected in the finite-topological evolution and in the associated one-dimensional homology.


\end{abstract}

\tableofcontents

% ====================================================================
\section{Introduction}\label{sec:introduction}
% ====================================================================

Persistent homology studies the evolution of homological information along a filtration. In many data-analytic settings, however, the filtration itself contains structural information that is not retained by its homological summaries. We introduce finite topological space filtrations as a framework in which the primary object is a controlled coarsening of finite topologies on a fixed data set. The topology at each level can be examined directly, while its $T_0$-quotient, associated poset, order complex, homology, and other derived objects provide complementary descriptions of the same evolution.

The mathematical framework is developed independently of a particular data-analysis task. We first define abstract FTS filtrations and their enriched evolution posets, and then study rule-generated and anchor-generated constructions. The latter provide finite practical representations of richer rule systems and lead to metric-guided implementations based on local candidate sets, A-set removal, and a B-set mechanism for uncovered points. The resulting filtrations are functorial through the identity map on the underlying data set and therefore induce persistence modules and barcodes.

Stability is treated at both theoretical and empirical levels. Local-refinement hypotheses give conditional interleaving results for rule-generated A-stages, full B-memory gives a monotonicity result for fixed anchor dictionaries, and Quillen-fibre conditions provide a reference criterion for transferring homological information to anchor-generated filtrations. Since the practical profile and adaptive constructions are sample-dependent, their robustness is also studied directly under controlled perturbations through finite-topological discrepancies, resolution diagnostics, and bottleneck distances.

The later sections use the additional structure retained by an FTS filtration rather than treating persistent homology as its sole output. \Cref{sec:clustering} studies finite-topological structure and clustering on the synthetic \texttt{mixed\_bumps} data and on Paul15, using the evolution of components, $T_0$-classes, minimal opens, core reductions, and removal. On \texttt{mixed\_bumps}, we also compare the resulting structures with ToMATo, first through localized $T_0$-classes and then through separating supports derived from minimal opens. \Cref{sec:incidence-obstruction} studies incidence-generated $H_1$ and geometric obstruction structure on \texttt{mixed\_holes} and COIL20. These two sections may be viewed as complementary investigations driven respectively by $\beta_0$- and $\beta_1$-behaviour: rather than retaining only ranks or birth and death information, we examine what the underlying finite-topological evolution reveals about the organization responsible for that behaviour.

\Crefrange{sec:background}{sec:stability} develop the required background, the abstract and practical constructions, and the stability theory and experiments. \Cref{sec:clustering,sec:incidence-obstruction} then give the synthetic and real-data analyses, followed by the conclusion.

The constructions described in this paper are implemented in SageMath in the code repository accompanying the paper,
\begin{quote}
\url{https://github.com/cg-istanbul/fts-filtrations}.
\end{quote}
The repository contains the common filtration pipeline, the finite-topological and homological constructions derived from it, and experiment-specific runners for the data sets studied in the paper. It also contains additional materials supporting the computational analyses, including complete stability results, algorithmic and implementation details, and extended comparisons whose full diagnostics are omitted here for space.

% ====================================================================
\section{Background}\label{sec:background}
% ====================================================================

This section recalls the background material needed in the rest of the paper. We review posets and order complexes, finite topological spaces, and the basic language of filtrations and persistence modules. Along the way, we fix the conventions and record the results that will be used later.

%--------------------------------------------------------------------
\subsection{Posets and order complexes}\label{subsec:order-complexes-background}
%--------------------------------------------------------------------

In this subsection we review the basic material on posets and order complexes. For a systematic treatment of poset combinatorics we refer to \cite{Sta12}. The passage from posets to simplicial complexes is provided by the order-complex construction. For the simplicial and algebraic-topological background used here, we refer to \cite{MKP25}. For broader accounts of topological methods in combinatorics and of the tools of poset topology, respectively, see \cite{Bjo95} and \cite{Wac07}. Our exposition largely follows \cite{Smi11}, whose treatment of posets and order complexes is particularly well suited to the viewpoint adopted in this paper.

%--------------------------------------------------------------------
\subsubsection*{Posets}
%--------------------------------------------------------------------

A set $P$ together with a reflexive and transitive relation $\leq$ is called a \emph{preordered set}. If, further, the relation is antisymmetric, we say $P$ is a \emph{partially ordered set} or simply a \emph{poset}. A poset in which every two elements are comparable is called a \emph{totally ordered set} or simply a \emph{chain}. Let $P$ be a poset equipped with the relation $\leq$. For two elements $x$ and $y$ of $P$ with $x\leq y$, we write
$$ [x,y] := \{z\in P\colon x\leq z\leq y\}\qquad \text{and}\qquad (x,y) := \{z\in P\colon x < z < y\}. $$
An element $y$ \emph{covers} an element $x$ if $x < y$ and $[x,y]=\{x,y\}$. In the \emph{Hasse diagram} of a poset, if $x < y$, then $y$ is drawn above $x$, and an edge is drawn from $x$ to $y$ precisely when $y$ covers $x$. In \Cref{fig:hasse_b3} the Hasse diagram of the powerset of a set of three elements is depicted. 

\begin{figure}[!ht]
  \centering
  \includegraphics[width=0.35\textwidth]{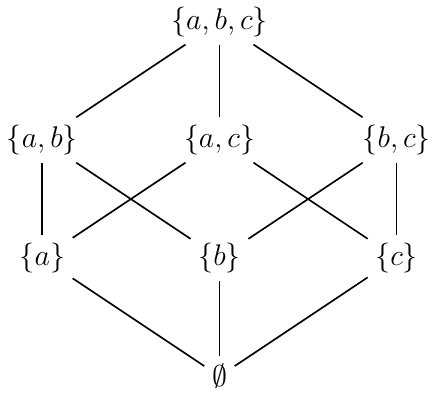}
  \caption{The Hasse diagram of the poset of subsets of the set $\{a,b,c\}$ equipped with the containment relation $\subseteq$.}
  \label{fig:hasse_b3}
\end{figure}

For a poset $P$ and an element $x\in P$, the \emph{principal order ideal} $P_{\leq x}$ and the \emph{principal dual order ideal} $P_{\geq x}$ are the subposets of $P$ defined as  
$$ P_{\leq x} := \{y\in P\colon y\leq x\}\qquad \text{and}\qquad P_{\geq x} := \{y\in P\colon y\geq x\}. $$
In an analogous way, we define the subposets $P_{< x}$ and $P_{> x}$. Notice that $[x,y]=P_{\geq x}\cap P_{\leq y}$ and $(x,y)=P_{> x}\cap P_{< y}$.
    
Let $P$ and $Q$ be partially ordered sets. A map $f\colon P\to Q$ is a \emph{poset map} or a \emph{poset morphism} if $x\leq y$ implies $f(x)\leq f(y)$. In other words, a poset morphism is just an order-preserving map between two posets.

%--------------------------------------------------------------------
\subsubsection*{Order complexes}
%--------------------------------------------------------------------

An \emph{abstract simplicial complex} is a collection $K$ of nonempty sets $\sigma$, called the \emph{simplices of $K$}, such that the subset-closure property holds, i.e., if $\sigma\in K$ then every nonempty $\tau\subseteq\sigma$ also belongs to $K$. The \emph{vertex set} of $K$, denoted $V(K)$, is the union of singletons in $K$. Let $K$ and $L$ be two simplicial complexes. A function $f\colon V(K)\to V(L)$ is a \emph{simplicial map} if $f(\sigma)\in L$ whenever $\sigma\in K$. If $f$ is a simplicial map, we simply write $f\colon K\to L$. Notice that a simplicial map preserves the face relation, i.e., $\sigma\subseteq\tau$ implies $f(\sigma)\subseteq f(\tau)$ for every $\sigma,\tau\in K$. Given an abstract simplicial complex $K$ one can construct a geometric simplicial complex $|K|$, called the \emph{geometric realization} of $K$, in a standard fashion.

\begin{lemma}[Functoriality of $|\cdot|$, see {\cite[Lemma~1.4.5]{Smi11}}]
  If $f\colon K\to L$ is a simplicial map of complexes, then $f$ induces a continuous map $f\colon |K|\to |L|$.
\end{lemma}

To any poset $P$ one can associate its \emph{order complex} $\mathcal{K}(P)$ which is the abstract simplicial complex of all totally ordered subsets (chains) of $P$. In \Cref{fig:poset_order} the order complex of a poset is presented.

\begin{figure}[!ht]
  \centering
  \includegraphics[width=0.65\textwidth]{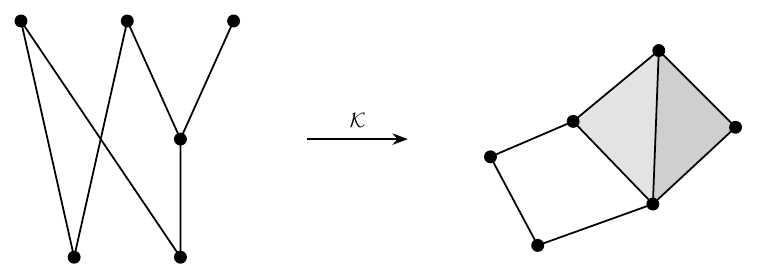}
  \caption{The Hasse diagram of a poset is given on the left. On the right its order complex is depicted.}
  \label{fig:poset_order}
\end{figure}

\begin{lemma}[Functoriality of $\mathcal{K}$, see {\cite[Lemma~1.4.7]{Smi11}}]
  Let $f\colon P\to Q$ be a poset map. Then $f$ induces a simplicial map $f\colon \mathcal{K}(P)\to \mathcal{K}(Q)$. Consequently, the map $f\colon |\mathcal{K}(P)|\to|\mathcal{K}(Q)|$ between geometric realizations is continuous.
\end{lemma}

For a simplicial complex $K$ its \emph{barycentric subdivision} $\sd(K)$ is the abstract simplicial complex whose simplices are the chains of simplices of $K$ under the face relation $\subseteq$.

\begin{lemma}[see {\cite[Lemma~1.6.4]{Smi11}}]
  For a simplicial complex $K$, $|\sd(K)|$ is homeomorphic to $|K|$.
\end{lemma}

The \emph{face poset} $\mathcal{P}(K)$ of a simplicial complex $K$ is the poset of all simplices of $K$ ordered by the face relation $\subseteq$.

\begin{lemma}[see {\cite[Lemma~1.6.7]{Smi11}}]
  For a simplicial complex $L$, $$\sd(L)=\mathcal{K}(\mathcal{P}(L)).$$
\end{lemma}

By the previous Lemma the barycentric subdivision of a simplicial complex is the order complex of its face poset. This motivates the following definition: For a poset $Q$ its \emph{barycentric subdivision} $\sd(Q)$ is defined as $$\sd(Q):=\mathcal{P}(\mathcal{K}(Q)).$$ 

%--------------------------------------------------------------------
\subsubsection*{Homotopy property}
%--------------------------------------------------------------------

Let $X$ and $Y$ be topological spaces and $I$ be the interval $[0,1]$ of real numbers with the standard topology. Let $f$ and $g$ be continuous maps from $X$ into $Y$. We say $f$ is \emph{homotopic} to $g$, and write $f\simeq g$, if there is a continuous map $F\colon X\times I\to Y$ such that $$ F(x,0) = f(x)\;\; \text{ and }\;\; F(x,1) = g(x)\quad \text{for all } x\in X. $$
The map $F$ is called a \emph{homotopy} between $f$ and $g$. If $f$ is homotopic to a constant map,  we say $f$ is \emph{nullhomotopic}.

Let $P$ and $Q$ be posets. If $f$ and $g$ are poset maps from $P$ to $Q$ such that $$ f(x)\leq g(x) \;\;\text{ for all } x\in P, $$ then we say $g$ \emph{dominates} $f$ and write $f\leq g$.

\begin{lemma}[see {\cite[1.3~Homotopy~property]{Qui78}} and {\cite[Lemma~3.1.7]{Smi11}}]
  \mbox{}\\[-3ex]
  \begin{enumerate}
  \item If $f,g\colon P\to Q$ are poset maps such that $f\leq g$, then $f\simeq g$, i.e., the induced continuous maps are homotopic.
  \item In particular, for a poset endomorphism $f$ with $f\leq \id_{P}$ (or with $\id_{P}\leq f$), we have $f\simeq \id_{P}$. 
  \end{enumerate}
\end{lemma}
 
Let $X$ and $Y$ be topological spaces. We say $X$ and $Y$ are \emph{homotopy equivalent} if there are continuous maps $f\colon X\to Y$ and $g\colon Y\to X$ whose compositions are homotopic to the identity maps, i.e., $gf\simeq \id_X$ and $fg\simeq \id_Y$. We write $X\simeq Y$ if $X$ and $Y$ are homotopy equivalent spaces. Homotopy equivalence is an equivalence relation. The equivalence class of a space is referred to be its \emph{homotopy type}.

\begin{proposition}[see {\cite[Proposition~3.1.12]{Smi11}}]
  Let $f\colon P\to P$ be a poset map such that $f\simeq \id_{P}$ (e.g. if $f\leq \id_{P}$ or if $\id_{P}\leq f$). Then $|\mathcal{K}(P)|\simeq |\mathcal{K}(f(P))|$.
\end{proposition}

\begin{remark}[Poset theoretical terminology]
  A \emph{Galois connection} between two posets $P$ and $Q$ is a pair of order-preserving maps $f\colon P\to Q$ and $g\colon Q\to P$ having the following properties:
$$ \text{(i)}\;\; x\leq gf(x)\; \text{ for all } x\in P\quad \text{and}\quad \text{(ii)}\;\; fg(y)\leq y\; \text{ for all } y\in Q. $$
Notice that a Galois connection between two posets implies the homotopy equivalence of the associated order complexes. A closure operator is a special case obtained by taking one of the maps to be the identity.
\end{remark}

A topological space is \emph{contractible} if it is homotopy equivalent to a single point. A simplicial complex $K$ is a \emph{cone} if it is a join of form $\{v\}*L$ for a subcomplex $L$ and a vertex $v$. A cone is contractible to any cone point (see \cite[Lemma~3.2.2]{Smi11}). For a poset $P$ and an element $x\in P$, the order complexes of the principal order ideal $P_{\leq x}$ as well as the principal dual order ideal $P_{\geq x}$ are cones and hence are contractible (see \cite[Lemma~3.2.3]{Smi11}). In particular, if a finite poset has a maximum or a minimum, then its order complex is a cone and hence contractible.

Let $P$ be a poset. We say $P$ is \emph{conically contractible} if there exist an element $x_0\in P$ and a poset endomorphism $f$ of $P$ satisfying $$ x\leq f(x) \geq x_0\;\;\text{ for all } x\in P. $$
If a poset is conically contractible, then it is contractible (see \cite[Lemma~3.3.3]{Smi11}).

The previous results provide several sufficient conditions for proving homotopy equivalences between order complexes. The following theorem of Quillen is considerably more general and will later be applied to poset maps comparing finite topological spaces.

\begin{theorem}[Quillen's fiber theorem, see {\cite[Proposition~1.6]{Qui78}} and {\cite[Theorem~4.2.1]{Smi11}}]\label{thm:quillen_fiber_theorem}
  Let $f\colon P\to Q$ be a poset map such that for all $x\in Q$ the fiber $f^{-1}(Q_{\leq x})$ is contractible. Then $f$ induces a homotopy equivalence of $|\mathcal{K}(P)|$ with $|\mathcal{K}(Q)|$. (A dual form holds using $Q_{\geq x}$.)
\end{theorem}

%--------------------------------------------------------------------
\subsection{Finite topological spaces}\label{subsec:finite-spaces-background}
%--------------------------------------------------------------------

Finite topological spaces admit a natural combinatorial description in terms of preorders and posets. We briefly recall the basic constructions and results that will be used throughout the paper; for a comprehensive modern treatment, see Barmak \cite{Bar11}, building on the foundational
work of Alexandroff, Stong and McCord (see, for example, \cite{Ale37,Sto66,McC66}).

%--------------------------------------------------------------------
\subsubsection*{Minimal basis and $T_0$-reduction}
%--------------------------------------------------------------------

Let $X$ be a finite topological space. For each $x\in X$, define
$$ U_x=\bigcap\{U\subseteq X\colon U\text{ is open and }x\in U\}. $$
Since $X$ is finite, each $U_x$ is open. Moreover, the collection
$$ \mathcal{U}=\{U_x\colon x\in X\} $$
forms a basis for the topology, called the \emph{minimal basis} of $X$. Define a relation on $X$ by
$$ x\leq y \quad\Longleftrightarrow\quad x\in U_y. $$
It is straightforward to verify that $\leq$ is a preorder. Conversely, every finite preordered set determines a unique topology by taking
$$ U_x=\{y\in X\colon y\le x\} $$
as the minimal open neighbourhood of $x$. Thus finite topological spaces and finite preordered sets are equivalent descriptions of the same objects. Define an equivalence relation by
$$ x\sim y \quad\Longleftrightarrow\quad x\leq y\text{ and }y\leq x. $$
Equivalently,
$$ x\sim y \quad\Longleftrightarrow\quad U_x=U_y. $$
The quotient $X_0=X/{\sim}$ inherits a natural topology. Since the induced preorder on $X_0$ is antisymmetric, it is a partial order. Consequently, finite $T_0$-spaces and finite posets may be regarded as equivalent descriptions of the same objects.

\begin{theorem}[{\cite[Proposition~1.3.1]{Bar11}}]\label{thm:t0_quotient}
Every finite topological space is homotopy equivalent to its quotient space $X/{\sim}$.
\end{theorem}

\begin{example}\label{ex:t0_quotient}
  Let $X=\{x_1,x_2,x_3,y,u,v_1,v_2,w\}$ be the finite space whose topology is generated by the collection
  $$ \mathcal{U}=\{\{x_1,x_2,x_3\}, \{y\}, \{x_1,x_2,x_3,y,u\}, \{x_1,x_2,x_3,y,v_1,v_2\}, \{x_1,x_2,x_3,y,u,w\}\}. $$
  Notice that the elements of $\mathcal{U}$ form the minimal basis of the topology of $X$. Accordingly, the $T_0$-classes are
  $$ \{x_1,x_2,x_3\},\; \{y\},\; \{u\},\; \{v_1,v_2\},\; \{w\}. $$
  Observe that the $T_0$-classes are equivalence classes of points and should not be confused with the elements of the minimal basis. Let
  $$ \bar{x}:=\{x_1,x_2,x_3\},\; \bar{y}:=\{y\},\; \bar{u}:=\{u\},\; \bar{v}:=\{v_1,v_2\},\; \bar{w}:=\{w\}. $$
  Then the minimal basis of the topology of $X/{\sim}$ is
  $$ \overline{\mathcal{U}}=\{\{\bar{x}\}, \{\bar{y}\}, \{\bar{x},\bar{y},\bar{u}\}, \{\bar{x},\bar{y},\bar{v}\}, \{\bar{x},\bar{y},\bar{u},\bar{w}\}\}. $$
  By \Cref{thm:t0_quotient} the spaces $X$ and $X/{\sim}$ are homotopy equivalent.
\end{example}

In studying the homotopy theory of finite spaces, one may restrict attention to finite $T_0$-spaces without loss of generality. In the FTS framework, however, the quotient map
$$ X\longrightarrow X/{\sim} $$
records which data points are indistinguishable under the current probe-generated topology. Thus, although $X$ and $X/{\sim}$ have the same homotopy type, the fibers of the quotient map carry additional combinatorial information that is retained throughout the filtration.

%--------------------------------------------------------------------
\subsubsection*{McCord's theorem and weak homotopy equivalence}
%--------------------------------------------------------------------

We next recall McCord's Theorem, which establishes the connection between finite $T_0$-spaces and simplicial complexes used for homology computations.

\begin{theorem}[McCord's theorem, see {\cite{McC66}} and {\cite[Theorem~1.4.6]{Bar11}}]\label{thm:weak}
If $X$ is a finite $T_0$-space with associated poset $P$, then $X$ is weakly homotopy equivalent to the order complex $\mathcal{K}(P)$.
\end{theorem}

\begin{example}\label{ex:weak}
  We revisit \Cref{ex:t0_quotient}. The Hasse diagram of the associated poset $P$ of $X/{\sim}$ and its order complex $\mathcal{K}(P)$ are depicted in \Cref{fig:t0_quotient}. Combining \Cref{thm:t0_quotient} and \Cref{thm:weak}, we conclude that $X$ is weakly homotopy equivalent to $\mathcal{K}(P)$.
\end{example}

\begin{figure}[!ht]
  \centering
  \includegraphics[width=0.6\textwidth]{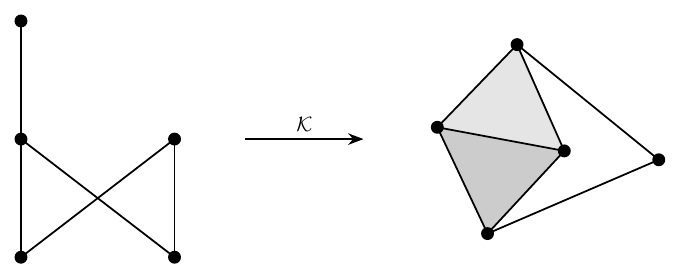}
  \caption{The Hasse diagram of $P$ in \Cref{ex:weak} and its order complex $\mathcal{K}(P)$.}
  \label{fig:t0_quotient}
\end{figure}

McCord's theorem provides a simplicial model for a finite $T_0$-space that preserves its weak homotopy type. Passing to the order complex therefore retains only weak homotopy information. In the FTS framework, however, this is not a limitation, since the finite topological spaces themselves remain the primary objects of study, while order complexes are used when simplicial-topological invariants are required. Consequently, constructions such as the $T_0$-merge tree rely on the finite topological spaces themselves rather than solely on their
order complexes.

%--------------------------------------------------------------------
\subsubsection*{Beat points and core}
%--------------------------------------------------------------------

Let $X$ be a finite $T_0$-space. A point $x\in X$ is called an \emph{up-beat point} if the set $\{y\in X\colon y>x\}$ has a minimum element. Dually, $x$ is called a \emph{down-beat point} if $\{y\in X\colon y<x\}$ has a maximum element. A \emph{beat point} is either an up-beat point or a down-beat point.

\begin{theorem}[Beat-point reduction, see \cite{Sto66} and {\cite[Proposition~1.3.4]{Bar11}}]\label{thm:beat}
If $x$ is a beat point of a finite $T_0$-space $X$, then $X\setminus\{x\}$ is a strong deformation retract of $X$.
\end{theorem}

A finite $T_0$-space is a \emph{minimal finite space} if it has no beat points. A \emph{core}
of a finite topological space is a strong deformation retract that is a minimal finite space.

\begin{theorem}[Stong's core theorem, see \cite{Sto66} and {\cite[Corollary~1.3.7]{Bar11}}]\label{thm:core}
Every finite topological space has a core, unique up to homeomorphism. Moreover, two finite topological spaces are homotopy equivalent if and only if their cores are homeomorphic.
\end{theorem}

\begin{example}\label{ex:core}
  We revisit \Cref{ex:weak}. Let $Y$ be the core of $X/{\sim}$, and let $C$ be its associated poset. The Hasse diagram of the poset $C$ and its order complex $\mathcal{K}(C)$ are shown in \Cref{fig:core}. By \Cref{thm:core} the finite space $X$ is weakly homotopy equivalent to $\mathcal{K}(C)$.
\end{example}

\begin{figure}[!ht]
  \centering
  \includegraphics[width=0.55\textwidth]{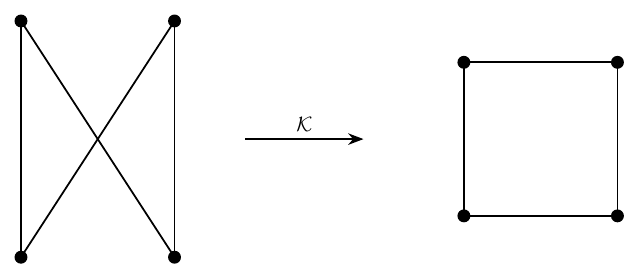}
  \caption{The Hasse diagram of $C$ in \Cref{ex:core} and its order complex $\mathcal{K}(C)$.}
  \label{fig:core}
\end{figure}

The constructions introduced in this subsection form the topological backbone of the FTS framework.  Starting from a finite topological space, we pass to its associated $T_0$-space (equivalently, its poset) and then to its order complex. Beat-point reduction supplies, in parallel, a strong-deformation-retract core of the finite space.  In the practical pipeline below the full $T_0$-poset is retained as the primary filtration object, while its core is used as an optional derived structure.

%====================================================================
\subsection{Filtrations, functoriality, and persistence modules}
\label{subsec:filtrations-background}
%====================================================================

Persistent homology studies the evolution of homological information along a parameterized family of spaces and compatible maps. The language of functors provides a natural framework for describing such constructions. We recall here the basic notions needed in this paper, following primarily \cite{Oud15}; see also \cite{CSGO16} for persistence modules and their stability theory, and \cite{EH10} for the classical computational-topology viewpoint.

Throughout this subsection, $(I,\leq)$ denotes a totally ordered set. We regard $I$ as a category whose objects are the elements of $I$, with a unique morphism $i\longrightarrow j$ whenever $i\leq j$. The filtrations computed later in the paper are indexed by finite totally ordered sets. In the stability discussion we will also use $I=\mathbb{R}$, where the parameter is a real-valued threshold.

%--------------------------------------------------------------------
\subsubsection*{Filtrations and functoriality}
%--------------------------------------------------------------------

Let $\mathbf{C}$ be a category. A \emph{filtration in $\mathbf{C}$ indexed by $I$} is a covariant functor
$$ \mathbb{X}\colon I\longrightarrow \mathbf{C}. $$
Thus, for every $i\in I$ there is an object $X_i$, and for every $i\leq j$ there is a morphism $f_{i,j}\colon X_i\longrightarrow X_j$ such that $f_{i,i}=\id_{X_i}$ and $f_{j,k}\circ f_{i,j}=f_{i,k}$ whenever $i\leq j\leq k$. Taking $\mathbf{C}=\mathbf{Top}$ gives a filtration of topological spaces. Similarly, a \emph{simplicial filtration} is a functor
$$ \mathbb{K}\colon I\longrightarrow \mathbf{SimpComp}. $$
It consists of simplicial complexes $K_i$ and simplicial maps $f_{i,j}\colon K_i\longrightarrow K_j$, $i\leq j$, satisfying the same compatibility conditions. Nested families
$$ K_0\subseteq K_1\subseteq\cdots\subseteq K_m $$
with the inclusion maps form the familiar special case, but the maps of a simplicial filtration need not be inclusions.

For every $n\geq 0$ and field $\Bbbk$, simplicial homology is functorial. That is to say, a simplicial map $f\colon K\longrightarrow L$ induces a linear map
$$ H_n(f)\colon H_n(K;\Bbbk)\longrightarrow H_n(L;\Bbbk), $$
and these induced maps are compatible with identities and composition. Consequently, applying homology to a simplicial filtration $\mathbb{K}\colon I\longrightarrow \mathbf{SimpComp}$ gives a functor
$$ H_n(\mathbb{K};\Bbbk)\colon I\longrightarrow \mathbf{Vect}_{\Bbbk}. $$
The analogous construction applies to filtrations of topological spaces using singular homology.

%--------------------------------------------------------------------
\subsubsection*{Persistence modules and barcodes}
%--------------------------------------------------------------------

A \emph{persistence module} indexed by $I$ over a field $\Bbbk$ is a covariant functor
$$ \mathbb{V}\colon I\longrightarrow \mathbf{Vect}_{\Bbbk}. $$
Equivalently, it consists of a family of vector spaces $\{V_i\}_{i\in I}$ and linear maps $v_{i,j}\colon V_i\longrightarrow V_j$, $i\leq j$, such that $v_{i,i}=\id_{V_i}$ and $v_{j,k}\circ v_{i,j}=v_{i,k}$ whenever $i\leq j\leq k$. Thus the homology of a filtration defines a persistence module.

A subset $J\subseteq I$ is an \emph{interval} if, whenever $i,k\in J$ and $i\leq j\leq k$, then $j\in J$. The \emph{interval module} $\Bbbk_J$ associated with $J$ is defined by
$$
(\Bbbk_J)_i=
\begin{cases}
\Bbbk & i\in J,\\
0 & i\notin J,
\end{cases}
$$
with the identity map between copies of $\Bbbk$ whenever both indices belong to $J$, and the zero map otherwise.

For the finite filtrations considered in the computational part of this paper, the following form of the interval decomposition theorem is sufficient.

\begin{theorem}[Interval decomposition, see {\cite{CSGO16}} and {\cite{Oud15}}]\label{thm:persistence-structure}
Let $\mathbb{V}$ be a pointwise finite-dimensional persistence module indexed by a finite totally ordered set $I$. Then there is an isomorphism
$$ \mathbb{V}\cong \bigoplus_{\alpha}\Bbbk_{J_\alpha}, $$
where the $J_\alpha$ are intervals in $I$. The multiset $\{J_\alpha\}_{\alpha}$ is uniquely determined by $\mathbb{V}$.
\end{theorem}

The multiset of intervals $\{J_\alpha\}_{\alpha}$ is called the \emph{barcode} of $\mathbb{V}$. If
$$ I=\{t_0<t_1<\cdots<t_m\}\subseteq \mathbb{R}, $$
then every interval in $I$ is a consecutive subset $\{t_p,t_{p+1},\ldots,t_q\}$, and is therefore determined by its endpoints.

Naturally isomorphic persistence modules have the same interval decomposition and therefore the same barcode. Notice that a collection of levelwise isomorphisms defines a natural isomorphism of persistence modules precisely when it is compatible with the linear maps defining the modules. In the case of persistent homology, these are the linear maps induced on homology by the maps of the filtration.

%--------------------------------------------------------------------
\subsubsection*{Interleavings and bottleneck distance}
%--------------------------------------------------------------------

For the stability questions considered later, we also use persistence modules indexed by $(\mathbb{R},\leq)$. For $\varepsilon\geq 0$, let
$$ T_\varepsilon\colon \mathbb{R}\longrightarrow\mathbb{R},\qquad t\longmapsto t+\varepsilon $$
denote translation by $\varepsilon$. If $\mathbb{V}\colon(\mathbb{R},\leq)\longrightarrow\mathbf{Vect}_{\Bbbk}$ is a persistence module with linear maps $v_{s,t}\colon V_s\longrightarrow V_t$ for $s\leq t$, then $\mathbb{V}\circ T_\varepsilon$ is the shifted persistence module $t\mapsto V_{t+\varepsilon}$. Moreover, the maps
$$ v_{t,t+\varepsilon}\colon V_t\longrightarrow V_{t+\varepsilon} $$
form a natural transformation
$$ \eta_{\mathbb{V}}^\varepsilon\colon\mathbb{V}\Longrightarrow\mathbb{V}\circ T_\varepsilon. $$

Let $\mathbb{V},\mathbb{W}\colon(\mathbb{R},\leq)\longrightarrow\mathbf{Vect}_{\Bbbk}$ be persistence modules, with linear maps $v_{s,t}\colon V_s\longrightarrow V_t$ and $w_{s,t}\colon W_s\longrightarrow W_t$. The modules $\mathbb{V}$ and $\mathbb{W}$ are \emph{$\varepsilon$-interleaved} if there exist natural transformations
$$ \varphi\colon\mathbb{V}\Longrightarrow\mathbb{W}\circ T_\varepsilon,\qquad \psi\colon\mathbb{W}\Longrightarrow\mathbb{V}\circ T_\varepsilon $$
such that, for every $t\in\mathbb{R}$,
$$ \psi_{t+\varepsilon}\circ\varphi_t=v_{t,t+2\varepsilon}\qquad\text{and}\qquad \varphi_{t+\varepsilon}\circ\psi_t=w_{t,t+2\varepsilon}. $$
Equivalently, these compositions agree with the canonical natural transformations $\eta_{\mathbb{V}}^{2\varepsilon}$ and $\eta_{\mathbb{W}}^{2\varepsilon}$. The \emph{interleaving distance} between $\mathbb{V}$ and $\mathbb{W}$ is
$$ d_I(\mathbb{V},\mathbb{W})=\inf\{\varepsilon\geq 0\colon \mathbb{V}\text{ and }\mathbb{W}\text{ are }\varepsilon\text{-interleaved}\}. $$

For an interval-decomposable persistence module $\mathbb V$ indexed by $\mathbb R$, each interval summand has birth and death endpoints, together with endpoint conventions specifying whether the endpoints are included. Suppressing these decorations gives a point $(b,d)$ in the extended open half-plane
$$
\mathbb H^\circ = \{(b,d)\in\overline{\mathbb R}^{\,2}\colon b<d\}.
$$
The resulting multiset is the \emph{persistence diagram}, denoted $\operatorname{dgm}(\mathbb V)$.

Let $P$ and $Q$ be persistence diagrams. A \emph{partial matching} $M\colon P\leftrightarrow Q$ pairs some copies of points of $P$ with some copies
of points of $Q$, with every point occurring in at most one pair. Write
$$
U_M = \bigl(P\setminus\pi_P(M)\bigr) \sqcup \bigl(Q\setminus\pi_Q(M)\bigr)
$$
for the multiset of unmatched points. The cost of $M$ is
$$
c(M) = \max\left\{\sup_{(p,q)\in M}\|p-q\|_\infty,\, \sup_{s=(s_x,s_y)\in U_M} \frac{|s_y-s_x|}{2} \right\}.
$$
The second term is the $\ell_\infty$-distance from an unmatched point to the diagonal. The \emph{bottleneck distance} is
$$
d_B(P,Q) = \inf_{M\colon P\leftrightarrow Q}c(M).
$$
In particular,
$$
d_B\bigl(\operatorname{dgm}(\mathbb V), \operatorname{dgm}(\mathbb W)\bigr)
$$
measures the smallest possible maximal cost of a partial matching between the two persistence diagrams. The usual extended-real conventions apply to points with an infinite coordinate.

The relation between the interleaving and bottleneck distances is given by the Isometry Theorem. We state the pointwise finite-dimensional case needed here.

\begin{theorem}[Isometry Theorem, pointwise finite-dimensional case; see \cite{CSGO16} and {\cite{Oud15}}]
\label{thm:persistence-isometry}
Let $\mathbb{V}$ and $\mathbb{W}$ be pointwise finite-dimensional persistence modules indexed by $\mathbb{R}$. Then
$$ d_I(\mathbb{V},\mathbb{W})=d_B\bigl(\operatorname{dgm}(\mathbb{V}),\operatorname{dgm}(\mathbb{W})\bigr). $$
\end{theorem}

Thus a bound on the interleaving distance between two persistence modules gives the same bound on the bottleneck distance between their persistence diagrams, and conversely. This relation will be used later when studying stability of the persistence modules obtained from finite topological space filtrations.

% ====================================================================
\section{Finite Topological Space Filtrations}
\label{sec:fts-filtrations}
% ====================================================================

This section develops the theoretical framework for finite topological space filtrations. We first give the general mathematical definition and introduce the evolution poset together with its level-annotated form. We then consider rule-generated filtrations as a natural reference model for hierarchies determined by admissible probes. In the basic construction an uncovered point has the whole space as its minimal open, so the associated order complex becomes a cone. Faithful uncovered policies address this premature collapse within the score-parameterized rule evolution.

The full rule system may nevertheless be too large to construct. Anchor-based filtrations replace it by a manageable selected family. The canonical comparison maps and the Quillen-fibre criterion give an exact standard for homotopy-preserving compression, while the maximum-fibre witness criterion gives a concrete sufficient condition for the relevant fibres. The final subsection revisits the level-annotated evolution poset and enriches it with the distance matrix. The resulting enriched evolution poset provides a convenient finite data structure from which persistent homology and its barcodes, cores, component and $T_0$-class merge trees, and incidence spines can be derived.

%====================================================================
\subsection{Abstract finite topological space filtrations}
\label{subsec:abstract-fts}
%====================================================================

The central object of this paper is a filtration whose members are finite topological spaces rather than simplicial complexes. We begin by introducing
this construction in its most general form, independently of any particular method used to generate the underlying topologies.

Let $X$ be a finite set and let $\mathcal T_1\supseteq \cdots \supseteq \mathcal T_r$ be a nested sequence of topologies on $X$. Since every open set of $\mathcal T_j$ is also open in $\mathcal T_i$ whenever $i\leq j$, the identity map on the underlying set
$$
\id_X\colon (X,\mathcal T_i)\longrightarrow(X,\mathcal T_j)
$$
is continuous for every $i\leq j$. Consequently, every coarsening sequence of finite topologies determines a filtration of finite topological spaces.

\begin{definition}[FTS filtration]\label{def:fts-filtration}
Let $I$ be a totally ordered set and let $X$ be a finite set. A
\emph{finite topological space filtration on $X$ indexed by $I$} is a covariant functor
$$
\mathbb X\colon I\longrightarrow\mathbf{FinTop},\qquad t\longmapsto (X,\mathcal T_t),
$$
determined by a family of topologies $\{\mathcal T_t\}_{t\in I}$ on $X$ satisfying $\mathcal T_s\supseteq\mathcal T_t$ whenever $s\leq t$. For $s\leq t$, the corresponding continuous map from $(X,\mathcal T_s)$ to $(X,\mathcal T_t)$ is the identity map on $X$.
\end{definition}

\begin{remark}
  Since $X$ is finite, only finitely many distinct topologies can occur in $\mathbb X$, even when the indexing set $I$ is infinite.
\end{remark}

The following example illustrates the construction. At this stage the sequence of topologies should be viewed simply as an abstract coarsening evolution; no particular principle governing the choice of the topologies is assumed.

\begin{example}\label{ex:abstract_filtration}
Let $I=\{1,2,3,4,5\}$ and $X=\{a,b,c,d\}$. Define an FTS filtration
$$
\mathbb X\colon I\longrightarrow\mathbf{FinTop},\qquad i\longmapsto (X,\mathcal T_i),
$$
by the following minimal open neighbourhoods:
\begin{align*}
U_1(a) &=\{a\},     & U_1(b) &=\{b\}, & U_1(c) &=\{c\},   & U_1(d) & =\{d\};\\
U_2(a) &=\{a,b\},   & U_2(b) &=\{b\}, & U_2(c) &=\{c\},   & U_2(d) &=\{c,d\};\\
U_3(a) &=\{a,b,c\}, & U_3(b) &=\{b\}, & U_3(c) &=\{c\},   & U_3(d) &=\{b,c,d\};\\
U_4(a) &=X,         & U_4(b) &=\{b\}, & U_4(c) &=\{c\},   & U_4(d) &=X;\\
U_5(a) &=X,         & U_5(b) &=X,     & U_5(c) &=X,       & U_5(d) &=X.
\end{align*}
For each $i$, the displayed sets satisfy
$$
y\in U_i(x)\quad\Longrightarrow\quad U_i(y)\subseteq U_i(x),
$$
and hence they are the minimal open neighbourhoods of a topology $\mathcal T_i$ on $X$. Moreover, for every $x\in X$ one has $U_1(x)\subseteq U_2(x)\subseteq\cdots\subseteq U_5(x)$, so the corresponding topologies satisfy
$$
\mathcal T_1\supseteq\mathcal T_2\supseteq\cdots\supseteq\mathcal T_5.
$$
Here $\mathcal T_1$ is discrete, whereas $\mathcal T_5$ is indiscrete.
\end{example}

The previous section shows that every finite topological space determines a finite preorder, whose $T_0$-quotient is naturally identified with a finite poset. Moreover, every poset determines an order complex, from which homology groups may be computed. Since these constructions are functorial, an FTS filtration gives rise to the composite
$$ I\longrightarrow \mathbf{FinTop}\longrightarrow \mathbf{FinTop}_0\longrightarrow \mathbf{Poset}\overset{\mathcal{K}}{\longrightarrow} \mathbf{SimpComp}\overset{H_n(-;\Bbbk)}{\longrightarrow} \mathbf{Vect}_{\Bbbk}. $$
Thus every finite topological space filtration canonically induces a persistence module. The importance of this observation is that persistent homology can be computed without requiring the filtration itself to consist of simplicial complexes. Instead, simplicial complexes arise naturally only after passing to the associated $T_0$-spaces and applying the order-complex construction.

For the posets $P_i$ associated with \Cref{ex:abstract_filtration}, the vertices of $\mathcal K(P_i)$ are the corresponding $T_0$-classes. At the first three levels all points remain $T_0$-distinct. At level $4$, the points $a$ and $d$ form one $T_0$-class, while at level $5$ all four points belong to the same class. The first complex consists of four isolated vertices, and the second is the disjoint union of two edges. The third is a four-cycle and is therefore homotopy equivalent to $S^1$. The fourth is a tree, and the fifth consists of a single vertex. Consequently,
$$
\begin{array}{c|ccccc}
i&1&2&3&4&5\\
\hline
\beta_0(\mathcal K(P_i))&4&2&1&1&1\\
\beta_1(\mathcal K(P_i))&0&0&1&0&0.
\end{array}
$$
All homology groups in degrees $n\geq2$ vanish. Thus the induced persistence module records both the successive merging of connected components and a degree-one class that appears at level $3$ and disappears at level $4$.

\Cref{ex:abstract_filtration} demonstrates that the essential object is the evolving family of finite topological spaces. The associated posets, order complexes, and homology groups are derived objects obtained by functorial constructions.

Let $\mathbb X\colon I\to\mathbf{FinTop}$ be a finite topological space filtration on the fixed set $X$. For each level $i\in I$ and each point $x\in X$, let $U_i(x)$ denote the minimal open neighbourhood of $x$ at level $i$, i.e., the intersection of all sets in $\mathcal T_i$ containing $x$, and set
$$
\mathcal M_i = \{U_i(x)\colon x\in X\}.
$$
The set $\mathcal M_i$, ordered by inclusion, is naturally isomorphic to the $T_0$-poset $P_i$ of $(X,\mathcal T_i)$: the class $[x]_i$ is represented by its minimal open $U_i(x)$.  We use this realization inside the Boolean lattice $\mathcal P(X)$ because it allows all levels to be compared in one ambient poset.

\begin{definition}[Evolution poset]\label{def:evolution-poset}
Let $\mathbb X\colon I\to\mathbf{FinTop}$ be a finite topological space filtration on the fixed set $X$. Its \emph{evolution poset} is
$$
\mathcal M(\mathbb X) := \bigcup_{i\in I}\mathcal M_i \subseteq \mathcal P(X),
$$
ordered by inclusion.
\end{definition}

\begin{figure}[!ht]
  \centering
  \includegraphics[width=0.47\textwidth]{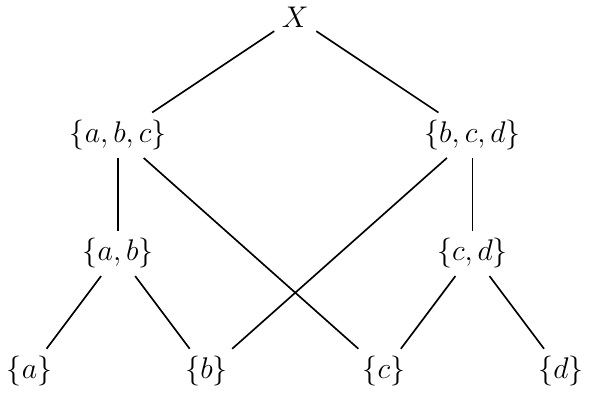}
  \caption{The Hasse diagram of the evolution poset of the filtration described in \Cref{ex:abstract_filtration}.}
  \label{fig:evolution_poset}
\end{figure}

The evolution poset records which subsets of $X$ occur as minimal open neighbourhoods and how they are related by inclusion. It does not, however, record the levels at which a given minimal open occurs. We therefore associate to each $U\in\mathcal M(\mathbb X)$ its \emph{level support}
$$
\mathsf L(U) = \{i\in I\colon U\in\mathcal M_i\}.
$$
The pair $(\mathcal M(\mathbb X),\mathsf L)$ will be called the \emph{level-annotated evolution poset}. The level annotation recovers every level slice:
$$
\mathcal M_i = \{U\in\mathcal M(\mathbb X)\colon i\in\mathsf L(U)\}.
$$
Consequently, the level-annotated evolution poset recovers the finite topology and its $T_0$-poset at every level. More precisely, for each $U\in\mathcal M_i$, the $T_0$-class represented by $U$ consists exactly of those points $x\in X$ for which $U$ is the least member of $\mathcal M_i$ containing $x$. Thus it determines the precise partition of $X$ into $T_0$-classes at each level, without requiring any additional annotation.

% The level-annotated evolution poset records the finite-topological filtration itself. After introducing rule-generated and anchor-based filtrations, \Cref{subsec:evolution-poset-derived-structures} supplements this object with the distance matrix. The resulting enriched evolution poset provides a common finite data structure from which persistent homology and its barcodes, cores, component and $T_0$-class merge trees, clustering constructions, and incidence spines can be derived. \Cref{sec:practical-constructions} then introduces the concrete metric-guided anchor-based filtrations used in practice.

%====================================================================
\subsection{Rule-generated finite topological space filtrations}
\label{subsec:rule-based}
%====================================================================

The previous subsection introduced finite topological space filtrations as abstract objects. We now describe a general mechanism for constructing such filtrations from families of subsets. The basic rule determines which subsets remain active at a given parameter value; an uncovered-point policy specifies the topology when the active family ceases to cover the whole data set.

%--------------------------------------------------------------------
\subsubsection*{Rule-generated topologies}
%--------------------------------------------------------------------

Let $X$ be a finite set and let $I$ be a totally ordered parameter set. Suppose that, for every parameter $t\in I$, a rule specifies a collection
$$
\mathcal R_t\subseteq\mathcal P(X)
$$
and that these collections are nested, i.e., $\mathcal R_s\supseteq\mathcal R_t$ whenever $s\leq t$. If the active rule sets alone are used as generators, they determine the topology $\langle\mathcal R_t\rangle$. Deleting members of a generating family can only coarsen the generated topology, so
these topologies already form a finite topological space filtration. Because $X$ is finite, the union
$$
\mathcal R=\bigcup_{t\in I}\mathcal R_t\subseteq \mathcal P(X)
$$
is finite. Consequently, even when $I=\mathbb{R}$, only finitely many strict changes of the active rule family can occur. The mathematical filtration may therefore be viewed as an $I$-indexed, piecewise constant filtration, while an implementation may propagate it only through its finitely many critical rule states.

A particularly useful parameterization is obtained from a score
$$
\rho\colon \mathcal R\longrightarrow\mathbb{R}.
$$
Set
$$
\mathcal R_t^{\rho} = \{R\in\mathcal R\colon \rho(R)>t\}.
$$
The strict inequality makes a rule disappear precisely when its score threshold is crossed. Ties therefore form simultaneous deletion blocks.

One natural example arises from density thresholding. Suppose each admissible subset $U\subseteq X$ is assigned a numerical score $\rho(U)$ measuring, for example, how densely its points are distributed. At threshold $t$, define
$$
\mathcal R_t = \{U\colon \rho(U)>t\}.
$$
When $t$ is small, many subsets satisfy the rule and the generated topology is fine. As the threshold increases, fewer subsets remain active, producing a coarsening sequence of finite topologies.

%--------------------------------------------------------------------
\subsubsection*{The uncovered-point phenomenon}
%--------------------------------------------------------------------

The active rules need not cover $X$. Put
$$
C_t = \bigcup_{R\in\mathcal R_t}R.
$$
For a covered point $x\in C_t$, the minimal open neighbourhood in the topology generated by the active rules is
$$
M_t(x) = \bigcap_{\substack{R\in\mathcal R_t\\x\in R}}R.
$$
If $x\notin C_t$, however, no active generator contains $x$. In the topology $\langle\mathcal R_t\rangle$ its minimal open neighbourhood is therefore the whole space $X$. Thus an uncovered point becomes a maximum of the associated poset after $T_0$-reduction, and the corresponding order complex is a cone. The rule-generated topology can consequently become weakly contractible as soon as coverage fails, even though the rule evolution may still contain substantial structure. This premature collapse is the reason that uncovered points are treated as part of the mathematical construction rather than as an implementation detail.

%--------------------------------------------------------------------
\subsubsection*{Faithful uncovered policies}
%--------------------------------------------------------------------

We do not prescribe how minimal open neighbourhoods are assigned to uncovered points at the rule level. Instead, we isolate the conditions that any such assignment must satisfy.

\begin{definition}[Faithful uncovered policy]\label{def:faithful-uncovered}
Let $\{\mathcal R_t\}_{t\in I}$ be a nested rule family on $X$, let $C_t$ be its covered set, and let $M_t(x)$ be the active-rule intersection defined above for $x\in C_t$. A \emph{faithful uncovered policy} assigns a topology $\mathcal T_t$ to each level $t$ such that, writing $U_t(x)$ for the minimal open neighbourhood of $x$ in $\mathcal T_t$,
\begin{enumerate}
\item for every covered point $x\in C_t$, the assigned topology does not alter the minimal opens forced by the current active rules, i.e., $U_t(x)=M_t(x)$;
\item the resulting topologies coarsen with the parameter, i.e., $\mathcal T_s\supseteq\mathcal T_t$ whenever $s\leq t$, or equivalently $U_s(x)\subseteq U_t(x)$ for every $x\in X$.
\end{enumerate}
The policy is free to choose the minimal opens of points in $X\setminus C_t$, subject to these requirements and the finite-topology axioms.
\end{definition}

The first condition is the sense in which the policy is faithful: active-rule information is represented exactly wherever the rule still provides coverage, while freedom of choice is confined to the uncovered region. In particular, if $C_t=X$, then the policy is forced to give $\mathcal T_t=\langle\mathcal R_t\rangle$.

A nested rule family together with a faithful uncovered policy therefore determines a rule-generated FTS filtration
$$
\mathbb X^{\mathcal R}\colon I\longrightarrow\mathbf{FinTop}.
$$
This filtration will serve as a reference for the anchor-based filtrations introduced in the next subsection. It should not, however, be regarded as an ``ideal'' practical filtration that is automatically more informative. Its importance is twofold. First, the admissible rule family gives a clear semantic meaning to the filtration: it states directly which local subsets are regarded as legitimate probes and how their lifetimes are thresholded. Second, this explicit rule-level description is especially convenient for stability arguments, because perturbations can often be compared through refinement or score shifts. These advantages do not imply that the resulting evolution is the most useful one for data analysis. A very rich or poorly chosen rule family can be computationally enormous, and a perfectly stable rule-generated filtration can still be topologically uninformative.

Thus the practical problem is not merely to approximate a theoretically superior filtration. It is to construct a manageable probe family whose meaning can be related to a rule-generated reference while retaining useful finite-topological evolution. The next subsection introduces anchor-based filtrations for this purpose. They replace a possibly very large rule family by a selected subfamily while retaining a canonical comparison whenever the uncovered policies are compatible.

%====================================================================
\subsection{Anchor-based finite topological space filtrations}
\label{subsec:anchor-based}
%====================================================================

The rule-generated construction of the previous subsection provides a conceptually natural finite-sample model, but its family of admissible probes may be too large to enumerate or manipulate directly. This motivates the use of a manageable selected family of subsets at each parameter value. Such a family may be constructed directly from the data and need not arise from a previously specified rule family.

We call the selected subsets \emph{anchor sets}, or \emph{A-sets}. The term ``anchor'' is used in a broad sense: an anchor set need not be a metric ball, a $k$-nearest-neighbourhood, or even a neighbourhood centred at a distinguished data point. It is simply a selected subset used to generate part of the finite topology.

\begin{definition}
Let $X$ be a fixed finite set and $I$ a totally ordered parameter set. An \emph{anchor system on $X$ indexed by $I$} is a nested family
$$
\{\mathcal A_t(X)\}_{t\in I},\qquad \mathcal A_t(X)\subseteq\mathcal P(X),
$$
such that each $\mathcal A_t(X)$ is computationally manageable and $\mathcal A_s(X)\supseteq\mathcal A_t(X)$ whenever $s\leq t$.
\end{definition}

Since the anchor families are nested, \Cref{def:faithful-uncovered} applies with $\mathcal A_t(X)$ in place of $\mathcal R_t(X)$. Thus an anchor system together with a faithful uncovered policy determines an FTS filtration
$$
\mathbb X^{\mathcal A}\colon I\longrightarrow\mathbf{FinTop}.
$$
The topology at level $t$ need not be generated only by the members of $\mathcal A_t(X)$. If anchor coverage fails, the policy assigns minimal opens to the uncovered points while leaving the anchor-defined minimal opens of covered points unchanged. The particular uncovered policy used in the practical construction is specified in \Cref{sec:practical-constructions}.

An anchor-based filtration does not require a rule-generated counterpart. When a rule-generated filtration is available, however, it can provide a useful reference for studying the information retained by the anchor construction. For this comparison, suppose that
$$
\mathbb X^{\mathcal R}\colon I\longrightarrow\mathbf{FinTop}
$$
is a rule-generated filtration equipped with a faithful uncovered policy and impose the additional condition
$$
\mathcal A_t(X)\subseteq\mathcal R_t(X),\qquad t\in I.
$$
This inclusion is a hypothesis for the rule-to-anchor comparison and is not part of the definition of an anchor system. Write $U_t^{\mathcal R}(x)$ and $U_t^{\mathcal A}(x)$ for the minimal open neighbourhoods of $x$ in the rule-generated and anchor-based topologies, respectively.

\begin{lemma}[Covered-point comparison]
\label{lem:anchor-covered-comparison}
Let $\mathbb X^{\mathcal R}$ and $\mathbb X^{\mathcal A}$ be rule-generated and anchor-based filtrations equipped with faithful uncovered policies, and suppose that
$$
\mathcal A_t(X)\subseteq\mathcal R_t(X).
$$
If $x$ is covered by the anchor family at level $t$, then it is also covered by the rule family and
$$
U_t^{\mathcal R}(x)\subseteq U_t^{\mathcal A}(x).
$$
\end{lemma}

\begin{proof}
Since $x$ belongs to some member of $\mathcal A_t(X)$ and $\mathcal A_t(X)\subseteq\mathcal R_t(X)$, it is also covered by the rule family. Faithfulness gives
$$
U_t^{\mathcal R}(x) = \bigcap_{\substack{R\in\mathcal R_t(X)\\x\in R}}R, \qquad U_t^{\mathcal A}(x) = \bigcap_{\substack{A\in\mathcal A_t(X)\\x\in A}}A.
$$
The first intersection is taken over a family containing the second, and therefore
$$
U_t^{\mathcal R}(x)\subseteq U_t^{\mathcal A}(x).
$$
\end{proof}

Thus the required rule-to-anchor containment holds automatically at points covered by the anchor family. Only points that are uncovered by the anchor family require an additional comparison condition.

\begin{definition}[Rule-compatible uncovered policy]
Let $\mathbb X^{\mathcal R}$ be a rule-generated filtration and let $\{\mathcal A_t(X)\}_{t\in I}$ be an anchor system satisfying
$$
\mathcal A_t(X)\subseteq\mathcal R_t(X)\qquad\text{for every }t\in I.
$$
A faithful uncovered policy for the anchor system is \emph{rule-compatible} if
$$
U_t^{\mathcal R}(x)\subseteq U_t^{\mathcal A}(x)\qquad \text{for every }x\in X\text{ and }t\in I.
$$
\end{definition}

Under this condition, the identity map
$$
\id_X\colon (X,\mathcal T_t^{\mathcal R})\longrightarrow (X,\mathcal T_t^{\mathcal A})
$$
is continuous at every level $t$, so the anchor topology is at least as coarse as the corresponding rule-generated topology. \Cref{lem:anchor-covered-comparison} shows that rule compatibility need only be checked at points uncovered by the anchor family.

%--------------------------------------------------------------------
\subsubsection*{The A-stage and anchor-specific B-completion}
%--------------------------------------------------------------------

The anchor system above determines the score- or schedule-parameterized part of the filtration produced directly from the active anchor families and the prescribed uncovered policy. We call this part the \emph{A-stage}. For a finite exhausting schedule, the A-stage includes the boundary level at which no A-set remains active. The rule-generated and anchor-generated filtrations are compared only over their common A-stage index set.

A finite anchor dictionary also retains an ordered record of the A-sets removed during the A-stage. Unlike the full rule family, this record is computationally available and can be used to append an anchor-specific sequence of levels after A-set exhaustion. We call this appended part the \emph{B-stage}, or \emph{B-completion}.  Removed A-sets are not reactivated as generators; their supports are used only as continuation data. Thus a completed anchor filtration has two conceptually distinct parts:
$$
\boxed{
\text{A-stage}\quad \longrightarrow\quad \text{B-completion}.
}
$$
The A-stage records the evolution produced by removing active anchors, whereas the B-stage uses the recorded anchor history to resolve classes surviving A-set exhaustion. The initial memory depth, the order in which the remaining history is admitted, and the resulting terminal topology are defined in \Cref{subsec:b-sets}. The B-stage is not part of the rule-generated reference or of the rule-to-anchor comparison developed below.

%--------------------------------------------------------------------
\subsubsection*{The canonical comparison maps}
%--------------------------------------------------------------------

Let $q_t^{\mathcal R}\colon X\to P_t^{\mathcal R}$ and $q_t^{\mathcal A}\colon X\to P_t^{\mathcal A}$ denote the $T_0$-quotient maps associated with $(X,\mathcal T_t^{\mathcal R})$ and $(X,\mathcal T_t^{\mathcal A})$, respectively. Here $P_t^{\mathcal R}$ and $P_t^{\mathcal A}$ are regarded as finite posets with the orders induced by their finite topologies.

Since $\mathcal T_t^{\mathcal A}\subseteq\mathcal T_t^{\mathcal R}$, the full rule-generated topology distinguishes at least as many points as the
anchor-generated topology. In particular,
$$
x\sim_t^{\mathcal R}y \quad\Longrightarrow\quad x\sim_t^{\mathcal A}y.
$$
It follows that there is a unique map
$$
\pi_t\colon P_t^{\mathcal R}\longrightarrow P_t^{\mathcal A}
$$
such that $q_t^{\mathcal A} = \pi_t\circ q_t^{\mathcal R}$. Explicitly, $\pi_t\bigl([x]_{\mathcal R,t}\bigr) = [x]_{\mathcal A,t}$. Since $\pi_t$ is induced by a continuous identity map between finite spaces, it is order-preserving. Thus it induces a simplicial map
$$
\mathcal K(\pi_t)\colon \mathcal K(P_t^{\mathcal R})\longrightarrow \mathcal K(P_t^{\mathcal A}).
$$
The maps $\pi_t$ are also compatible with the filtration structure. For $s\leq t$, let
$$
p_{s,t}^{\mathcal R}\colon P_s^{\mathcal R}\longrightarrow P_t^{\mathcal R}\qquad\text{and}\qquad p_{s,t}^{\mathcal A}\colon P_s^{\mathcal A}\longrightarrow P_t^{\mathcal A}
$$
be the maps induced by the identity on $X$. Then the following square commutes:
$$
\begin{tikzcd}
P_s^{\mathcal R} \arrow[r,"p_{s,t}^{\mathcal R}"] \arrow[d,"\pi_s"']
&
P_t^{\mathcal R} \arrow[d,"\pi_t"] \\
P_s^{\mathcal A} \arrow[r,"p_{s,t}^{\mathcal A}"']
&
P_t^{\mathcal A}.
\end{tikzcd}
$$
Indeed, the two routes around the square send the class of $x$ at level $s$ in the rule-generated filtration to the class of the same point $x$ at level $t$ in the anchor-generated filtration. Consequently, the family $\pi=\{\pi_t\}_{t\in I}$ defines a natural transformation
$$
\pi\colon \mathbb P^{\mathcal R}\Longrightarrow\mathbb P^{\mathcal A}.
$$

Naturality alone does not imply that the rule- and anchor-generated filtrations have equivalent persistence modules; it only provides a compatible comparison between them. In the next subsection, the Quillen-fibre condition ensures that each map $\mathcal K(\pi_t)$ is a homotopy equivalence. Together with the commutative squares above, these levelwise equivalences induce a natural isomorphism between the corresponding persistence modules.

%--------------------------------------------------------------------
\subsubsection*{Quillen-faithful anchor systems}
%--------------------------------------------------------------------

The anchor-generated filtration is intended to be a computationally feasible replacement for a much larger rule-generated filtration. Merely requiring $\mathcal A_t(X)\subseteq\mathcal R_t(X)$ does not ensure that this replacement is topologically faithful. A very small anchor family may generate a topology that is much too coarse and may destroy homotopy information already present in the full system. Conversely, stability of the full rule system is useful only when the admissible rule family and its score produce an informative reference filtration.  The pair consisting of the admissible family and the rule score is therefore part of the modelling choice; stability alone does not make a rule system suitable for compression.

We first isolate the exact condition under which the rule-to-anchor compression is homotopically faithful.  The canonical comparison maps allow \Cref{thm:quillen_fiber_theorem} to be applied at every A-stage level. Suppose the uncovered policies are compatible, so that the continuous identity map from the rule topology to the anchor topology induces a canonical poset map $\pi_t\colon P_t^{\mathcal R}\longrightarrow P_t^{\mathcal A}$. Recall that, with our order convention, the minimal open neighbourhood of an element $a=[x]_{\mathcal A,t}\in P_t^{\mathcal A}$ is the principal order ideal $(P_t^{\mathcal A})_{\leq a}$.

\begin{definition}[Quillen-faithful anchor system]
\label{def:quillen-faithful-anchor}
Let $\{\mathcal A_t(X)\}_{t\in I}$ be an anchor system over a common A-stage index set $I$.  It is \emph{Quillen-faithful} with respect to a rule-generated system if
\begin{enumerate}
\item the associated uncovered policy is compatible, so that the canonical poset maps $\pi_t\colon P_t^{\mathcal R}\longrightarrow P_t^{\mathcal A}$ are defined; and
\item for every $t\in I$ and every $a\in P_t^{\mathcal A}$, the lower fibre $\pi_t^{-1}((P_t^{\mathcal A})_{\leq a})$ is contractible.
\end{enumerate}
\end{definition}

By the dual form of \Cref{thm:quillen_fiber_theorem}, one may instead require the inverse images of the principal upper ideals to be contractible. We use the lower-ideal formulation because minimal open neighbourhoods are principal order ideals under our order convention. The fibre condition in \Cref{def:quillen-faithful-anchor} is precisely the levelwise hypothesis of \Cref{thm:quillen_fiber_theorem}.  Together with the naturality established above, it has the following persistence consequence.

\begin{theorem}[Persistence preservation by Quillen-faithful anchors]
\label{thm:quillen-faithful-persistence}
Let
$$
\mathbb P^{\mathcal R},\mathbb P^{\mathcal A}\colon I\longrightarrow\mathbf{Poset}
$$
be the rule-generated and anchor-generated $T_0$-poset filtrations over a common A-stage index set $I$, and let
$$
\pi\colon\mathbb P^{\mathcal R}\Longrightarrow\mathbb P^{\mathcal A}
$$
be the canonical natural transformation.  If the anchor system is Quillen-faithful with respect to the rule-generated system, then, for every
$t\in I$, the map
$$
\mathcal K(\pi_t)\colon \mathcal K(P_t^{\mathcal R}) \longrightarrow \mathcal K(P_t^{\mathcal A})
$$
induces a homotopy equivalence on geometric realizations.  Moreover, for every
$k\geq0$, the maps
$$
H_k(\mathcal K(\pi_t))\colon H_k(\mathcal K(P_t^{\mathcal R});\Bbbk) \longrightarrow H_k(\mathcal K(P_t^{\mathcal A});\Bbbk)
$$
form a natural isomorphism between the persistence modules induced by the two A-stage filtrations.  Consequently, the rule-generated and anchor-generated A-stages have identical barcodes in every homological degree.
\end{theorem}

\begin{proof}
By Quillen faithfulness and \Cref{thm:quillen_fiber_theorem}, each map $\mathcal K(\pi_t)$ induces a homotopy equivalence on geometric realizations and hence an isomorphism on homology. Naturality of $\pi$ implies, after applying the order-complex and homology functors, that these levelwise isomorphisms commute with the persistence maps. They therefore form a natural isomorphism of persistence modules, and equality of barcodes follows.
\end{proof}

The theorem gives an exact meaning to a \emph{Quillen-faithful} anchor realization of a rule-generated A-stage. The two finite spaces need not be homeomorphic and their $T_0$-posets may have different numbers of elements; what is preserved is the homotopy type of their order complexes across the A-stage. The maximum-fibre witness criterion below gives a sufficient condition for an individual lower fibre to be contractible; it does not replace verification of Quillen faithfulness over the full A-stage.

%--------------------------------------------------------------------
\subsubsection*{The maximum-fibre witness criterion}
%--------------------------------------------------------------------

The simplest way for a lower fibre to be contractible is for it to have a maximum. This observation leads to the following concrete criterion.

\begin{lemma}[Maximum-fibre witness criterion]
\label{lem:maximum-fibre-witness}
Assume that the rule and anchor systems are compatible.  Fix $t\in I$ and $a=[x]_{\mathcal A,t}\in P_t^{\mathcal A}$, and write $U=U_t^{\mathcal A}(x)$. Then the lower fibre
$$
F_t(a)=\pi_t^{-1}\bigl((P_t^{\mathcal A})_{\leq a}\bigr)
$$
has a maximum if and only if there exists $z\in U$ such that
$$
U_t^{\mathcal R}(z)=U.
$$
In that case $[z]_{\mathcal R,t}$ is the unique maximum of $F_t(a)$.
\end{lemma}

\begin{proof}
A rule class $[y]_{\mathcal R,t}$ belongs to $F_t(a)$ exactly when $[y]_{\mathcal A,t}\leq[x]_{\mathcal A,t}$, equivalently when $y\in U$. Hence
$$
F_t(a)=q_t^{\mathcal R}(U).
$$
If $U_t^{\mathcal R}(z)=U$, then the principal lower ideal generated by $[z]_{\mathcal R,t}$ is precisely $q_t^{\mathcal R}(U)$, so $[z]_{\mathcal R,t}$ is the maximum of the fibre. Conversely, if the fibre has a maximum represented by $z\in U$, then its principal lower ideal equals $q_t^{\mathcal R}(U)$.  Since both $U_t^{\mathcal R}(z)$ and $U$ are unions of rule-equivalence classes, this implies $U_t^{\mathcal R}(z)=U$.
\end{proof}

Here the relevant maximum is a maximum of the comparison fibre; no maximality condition on the selected generators is involved.

\subsubsection*{Interpretation and limitations}

For $a\in P_t^{\mathcal A}$, the lower fibre $\pi_t^{-1}\bigl((P_t^{\mathcal A})_{\leq a}\bigr)$ consists of the rule-generated $T_0$-classes whose anchor images lie below $a$. Its contractibility means that the identifications made by the anchor system introduce no homotopical obstruction over this principal ideal. The different roles of the constructions above may be summarized as follows.

\begin{center}
\fbox{%
\begin{minipage}{0.92\linewidth}
\textit{Rule-generated A-stage.}
A semantic and stability reference that may be computationally large and need not be informative.

\smallskip
\textit{Quillen-faithful anchor system.}
An exact condition ensuring that anchor compression preserves the persistence modules induced over the A-stage.

\smallskip
\textit{Maximum-fibre witness criterion.}
A concrete sufficient condition for an individual lower fibre to be contractible.

\smallskip
\textit{Practical topology mode.}
A computationally feasible anchor construction motivated by the comparison theory but not automatically Quillen-faithful.

\smallskip
\textit{B-completion.}
An anchor-specific post-A-stage construction that lies outside the rule-to-anchor Quillen comparison.
\end{minipage}%
}
\end{center}

The maximum-fibre witness criterion is an exact local diagnostic for an individual comparison fibre, not an automatic certificate of Quillen faithfulness for an implemented topology mode. Verifying the full condition may itself require access to the large rule-generated system. The concrete choices of anchor dictionary, removal score, uncovered policy, and B-completion are introduced in \Cref{sec:practical-constructions}, and their perturbation behaviour is studied in \Cref{sec:stability}.

%====================================================================
\subsection{The enriched evolution poset and its derived structures}
\label{subsec:evolution-poset-derived-structures}
%====================================================================

Recall that the level-annotated evolution poset consists of the evolution poset $\mathcal M(\mathbb X)$ together with the level support $\mathsf L$ of its nodes. The level annotation recovers every slice
$$
\mathcal M_i = \{U\in\mathcal M(\mathbb X)\colon i\in\mathsf L(U)\},
$$
and hence all minimal open neighbourhoods, $T_0$-classes, and associated posets at level $i$. Since all maps in the filtration are induced by $\id_X$, these structures can be compared and the original data points can be followed directly across levels. The distance matrix $D$ relates this finite-topological evolution to the metric geometry of the data.

\begin{definition}[Enriched evolution poset]
\label{def:enriched-evolution-poset}
Let $(X,D)$ be a finite metric data set and let
$$
\mathbb X\colon I\longrightarrow\mathbf{FinTop}
$$
be an FTS filtration on $X$. The \emph{enriched evolution poset} of $\mathbb X$ is
$$
\mathcal E(\mathbb X) = \bigl(\mathcal M(\mathbb X),\mathsf L,D\bigr),
$$
where $(\mathcal M(\mathbb X),\mathsf L)$ is the level-annotated evolution poset of $\mathbb X$.
\end{definition}

Several structures used in this paper are derived from the enriched evolution poset.

\begin{enumerate}
\item \emph{Persistent homology.} The level slices determine the $T_0$-posets and their order complexes, while the identity-induced maps produce the corresponding simplicial filtration. Applying homology gives persistence modules and their barcodes. Selected cycle representatives are retained when a later analysis requires localization.

\item \emph{Finite-topological and homotopical structure.} The finite topologies retain information that is deliberately suppressed by persistent homology. Minimal opens, $T_0$-classes, associated posets, and beat-point reductions can therefore be studied directly throughout the filtration. Cores, core minimal-open blocks, and reduction records retain distinctions that need not be visible in a barcode.

\item \emph{Hierarchical and clustering structure.} Following connected components and $T_0$-classes across filtration levels produces component merge trees and $T_0$-class merge trees. The clustering constructions of \Cref{sec:clustering} use minimal opens and $T_0$-classes directly and also study the components obtained by removing selected points and rerunning the filtration.

\item \emph{Incidence structure.} Inclusions among minimal opens and the component-merger structure identify the local incidence relations, while the distance matrix supplies their scores and point-level realization. Together these determine the incidence spine. The resulting spine is subsequently used for two distinct purposes: incidence-based $H_1$-association and geometric obstruction analysis.
\end{enumerate}

These are the principal derived structures studied in the present paper, but they do not exhaust the possible uses of an FTS filtration. The enriched evolution poset retains the finite-topological evolution together with its relation to the metric data, allowing further task-specific constructions to be developed without changing the underlying filtration. \Cref{sec:practical-constructions} now introduces the concrete metric-guided filtrations used in the remainder of the paper.

% ====================================================================
\section{Practical Constructions}
\label{sec:practical-constructions}
% ====================================================================

The constructions of the previous section describe finite topological space filtrations independently of the particular rules used to generate their topologies. We now specify the concrete metric-guided realization used by the current implementation. Its scientific input is a finite metric data set, represented entirely by its distance matrix. For each data point and each admissible value of $k$, the implementation forms the tie-inclusive $k$-nearest-neighbour set. Depending on the chosen configuration, this set is retained unchanged or replaced by its coherent core. The resulting candidates are then used to select a finite A-set dictionary. Each selected A-set receives a lifetime score, and the dictionary is removed according to a common schedule. Whenever the surviving A-sets fail to cover some data points, the B-set mechanism supplies continuation neighbourhoods implementing the faithful uncovered policy of the previous section.

The practical filtration is now explicitly \emph{completed}. Its first part is the A-removal stage
$$
\mathcal T_1\supseteq\cdots\supseteq\mathcal T_m,
$$
whose last level has no active A-set. A second B-completion stage then enlarges the available A-set history without reactivating any A-set as a generator:
$$
\mathcal T_m \supseteq \mathcal T_{m+1}^{\mathcal{B}} \supseteq\cdots\supseteq \mathcal T_L^{\mathcal{B}}.
$$
At the completed endpoint, the minimal open neighbourhood of a point $x$ is the set of points connected to $x$ through a chain of overlapping A-sets. Thus the terminal $T_0$-quotient has one isolated point for each connected component of the full A-set overlap system. Its order complex is discrete, so all positive-dimensional homology groups vanish. A homology class that is still present when all A-sets have been removed is therefore not automatically essential in the completed filtration. It is followed through the B-completion stage and may disappear as progressively earlier A-set information is admitted.

There are exactly two active topology modes,
$$
\boxed{\texttt{profile}\qquad\text{and}\qquad\texttt{adaptive}.}
$$
They use the same candidate preparation, coherent-core extraction, selection scores, and removal scores. They differ only in how multiple candidates are selected for a centre. Profile specifies the $k$-intervals in advance and selects a highest-scoring candidate from each interval. Adaptive begins with a single $k$-interval, divides the ordered score values into retained regimes, and selects a highest-scoring candidate from each regime. When exactly one A-set is requested per centre over the same interval, the two constructions are identical.

Once the active A-sets and the minimal opens assigned to uncovered points have been determined at every displayed level, the resulting finite topologies form the practical FTS filtration. Passing to the $T_0$-quotients gives a poset filtration
$$
P_1\longrightarrow\cdots\longrightarrow P_m
\longrightarrow
P_{m+1}^{\mathcal B}\longrightarrow\cdots\longrightarrow P_L^{\mathcal B}.
$$
As explained in \Cref{subsec:filtrations-background} and \Cref{subsec:abstract-fts}, the identity maps on $X$ induce the maps between these posets and hence simplicial maps between their order complexes. Applying homology produces the corresponding persistence modules and barcodes. The construction of the practical filtration may be summarized as
$$
\boxed{
\begin{gathered}
D
\longrightarrow
\text{canonicalized }D
\longrightarrow
\text{tie-inclusive }k\text{NN sets}
\\[1mm]
\overset{\substack{\text{optional coherent-core extraction}\\
\text{profile/adaptive selection}}}{\Downarrow}
\\[-1mm]
\text{selected A-set dictionary}
\longrightarrow
\text{lifetime scores}
\longrightarrow
\text{ordered A-set dictionary}
\\[1mm]
\overset{\substack{\text{removal schedule}\\
\text{faithful uncovered policy}}}{\Downarrow}
\\[-1mm]
\text{A-removal stage}
\longrightarrow
\text{B-completion}
\longrightarrow
\text{completed FTS filtration }\mathbb X .
\end{gathered}
}
$$
The subsequent subsections describe this construction and the practical realization of the derived structures introduced in \Cref{subsec:evolution-poset-derived-structures}: persistent homology, deterministic core reduction, component and $T_0$-class merge trees, and incidence spines. Further algorithmic details and implementation alternatives are collected in the accompanying guide.

% ====================================================================
\subsection{Metric input and canonicalization}
\label{subsec:data-metric}
% ====================================================================

The synthetic and real data sets summarized in \Cref{tab:data-set-catalogue} provide the running examples for the practical constructions developed in this section and for the theoretical and empirical analyses of the later sections.
\begin{table}[!ht]
\centering
\small
\caption{Data sets used in the experiments. The reported sample sizes are those used in the principal analyses.}
\label{tab:data-set-catalogue}
\setlength{\tabcolsep}{4pt}
\renewcommand{\arraystretch}{1.15}
\begin{tabular}{@{}
  >{\raggedright\arraybackslash}p{0.16\textwidth}
  >{\raggedright\arraybackslash}p{0.18\textwidth}
  >{\raggedright\arraybackslash}p{0.24\textwidth}
  >{\raggedright\arraybackslash}p{0.28\textwidth}
  @{}}
\toprule
Data set
& Sample size
& Metric input
& Role in the experiments
\\
\midrule

Letters
& 500 points for each letter
& Euclidean distances between normalized planar coordinates
& Controlled planar examples used in the stability experiments.
\\
\addlinespace

Mixed bumps
& 775 points
& Euclidean distances in $\mathbb R^2$
& Synthetic clustering example with a heterogeneous background and six planted regions.
\\
\addlinespace

Mixed holes
& 800 points
& Euclidean distances in $\mathbb R^2$
& Synthetic example used for $H_1$-association, and geometric obstruction analysis.
\\
\addlinespace

Sphere
& 642 points for each variant
& Euclidean distances in $\mathbb R^3$
& Icosahedral and uniformly sampled unit-sphere variants used in the stability experiments.
\\
\addlinespace

Ellipsoid
& 642 points for each variant
& Euclidean distances in $\mathbb R^3$
& Icosahedral and uniformly sampled ellipsoid variants, with semiaxes $(1.4,1,0.7)$, used in the stability experiments.
\\
\addlinespace

Paul15
& 2730 cells; three subsamples of 1000 cells
& Euclidean distances using the first 30 principal-component coordinates
& Real-data case study of finite-topological clustering and removal-and-rerun analysis.
\\
\addlinespace

COIL20
& 20 objectwise data sets, with 72 images in each
& Euclidean distances between flattened $32\times32$ grayscale images
& Real-data case study of $H_1$-association, and geometric obstruction analysis.
\\
\bottomrule
\end{tabular}
\end{table}

The input to the FTS filtration is a finite distance matrix.  Let $X=\{x_1,\ldots,x_n\}$ and let
$$
D=\bigl(d(x_i,x_j)\bigr)_{1\leq i,j\leq n}
$$
be its distance matrix. No coordinate representation of $X$ is required. When coordinates are available, they are retained only for visualization. Before constructing the filtration, the implementation deterministically reorders the data points using a preliminary finite topology obtained from the complete selected A-set dictionary. The points belonging to the same preliminary $T_0$-class are placed consecutively, and nearby classes in the Hasse diagram are placed near one another. The resulting ordering $x_{\sigma(1)},\ldots,x_{\sigma(n)}$ gives the canonicalized distance matrix
$$
D^{\mathrm{can}}_{ij}=D_{\sigma(i),\sigma(j)}.
$$
Both $\sigma$ and $\sigma^{-1}$ are retained so that every computed object can be expressed in canonicalized or original labels. Simultaneously permuting the rows and columns changes only the presentation of the finite metric space: it preserves all distances and does not make the matrix sparse. The procedure is deterministic once the indexed distance matrix, topology mode, parameters, and tie-breaking conventions have been fixed; because original labels resolve residual ties, no invariance under arbitrary relabelling is claimed. The complete ordering procedure is given in the accompanying guide.

\begin{figure}[!ht]
\centering
\includegraphics[width=\textwidth]
{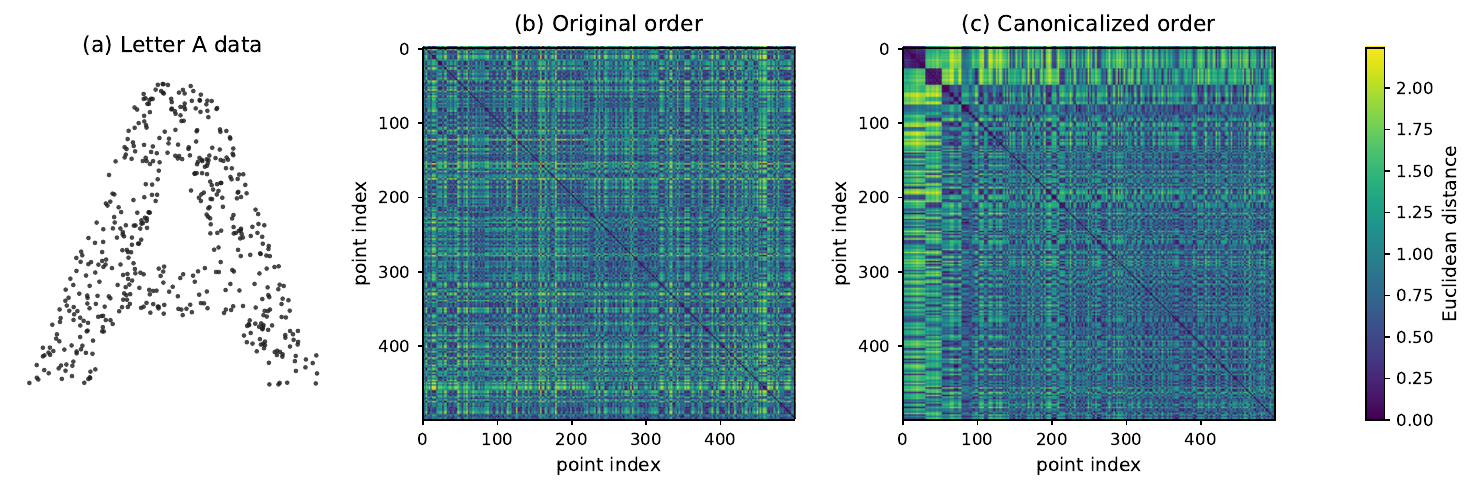}
\caption{Canonicalization for the letter A data set. Panel~(a) shows the sampled data. Panels~(b) and~(c) show the distance matrix before and after canonicalization, respectively, using the same colour scale. The rows and columns are simultaneously permuted, so all pairwise distances are preserved. The reordered matrix makes the block structure associated with the preliminary $T_0$-poset more visible.}
\label{fig:canonicalization}
\end{figure}

% ====================================================================
\subsection{Candidate A-sets, selection, and removal scores}
\label{subsec:a-set-dictionary}
% ====================================================================

The practical A-set dictionary is constructed before the filtration begins. Both topology modes use the same sequence of operations: tie-inclusive $k$NN sets $\longrightarrow$ optional coherent transformation $\longrightarrow$ selection $\longrightarrow$ removal scoring.

For a centre $x\in X$, let $r_k(x)$ be the distance to its $k$th nearest neighbour, not counting $x$, and set
$$
N_k(x)=\{x\}\cup\{y\in X\setminus\{x\}\colon d(x,y)\leq r_k(x)\}.
$$
All points at the boundary distance are included, and exact duplicate candidates associated with the same centre and selection range are identified. When coherent A-sets are requested, the candidate is replaced by a centre-preserving coherent core obtained from gaps in the average incoming-neighbour-rank spectrum. The generating centre is never removed and the procedure never produces a singleton. The complete iterative transformation and its tie conventions are given in the accompanying guide.

For a prepared candidate $A$, write $s(A)=|A|$ and define
$$
\delta(A;x,k)=\begin{cases}r_k(x),&A=N_k(x),\\[1mm]\operatorname{diam}(A),&A\text{ is a coherent A-set}.
\end{cases}
$$
The available size--distance scores are
$$
\frac{s}{\delta},\qquad\frac1\delta,\qquad\frac1s,\qquad\frac{\delta}{s}.
$$
A selection score decides which prepared candidates enter the dictionary, whereas a removal score determines their relative lifetimes after the dictionary has been fixed. The two scores need not agree. Exact ties and duplicate source assignments are handled deterministically, as specified in the guide.

In \texttt{profile} mode, the neighbour-count intervals are prescribed in advance and the requested number of highest-scoring candidates is retained from each interval. The standard two-set profile selects one candidate from each of $[10,20]$ and $[20,30]$. In \texttt{adaptive} mode, a single interval is prescribed and the ordered centrewise score spectrum is separated through a sparse middle region; one highest-scoring candidate is retained from each resulting regime. The experiments request at most two candidates per centre and use sparse-middle fraction $0.30$. One-set adaptive and one-slot profile agree whenever their interval, candidate preparation, and selection score agree.

After selection, exact copies of the same subset arising from different centres or regimes are combined while their source assignments are retained. Let
$$
\mathcal D_A=\{A_1,\ldots,A_N\}
$$
be the exact-deduplicated dictionary and let $\rho(A_j)$ be the removal score. Under forward removal, A-sets are deleted in increasing lexicographic score order, so higher-priority sets survive longer. Reverse removal is used only as an experimental control. These generator-level scores provide practical ways to construct and order the dictionary; they do not by themselves certify the Quillen-fibre condition of \Cref{sec:fts-filtrations}.

% ====================================================================
\subsection{The quantile schedule}
\label{subsec:schedules}
% ====================================================================

Write
$$
A_{(1)},A_{(2)},\ldots,A_{(N)}
$$
for the exact-deduplicated dictionary in removal order. The quantile schedule selects $m$ displayed levels by rounding
$$
c_i\approx\frac{i-1}{m-1}N,\qquad i=1,\ldots,m,
$$
to admissible integers with $c_1=0$ and $c_m=N$. Thus the first displayed level uses the complete dictionary, the last has no active A-set, and $c_i/N$ is the fraction invalidated. This schedule distributes approximately equal numbers of invalidation events among the displayed transitions and is used in every experiment reported in the paper. Other available schedule modes and their exact conventions are described in the accompanying guide.

% ====================================================================
\subsection{Uncovered points, B-memory, and the completion stage}
\label{subsec:b-sets}
% ====================================================================

The active A-sets need not cover the data throughout the removal stage. This subsection defines how uncovered points inherit finite-topological neighbourhoods from earlier A-set information and explains how the same history is subsequently used to complete the filtration without reactivating removed A-sets. The construction is entirely anchor-specific: $\eta_0$ is a parameter of the practical B-set policy and not of the rule-generated reference.

Let
$$
0=c_1\leq c_2\leq\cdots\leq c_m=N
$$
be the major schedule cuts, where
$$
A_{(1)},A_{(2)},\ldots,A_{(N)}
$$
is the A-set dictionary in removal order. Thus the first displayed level uses the complete dictionary, whereas the last has no active A-set. At level $i$, write
$$
\mathcal A_i = \{A_{(j)}\colon c_i<j\leq N\}
$$
for the active A-set family and
$$
C_i = \bigcup_{A\in\mathcal A_i}A
$$
for the currently covered set. The complete dictionary is required to cover $X$, so $C_1=X$. For a covered point $p\in C_i$, set
$$
M_i(p) = \bigcap_{\substack{A\in\mathcal A_i\\p\in A}}A.
$$
Thus $M_i(p)$ is the minimal open forced by the active A-sets containing $p$. When B-sets are assigned to uncovered points, these generator-determined minimal opens are kept unchanged.

Coverage can fail as A-sets are invalidated. The B-set mechanism assigns minimal opens to uncovered points by combining their preceding minimal opens with a finite window of recently invalidated A-sets. The user-facing integer
$$
\eta_0\in\{0,1,\ldots,m\}
$$
is called the \emph{initial B-memory depth}; in the implementation it is \texttt{b\_memory\_eta}. The subscript emphasizes that this is the memory depth used during A-removal. After A-exhaustion, the effective memory depth will be increased systematically until the full A-set history has been admitted.

For a memory depth $r\geq0$, define the remembered A-set family at level $i$ by
$$
\mathcal A_i^{(r)} = \left\{A_{(j)}\colon c_{\max\{1,i-r\}}<j\leq c_i\right\}.
$$
This family consists of invalidated A-sets and should not be confused with the active family $\mathcal A_i$. In particular,
$$
\mathcal A_i^{(0)}=\varnothing.
$$
The family $\mathcal A_i^{(1)}$ consists of the A-sets invalidated during the transition from level $i-1$ to level $i$, while larger values of $r$ include earlier removal blocks as well. During the A-removal stage, the effective depth is fixed at $r=\eta_0$.

Since the complete dictionary covers $X$, no B-sets are needed at the first level. Between two consecutive major levels, the topology may be updated either only at the next major cut or successively after the intervening one-A-set deletions. We call every state at which this update is performed an \emph{active propagation state}. For a major level $i>1$, let $i^{-}$ denote the active propagation state immediately preceding it. Thus $i^{-}=i-1$ when there are no intermediate propagation states; otherwise, $i^{-}$ is the last intermediate state before level $i$.

At an intermediate propagation state between levels $i-1$ and $i$, the active family is the current suffix of the ordered A-set dictionary, and the remembered family is the portion of $\mathcal A_i^{(\eta_0)}$ already invalidated at that state. The construction below is applied after each such deletion. In particular, the memory depth remains measured in major removal blocks, even when the minimal opens are propagated through the intermediate states.

Suppose that the topology at $i^{-}$ has already been constructed, and let $U_{i^{-}}(x)$ denote the minimal open of $x$ at that state. For any subset $Q\subseteq X$, define its open hull in the preceding topology by
$$
\operatorname{Hull}_{i^{-}}(Q) = \bigcup_{z\in Q}U_{i^{-}}(z).
$$
For an uncovered point $y\notin C_i$, collect the remembered A-sets containing $y$ by setting
$$
Q_i^{(\eta_0)}(y) = \bigcup_{\substack{A\in\mathcal A_i^{(\eta_0)}\\y\in A}}A.
$$
The preliminary B-set seed is
$$
S_i^{(\eta_0)}(y) = U_{i^{-}}(y) \cup \operatorname{Hull}_{i^{-}}\!\left(Q_i^{(\eta_0)}(y)\right).
$$
If no remembered A-set contains $y$, the second term is empty and the seed is simply $U_{i^{-}}(y)$. The seeds must be completed to a compatible family of minimal opens. Set
$$
L_i^{(0)}(x) = 
\begin{cases}
M_i(x),&x\in C_i,\\
S_i^{(\eta_0)}(x),&x\notin C_i.
\end{cases}
$$
The sets belonging to covered points are held fixed. For every uncovered point $y$, repeatedly enlarge
$$
L_i^{(n+1)}(y) = L_i^{(n)}(y)\cup \bigcup_{z\in L_i^{(n)}(y)}L_i^{(n)}(z),
$$
while retaining
$$
L_i^{(n+1)}(p) = M_i(p)\qquad(p\in C_i).
$$
Since $X$ is finite, this iteration stabilizes. Denote its stable value at an uncovered point $y$ by
$$
B_i^{(\eta_0)}(y).
$$
We then define
$$
U_i(x) = 
\begin{cases}
M_i(x),&x\in C_i,\\
B_i^{(\eta_0)}(x),&x\notin C_i.
\end{cases}
$$
The same construction, with the current active and remembered families, is used at every intermediate propagation state.

\begin{proposition}[Faithfulness of the practical B-set policy]
\label{prop:practical-b-faithful}
For every permitted initial memory depth $0\leq\eta_0\leq m$, the construction above terminates at every active propagation state, and the resulting pointwise sets form the minimal basis of a finite topology. At every major level $i$,
$$
U_i(p)=M_i(p)\qquad(p\in C_i).
$$
Moreover, if $s<t$ are consecutive active propagation states, then
$$
U_s(x)\subseteq U_t(x)\qquad(x\in X).
$$
Consequently, chaining the active propagation states between two consecutive major levels gives
$$
U_{i-1}(x)\subseteq U_i(x)\qquad(x\in X),
$$
so the A-removal stage is a coarsening filtration and the practical B-set construction is a faithful uncovered policy.
\end{proposition}

\begin{proof}
Fix an active propagation state and suppress its index. We verify the two minimal-open conditions
$$
x\in U(x)
\qquad\text{and}\qquad
z\in U(x)\Longrightarrow U(z)\subseteq U(x).
$$
If $p$ is covered, then $U(p)=M(p)$ and $p\in M(p)$. If $y$ is uncovered, its seed contains its preceding minimal open and hence contains $y$. Therefore $y\in B^{(\eta_0)}(y)=U(y)$. If $p$ is covered and $z\in U(p)=M(p)$, every active A-set containing $p$ also contains $z$. Thus $z$ is covered and $U(z)=M(z)\subseteq M(p)=U(p)$. For uncovered points, the fixed-point iteration is increasing and terminates because $X$ is finite. At its stable value it satisfies
$$
z\in B^{(\eta_0)}(y)
\quad\Longrightarrow\quad
U(z)\subseteq B^{(\eta_0)}(y).
$$
Hence the resulting pointwise sets form the minimal basis of a finite topology. 

Now let $s<t$ be consecutive active propagation states. If $x$ is covered at $t$, then it was covered at $s$, and the active A-sets containing $x$ at $t$ form a subfamily of those containing it at $s$. Therefore
$$
U_s(x)=M_s(x)\subseteq M_t(x)=U_t(x).
$$
If $x$ is uncovered at $t$, its seed contains $U_s(x)$, and the fixed-point iteration only enlarges that seed. Hence again $U_s(x)\subseteq U_t(x)$. This applies to every consecutive pair of active propagation states. Chaining the resulting containments from the major level $i-1$ to the major level $i$ gives
$$
U_{i-1}(x)\subseteq U_i(x)
$$
for every $x\in X$. Together with $U_i(p)=M_i(p)$ for covered points, this proves faithfulness.
\end{proof}

% --------------------------------------------------------------------
\subsubsection*{B-completion after A-exhaustion}
% --------------------------------------------------------------------

At the terminal A-removal level $m$, no A-set remains active and hence every point is uncovered. This level is a distinguished boundary of the filtration, not its endpoint. Its topology already incorporates the last $\eta_0$ major removal blocks through the B-memory construction above. B-completion successively admits the remaining, earlier blocks of the same A-set history.

For an effective memory depth
$$
r\in\{\eta_0,\eta_0+1,\ldots,m\},
$$
the family
$$
\mathcal A_m^{(r)} = \left\{A_{(j)}\colon c_{\max\{1,m-r\}}<j\leq N\right\}
$$
is the portion of the removal history available at the A/B boundary. Thus $\mathcal A_m^{(\eta_0)}$ has already been incorporated into the topology at level $m$. When the effective depth increases from $r-1$ to $r$, the newly admitted A-sets form the family
$$
\mathcal N_r = \mathcal A_m^{(r)}\setminus \mathcal A_m^{(r-1)}.
$$
This family may be empty when schedule cuts are repeated or when the memory window has already reached the beginning of the removal history. Write
$$
U^{[\eta_0]}(x) = U_m(x)
$$
for the minimal opens at the A/B boundary. Suppose that the topology at effective depth $r-1$ has already been constructed. For $Q\subseteq X$, define its open hull by
$$
\operatorname{Hull}^{[r-1]}(Q) = \bigcup_{z\in Q}U^{[r-1]}(z).
$$
For each $x\in X$, collect the newly admitted A-sets containing $x$:
$$
Q^{[r]}(x) = \bigcup_{\substack{A\in\mathcal N_r\\x\in A}}A.
$$
The corresponding B-set seed is
$$
S^{[r]}(x) = U^{[r-1]}(x)\cup \operatorname{Hull}^{[r-1]}\!\left(Q^{[r]}(x)\right).
$$
If no newly admitted A-set contains $x$, this seed is simply $U^{[r-1]}(x)$. Starting from
$$
L^{[r,0]}(x)=S^{[r]}(x),
$$
repeatedly set
$$
L^{[r,n+1]}(x) = L^{[r,n]}(x)\cup \bigcup_{z\in L^{[r,n]}(x)}L^{[r,n]}(z).
$$
The iteration stabilizes because $X$ is finite. Let $U^{[r]}(x)$ denote its stable value. As in the proof of \Cref{prop:practical-b-faithful}, the family
$$
\mathcal B^{[r]} = \{U^{[r]}(x)\colon x\in X\}
$$
is the minimal basis of a finite topology. Moreover,
$$
U^{[r-1]}(x)\subseteq U^{[r]}(x)
$$
for every $x\in X$, so the B-completion continues the coarsening filtration. No A-set becomes active again during this process. Its support is used only as historical information for enlarging the B-set seeds.

The same update has a concise preorder formulation. Let
$$
z\preceq^{[r]}x\quad\Longleftrightarrow\quad z\in U^{[r]}(x).
$$
Admitting an A-set $A$ makes all of its points mutually related before transitive closure. Consequently,
$$
\boxed{
\preceq^{[r]} = \operatorname{TC}\!\left(\preceq^{[r-1]}\cup \bigcup_{A\in\mathcal N_r}A\times A\right),\qquad r=\eta_0+1,\ldots,m.
}
$$
Indeed, the relations contributed by $A\times A$, together with the preceding preorder, insert precisely the preceding open hulls appearing in the B-set seeds. The implementation evaluates this equivalent update on the $T_0$-quotient at the A/B boundary, thereby avoiding repeated expansion of the full relation on $X$.

The number of displayed levels in the completed filtration is
$$
L = m+(m-\eta_0) = 2m-\eta_0.
$$
For the standard choice $m=30$, the horizons are therefore $60$, $59$, and $58$ for $\eta_0=0,1,2$, respectively. The common A-removal boundary remains level $m=30$ in all three cases. Consecutive B-completion levels may coincide if a newly admitted family is empty or if its addition does not change the topology.

Let $G_{\mathcal A}$ be the undirected point graph induced by the complete A-set dictionary: two distinct points $x,y\in X$ are adjacent whenever some A-set contains both of them. Equivalently, each A-set spans a clique in $G_{\mathcal A}$. Every active-A minimal open, inherited minimal open, remembered A-set hull, and fixed-point closure remains within a connected component of this graph. Thus the practical A/B construction never creates a relation between distinct components of $G_{\mathcal A}$.

\begin{proposition}[Completed endpoint of the practical filtration]
\label{prop:practical-b-completion-endpoint}
After the full A-set history has been admitted, the minimal open of every point is the connected component of $G_{\mathcal A}$ containing that point. Consequently, the open sets of the terminal topology are precisely the unions of the connected components
$$
X=\bigsqcup_{\lambda}X_\lambda.
$$
Equivalently, the terminal topology is indiscrete on each block $X_\lambda$, and its $T_0$-quotient is the discrete set of these components. In particular,
$$
\beta_k=0\qquad(k\geq1)
$$
and
$$
\beta_0 = \#\{\text{connected components of }G_{\mathcal A}\}.
$$
The terminal topology is indiscrete exactly when $G_{\mathcal A}$ is connected.
\end{proposition}

\begin{proof}
Every A-set admitted during B-completion contributes the complete bidirected relation $A\times A$. A chain of overlapping A-sets therefore generates mutual reachability among all points in the corresponding connected component of $G_{\mathcal A}$. Once the full history has been admitted, all points in the same component have the same minimal open and are topologically indistinguishable.

Conversely, the component-respecting property prevents relations between different components. The components of $G_{\mathcal A}$ are therefore exactly the terminal $T_0$-classes and the minimal opens of their points. The description of the terminal topology and its homology follows.
\end{proof}

B-completion changes the interpretation of late persistence. A degree-one class crossing the A/B boundary at level $m$ survives complete A-set removal, but it is not declared essential solely for that reason. Its subsequent lifetime records how much additional A-set history must be admitted before the class disappears. The completed endpoint therefore eliminates the artificial infinite $H_1$ bars that would arise if the filtration ended at A-exhaustion. Essential degree-zero classes may remain, but they then record genuinely distinct connected components of the complete A-set system.

The role of $\eta_0$ is correspondingly sharper in the completed filtration. It does not determine whether the full history eventually becomes available---B-completion guarantees this for every $\eta_0$. Instead, it determines how much history is already available during A-removal. It therefore affects both the topology at the A/B boundary and the path by which the common saturated endpoint is reached. Increasing $\eta_0$ may stabilize some continuation events, but it may also produce earlier coarsening during A-removal. This resolution--memory tradeoff is studied empirically in \Cref{sec:stability}.

The implementation records the initial depth $\eta_0$, the effective memory depth at every B-completion level, the newly admitted A-set family, and the A/B boundary. The main tables and plots display the two stages separately.

% ====================================================================
\subsection{The topology modes}
\label{subsec:topology-modes}
% ====================================================================

The implementation provides two topology modes, \texttt{profile} and \texttt{adaptive}, which determine only how the prepared candidates are selected centrewise to form the A-set dictionary. Candidate preparation, score families, A-set removal, scheduling, and B-memory are shared by both modes. A topology mode is therefore one component of a complete configuration, not a specification of the entire filtration.

% --------------------------------------------------------------------
\subsubsection*{The \texttt{profile} topology mode}
\label{subsubsec:profile-mode}
% --------------------------------------------------------------------

Profile prescribes one or more $k$-intervals, called \emph{slots}. For a slot $s$, the user specifies $\bigl(n_s,k_{\min}^{(s)},k_{\max}^{(s)}\bigr)$, where $n_s$ is the requested number of A-sets per centre from the interval $[k_{\min}^{(s)},k_{\max}^{(s)}]$. The corresponding implementation fields are
\begin{Verbatim}[xleftmargin=2em]
profile_count_s
profile_k_min_s
profile_k_max_s
profile_selection_score
\end{Verbatim}

For each centre and slot, the tie-inclusive $k$NN sets are prepared as in \Cref{subsec:a-set-dictionary}. Exact duplicate candidates are removed, the remaining candidates are ordered by the chosen selection score, and the first $n_s$ candidates are retained. The standard two-set profile uses one A-set from each of $[10,20]$ and $[20,30]$. A one-set profile over $[10,30]$ instead selects one highest-scoring candidate from the complete interval. Profile therefore prescribes the candidate regimes directly in $k$-space. Once the A-sets have been selected, all subsequent operations use the common removal score, schedule, uncovered-point policy, and B-completion.

% --------------------------------------------------------------------
\subsubsection*{The \texttt{adaptive} topology mode}
\label{subsubsec:adaptive-mode}
% --------------------------------------------------------------------

Adaptive begins with one $k$-interval and uses the centrewise selection-score spectrum to determine the retained regimes. Its mode-specific fields are
\begin{Verbatim}[xleftmargin=2em]
adaptive_k_min
adaptive_k_max
adaptive_count
adaptive_sparse_middle_fraction
adaptive_selection_score
\end{Verbatim}

The paper uses the interval $[10,30]$ and requests at most two A-sets per centre. When two are requested, the ordered score spectrum is separated through a sparse middle region and one highest-scoring candidate is selected from each retained regime. The parameter \nolinkurl{adaptive\_sparse\_middle\_fraction} controls the maximum fraction assigned to this separating region; the experiments use
$0.30$. When only one A-set is requested, no separation is needed and adaptive selects the highest-scoring candidate on the complete interval. Consequently,
$$
\boxed{
\text{one-set adaptive} = \text{one-slot profile}
}
$$
whenever the $k$-interval, candidate preparation, and selection score agree. The distinction between the two modes is therefore:
$$
\boxed{
\begin{aligned}
\texttt{profile}:&
\quad\text{the candidate regimes are prescribed in $k$-space},\\
\texttt{adaptive}:&
\quad\text{the candidate regimes are determined from the local score spectrum}.
\end{aligned}}
$$

% --------------------------------------------------------------------
\subsubsection*{Configurations used in the paper}
% --------------------------------------------------------------------

The later sections use controlled families of configurations rather than selecting one topology mode in advance. The principal choices are summarized in \Cref{tab:paper-topology-configurations}.

\begin{table}[!ht]
\centering
\small
\caption{Principal practical FTS configurations used in the later sections.}
\label{tab:paper-topology-configurations}
\begin{tabular}{p{0.15\textwidth}p{0.75\textwidth}}
\toprule
Analysis & Configurations \\
\midrule
Stability &
One-set profile, two-set profile, and two-set adaptive selection; raw and coherent A-sets; the four aligned selection and removal scores $s/\delta$, $1/\delta$, $1/s$, and $\delta/s$; and $\eta_0\in\{0,1,2,m\}$. The runs use $m=30$, the quantile schedule, forward removal, and major-level propagation. \\[1mm]

Clustering &
Two-set profile and adaptive selection, raw and coherent A-sets, and the aligned scores $s/\delta$ and $1/\delta$, with $\eta_0=0$. The \texttt{mixed\_bumps} and Paul15 analyses use the standard intervals $[10,20]$, $[20,30]$, and $[10,30]$, together with $m=30$, the quantile schedule, and major-level propagation. \\[1mm]

$H_1$-association and obstruction analysis &
For \texttt{mixed\_holes}, the same eight two-set combinations of profile/adaptive, raw/coherent, and the score choices $s/\delta$ and $1/\delta$ are used, with each score applied to both selection and removal. All configurations have $m=30$ and $\eta_0=0$. The obstruction analysis reuses the corresponding filtrations and incidence spines rather than selecting a new topology mode. \\[1mm]

COIL20 &
The same profile/adaptive and raw/coherent combinations are used with the score choices $s/\delta$ and $1/\delta$, again applying each score to both selection and removal. All configurations have $m=20$ and $\eta_0=0$. Adaptive uses $[6,22]$, while profile uses the two slots $[6,14]$ and $[14,22]$. \\
\bottomrule
\end{tabular}
\end{table}

The stability study is deliberately broader than the later structural applications. It includes the one-set construction, the $1/s$ and $\delta/s$ scores, and several values of $\eta_0$ in order to separate the effects of candidate selection, candidate representation, removal ordering, and continuation memory. The clustering, $H_1$-association, and obstruction analyses concentrate on the two-set modes and compare $s/\delta$ with $1/\delta$. These are the configurations analyzed in the paper, not restrictions imposed by the implementation. Further configuration fields and alternatives are documented in the accompanying guide.

% ====================================================================
\subsection{Effective levels and functorial outputs}
\label{subsec:poset-and-evolution}
% ====================================================================

At every displayed A-removal or B-completion level, the pointwise minimal opens $\{U_i(x)\colon x\in X\}$ determine the finite topology. Identifying points with the same minimal open gives the $T_0$-poset $P_i$, naturally isomorphic to the distinct minimal opens ordered by inclusion. The identity on $X$ induces the maps between successive posets and therefore the functor
$$
\mathbb P\colon I_{\mathrm{eff}}\longrightarrow\mathbf{Poset},
$$
where $I_{\mathrm{eff}}$ is the finite ordered set of displayed A-removal and B-completion levels. Collecting the distinct minimal opens from these slices, with their level supports and the distance matrix, gives the enriched evolution poset $\mathcal E(\mathbb X)$ introduced in \Cref{subsec:evolution-poset-derived-structures}; it is a representation of the completed filtration rather than an additional filtration stage.

Applying the order-complex functor gives
$$
\mathbb K=\mathcal K\circ\mathbb P\colon I_{\mathrm{eff}}\longrightarrow\mathbf{SimpComp}.
$$
For $i\leq j$, the map $p_{i,j}$ sends $[x]_i$ to $[x]_j$. The induced simplicial maps need not be inclusions because several vertices can be identified at a later level. Applying homology over $\mathbb F_2$ produces the degree-zero and degree-one persistence modules and their barcodes. The sequence of Betti numbers alone does not determine these barcodes; the induced maps are needed to follow births and deaths.

\begin{figure}[!ht]
  \centering \includegraphics[width=\textwidth]{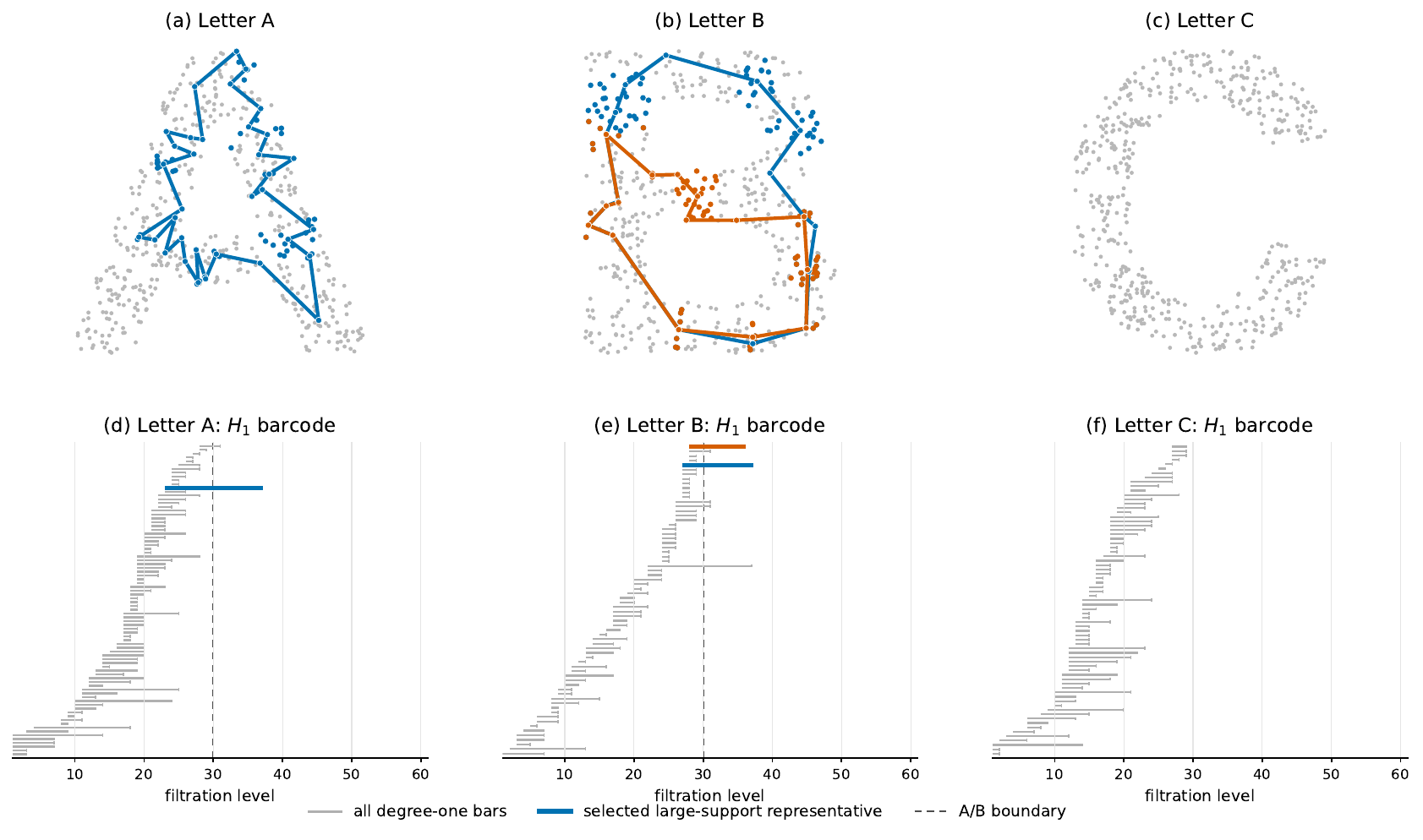}
  \caption{The letter data sets $A$, $B$, and $C$ and their degree-one barcodes for the same completed FTS configuration: $500$ sampled points, seed $1$, two-slot \texttt{profile} mode with raw A-sets, \texttt{size\_over\_distance} selection and removal scores, $\eta_0=0$, and $m=30$ quantile-scheduled A-removal levels. Gray lines denote all degree-one bars, and the dashed line marks the boundary between A-removal and B-completion. A stored birth representative is classified as large-support when its $T_0$-classes contain at least $10\%$ of the data. Matching colours identify selected bars and their displayed representatives. One such representative is shown for $A$ and two with different supports are shown for $B$; no representative for $C$ satisfies the support criterion. For visualization, the vertices of each representative are placed at the centroids of their birth-level $T_0$-classes.}
  \label{fig:letter-h1-barcodes}
\end{figure}

\Cref{fig:letter-h1-barcodes} illustrates the degree-one output of the order-complex filtration. It should not be interpreted as a shape-recovery test in which persistent degree-one classes are expected simply to count the visible planar holes. The order complexes encode inclusion relations among minimal opens, so their degree-one homology records cyclic finite-topological organization. Moreover, a homology class has many possible cycle representatives, and the polygonal paths in the figure are geometric displays of selected representatives rather than edges of a geometric complex constructed directly on the data.

Nevertheless, some representatives can have a suggestive relation to the geometry of the data. The representative selected for $A$ passes around its prominent central empty region, while the two representatives selected for $B$ have different point supports concentrated around different prominent regions. By contrast, the barcode for $C$ contains degree-one classes, but none of their stored birth representatives meets the large-support criterion. These examples indicate that neither the existence of a bar nor the geometry of one chosen representative is sufficient by itself to explain the meaning of an FTS degree-one class. \Cref{sec:incidence-obstruction} returns to this question. There we use the incidence spine to associate particular incidence paths with $H_1$-classes and then examine separately whether their point-level geometry indicates an obstruction.

% ====================================================================
\subsection{Deterministic core reduction}
\label{subsec:core-reduction}
% ====================================================================

In parallel with the full $T_0$-posets $P_i$, the implementation can reduce each level by deterministic beat-point removals. Let $q_i\colon X\longrightarrow P_i$ be the quotient map. Successively removing up-beat and down-beat points gives a Stong core $C_i\subseteq P_i$ and a retained retraction $r_i\colon P_i\to C_i$. When a beat point $a$ is removed with target $b$, its accumulated point block is transferred by
$$
\Pi(b)\longleftarrow\Pi(b)\cup\Pi(a).
$$
Consequently, a survivor $c\in C_i$ carries the core-class block
$$
\Pi_i(c)=\bigcup_{\substack{a\in P_i\\r_i(a)=c}}q_i^{-1}(a),\qquad X=\bigsqcup_{c\in C_i}\Pi_i(c).
$$
These blocks are unions of the original $T_0$-classes. They are distinct from the reduced minimal opens, which are
$$
\widetilde U_i(c)=\bigcup_{\substack{d\in C_i\\d\leq c}}\Pi_i(d).
$$
For consecutive displayed levels, the reduced maps are $g_{i,j}=r_j\circ p_{i,j}\circ\iota_i$. The standard barcodes are nevertheless computed from the full posets; cores and transferred blocks are retained as additional structural outputs. The deterministic removal priorities and complete audit records are specified in the accompanying guide.

\begin{figure}[!ht]
  \centering
  \includegraphics[width=\textwidth]{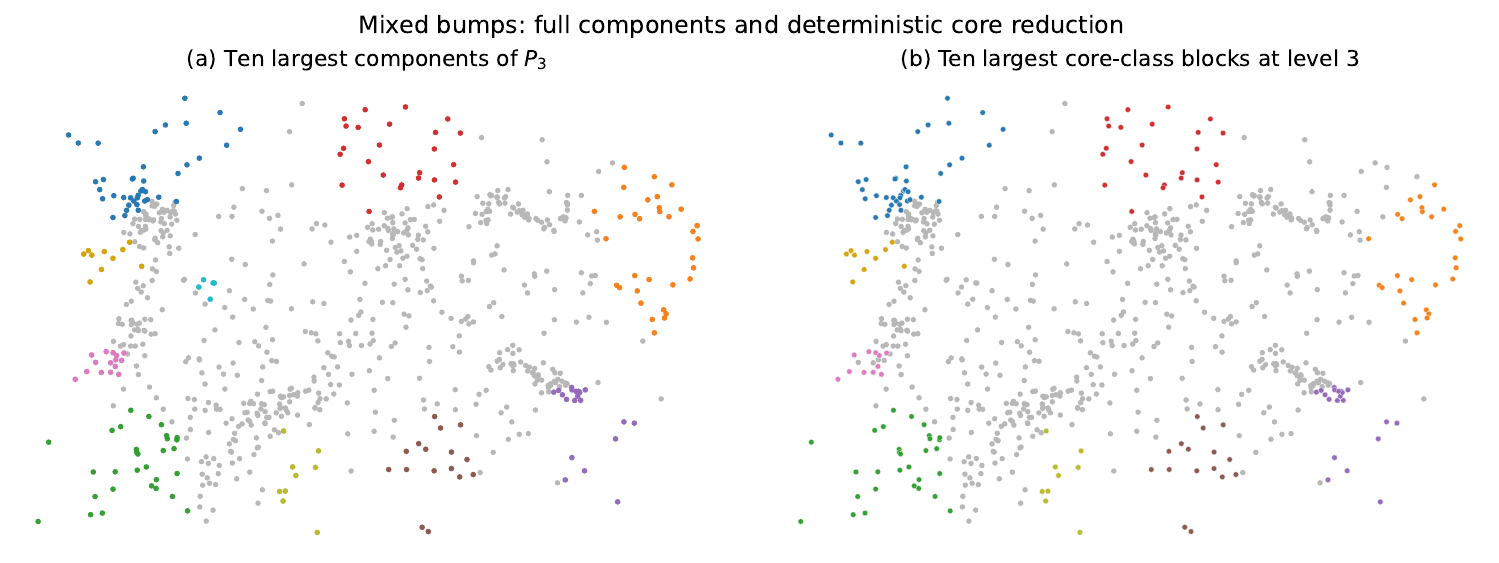}
  \caption{The ten largest connected components of $P_3$ and the ten largest core-class blocks obtained by deterministic core reduction for the mixed-bumps data. Each core-class block is coloured according to the connected component of $P_3$ containing it. Distinct core-class blocks may therefore have the same colour; in this example, the two purple groups in the right-hand panel are distinct blocks contained in the same component.}
  \label{fig:mixed-bumps-core-reduction}
\end{figure}

\Cref{fig:mixed-bumps-core-reduction} illustrates the relation between connected components and core-class blocks at a single displayed level: every core-class block is contained in a connected component, but one component may contain several distinct core-class blocks. The reductions are performed independently at the displayed levels, so the resulting core-class partitions need not be nested. The Paul15 analysis therefore treats the blocks and reduced minimal opens as levelwise structural objects and uses overlap-based correspondences when cross-level comparisons are required.

% ====================================================================
\subsection{Component and $T_0$-class merge trees}
\label{subsec:merge-trees}
% ====================================================================

The connected components and $T_0$-classes determine two coarsening hierarchies of partitions of $X$. For $x\in X$, write $[x]_i$ for its $T_0$-class and $\operatorname{Comp}_i(x)$ for the point support of its connected component. If $i\leq j$, then
$$
[x]_i\subseteq[x]_j\qquad\text{and}\qquad\operatorname{Comp}_i(x)\subseteq\operatorname{Comp}_j(x),
$$
so neither partition can split as the filtration advances. Organizing distinct component supports by their unique containing blocks produces the component merge forest; when there is one terminal component, it is a rooted merge tree. The same construction applied to $T_0$-class supports gives the finer $T_0$-class merge forest. Several $T_0$-classes can merge inside an already connected component, so these events need not change $\beta_0$. \Cref{fig:mixed-bumps-merge-trees} displays both hierarchies for the mixed-bumps filtration; the aligned levels make the finer $T_0$-class evolution within the component hierarchy visible.
\begin{figure}[!ht]
  \centering
  \includegraphics[width=\textwidth,height=0.54\textheight,keepaspectratio]{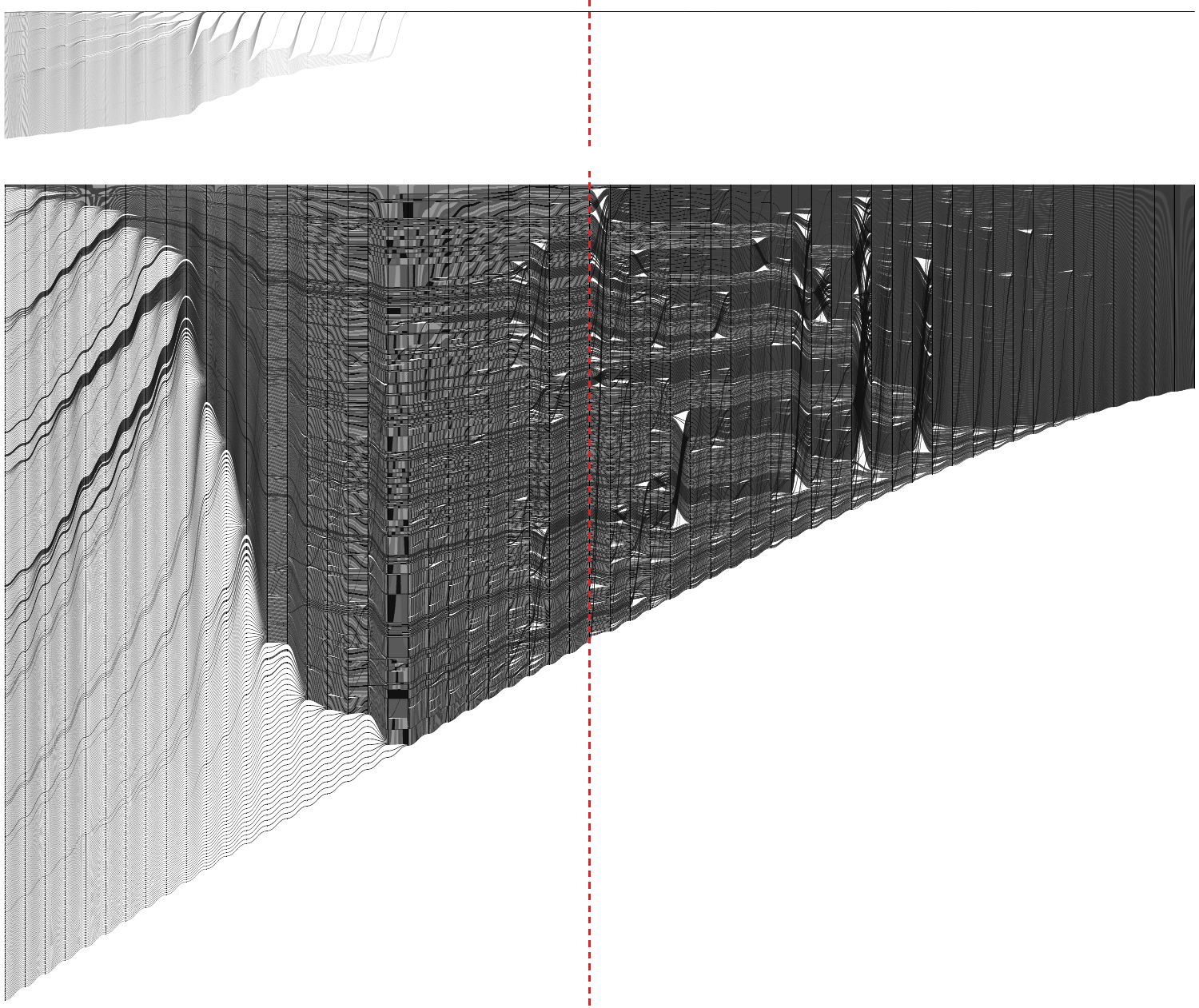}
  \caption{Component and $T_0$-class evolution for the mixed-bumps filtration. The upper panel shows the component merge tree, with line weight indicating component size. The lower panel gives a point-strand realization of the finer $T_0$-class hierarchy. Adjacent strands belonging to the same $T_0$-class form a block, while larger gaps separate connected components. The panels have aligned filtration levels, and the dashed red line marks level $m=30$, the boundary between A-set removal and B-completion.}
  \label{fig:mixed-bumps-merge-trees}
\end{figure}

Both hierarchies are determined by the level slices of the evolution poset and the common point supports; the metric enrichment is not required. Their implementation and deterministic drawing conventions are described in the accompanying guide.

% ====================================================================
\subsection{Incidence spines}
\label{subsec:incidence-spines}
% ====================================================================

The component merge forest records which connected components merge and when, but it does not describe the minimal-open incidences witnessing a merger or select point-level connections representing it. The incidence-spine construction supplements the component hierarchy with this information. It uses the completed filtration retrospectively: starting from a root component, it passes backward through the component merge hierarchy to the smaller components from which that component was formed.

The motivating idea is that component mergers can produce comparability chains surrounding a void region without enough transverse relations to form the triangles that would fill the resulting cycle in the order complex. This is only a motivating picture: a large-support $H_1$-class need not represent a geometric obstruction, and the incidence spine is itself a forest rather than a collection of cycles. \Cref{sec:incidence-obstruction} treats $H_1$-association and geometric obstruction as separate downstream analyses.

Let $C\subseteq X$ be a connected component that first occurs at an effective filtration level $i$, and let $i^-$ be the preceding effective level. The components at level $i^-$ contained in $C$ form its family of predecessor components
$$
\mathcal S_i(C)=\{D_1,\ldots,D_k\}.
$$
When $C$ is created by a merger, these components partition $C$ and disappear as separate components after level $i^-$. If $k=1$ and $D_1=C$, the component has merely persisted and there is no merger to encode.

For a component or an evolution-poset node $S$, write $b(S)$ and $d(S)$ for the first and last effective levels at which it occurs. A minimal open $V\in\mathcal M_{i^-}$ is \emph{eligible} for a predecessor component $D\in\mathcal S_i(C)$ when
$$
V\subseteq D,\qquad d(V)=d(D)=i^-.
$$
Thus eligibility identifies minimal opens whose disappearance belongs to the same transition as the merger of their components. Now let $U\in\mathcal M_i$ be a minimal open born together with $C$, so that
$$
b(U)=b(C)=i,\qquad U\subseteq C.
$$
The part of the earlier topology lying immediately below $U$ is its \emph{predecessor frontier}
$$
\mathcal F_{i^-}(U) = \operatorname{Max}\{V\in\mathcal M_{i^-}\colon V\subsetneq U\},
$$
where maximality is taken with respect to inclusion. The frontier is read in the preceding slice rather than from the Hasse diagram of the union of the two slices. Minimal opens born during the transition therefore cannot conceal the earlier branches participating in the merger.

The node $U$ is \emph{activated} when its predecessor frontier contains at least one eligible minimal open of some $D\in\mathcal S_i(C)$. Activation and grouping have different roles. Once $U$ has been activated, every member of its predecessor frontier is grouped by the unique predecessor component that contains it, whether or not that frontier member is itself eligible. Set
$$
\mathcal F_{i^-}(U;D) = \{V\in\mathcal F_{i^-}(U)\colon V\subseteq D\},\qquad \Gamma_U(D) = \bigcup_{V\in\mathcal F_{i^-}(U;D)}V.
$$
An activated node that represents at least two predecessor components will be called a \emph{merger witness}. For two predecessor components represented by the same merger witness $U$, define
$$
s_U(D_a,D_b) = \frac{1}{\operatorname{avgdist}\bigl(\Gamma_U(D_a),\Gamma_U(D_b)\bigr)},
$$
where
$$
\operatorname{avgdist}(A,B) = \frac{1}{|A||B|}\sum_{x\in A}\sum_{y\in B}D(x,y).
$$
If the nonempty groups represented by $U$ are indexed by $D_{j_1},\ldots,D_{j_q}$, put
$$
f(U) = \sum_{1\leq r<s\leq q}s_U(D_{j_r},D_{j_s}).
$$
For each represented pair, define its \emph{percentage score} by
$$
p_U(D_a,D_b) = \frac{s_U(D_a,D_b)}{f(U)}.
$$
Thus $f(U)$ is the total score of the connections represented by $U$, while $p_U(D_a,D_b)$ is the share contributed by the pair $(D_a,D_b)$. In particular, the percentage scores associated with a fixed merger witness sum to one.

\begin{construction}[Retrospective incidence-spine construction]
\label{constr:incidence-spine}
For a component $C$ born by a merger, form all scored triples $(U,D_a,D_b)$. Begin with a triple having highest percentage score. Starting from its two predecessor components, repeatedly select a highest-scoring triple having one endpoint in the current graph and the other in an as-yet unused member of $\mathcal S_i(C)$. The scores are fixed throughout this procedure, and exact ties are resolved deterministically. Each step adds one vertex and one edge, so the resulting local graph is a tree on the predecessor components.

Begin with a root component of the component merge forest and apply this local construction to its predecessor components. Then apply the same construction recursively to each predecessor component, moving toward earlier filtration levels and progressively smaller components, until no nontrivial predecessor decomposition remains. This backward recursion defines the hierarchical incidence structure of the root component.

The point-level realization is assembled in the opposite direction. An atomic component is represented by a deterministically chosen data point. After the point-level trees of two child components have been realized, an abstract edge joining them is represented by a closest pair of already realized vertices, one from each child tree.
\end{construction}

\begin{definition}[Incidence spine]
\label{def:incidence-spine}
For each root component $R$ of the component merge forest, let $\mathcal I(R)$ be the recursively organized, decorated point-level tree obtained from \Cref{constr:incidence-spine}. The \emph{incidence spine} of the completed filtration is the forest
$$
\mathcal I = \bigsqcup_R\mathcal I(R).
$$
Its vertices are selected points of $X$, and each edge retains its native merger level, witness minimal open, predecessor-frontier groups, and point-level metric realization. We call $\mathcal I(R)$ the incidence tree of $R$; the underlying graph of $\mathcal I$ is the incidence forest.
\end{definition}

The hierarchy is discovered retrospectively, from root components toward their earlier constituents, while its point-level realization is assembled from the atomic components upward. The component-merger events and minimal-open incidences determine the recursive organization, and the distance matrix ranks admissible connections and realizes the selected edges between data points. For each level $i$, let $\mathcal I_i$ denote the retrospective realization of the same fixed hierarchy. These realizations are not incidence spines recomputed independently from the filtration slices: passages contracted inside a component at a coarser level may expand or refine when moving backward to $i^-$. The maximally resolved realization $\mathcal I_1$ is the point-level carrier used in the later $H_1$-association and close-return constructions.

\begin{figure}[!ht]
  \centering
  \includegraphics[
  width=0.78\textwidth,
  height=0.48\textheight,
  keepaspectratio
  ]{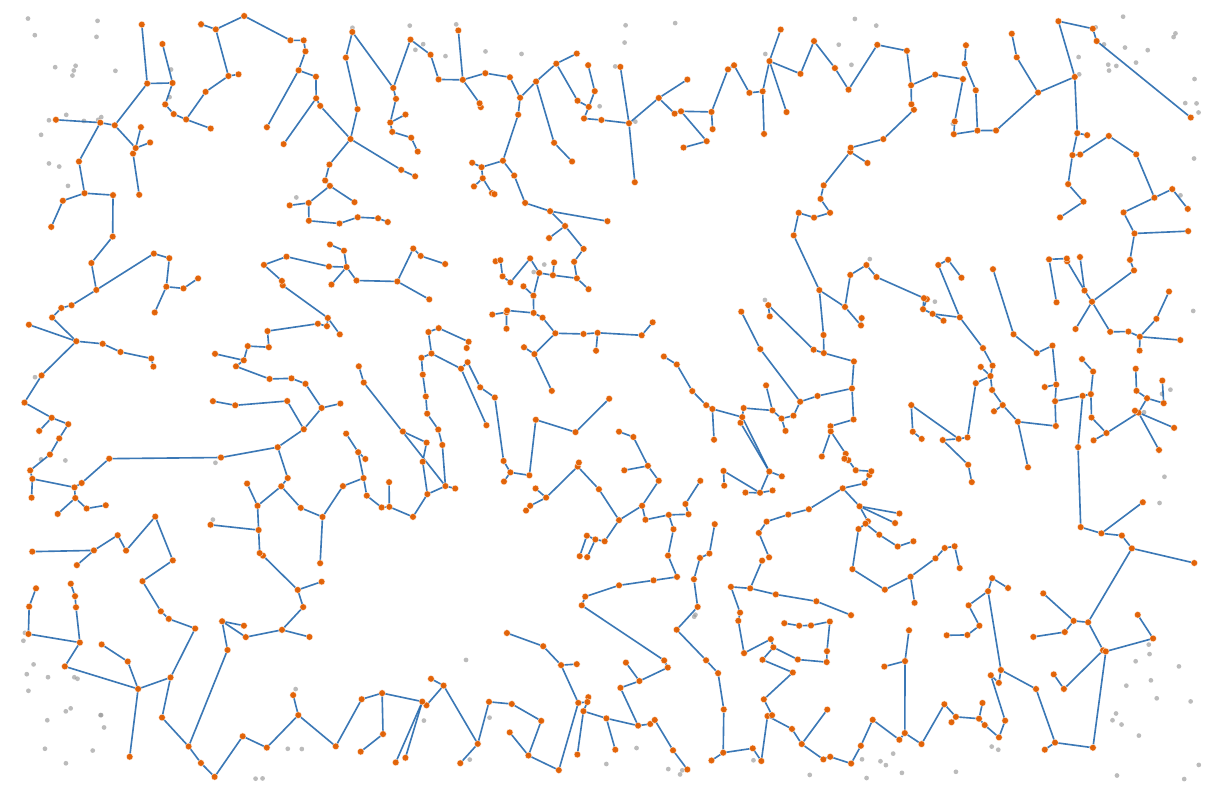}
  \caption{The maximally resolved level-one incidence spine for the mixed-holes data. Gray points are data points not used as spine vertices, blue segments are incidence edges, and orange points are the data points used as spine vertices. The drawing uses the original planar coordinates, while the filtration and the incidence-spine construction use only the distance matrix derived from them and do not otherwise use the planar embedding.}
  \label{fig:mixed-holes-incidence-spine}
\end{figure}

\Cref{fig:mixed-holes-incidence-spine} shows the maximally resolved carrier for the mixed-holes data. The displayed tree records component mergers and their minimal-open witnesses rather than directly recovering the visible holes. It nevertheless contains long paths with metrically close endpoints; closing such a path by its short return produces a candidate for the obstruction analysis of \Cref{sec:incidence-obstruction}, where the geometric criterion is tested separately.

\begin{figure}[!ht]
  \centering
  \includegraphics[
  width=\textwidth,
  height=0.42\textheight,
  keepaspectratio
  ]{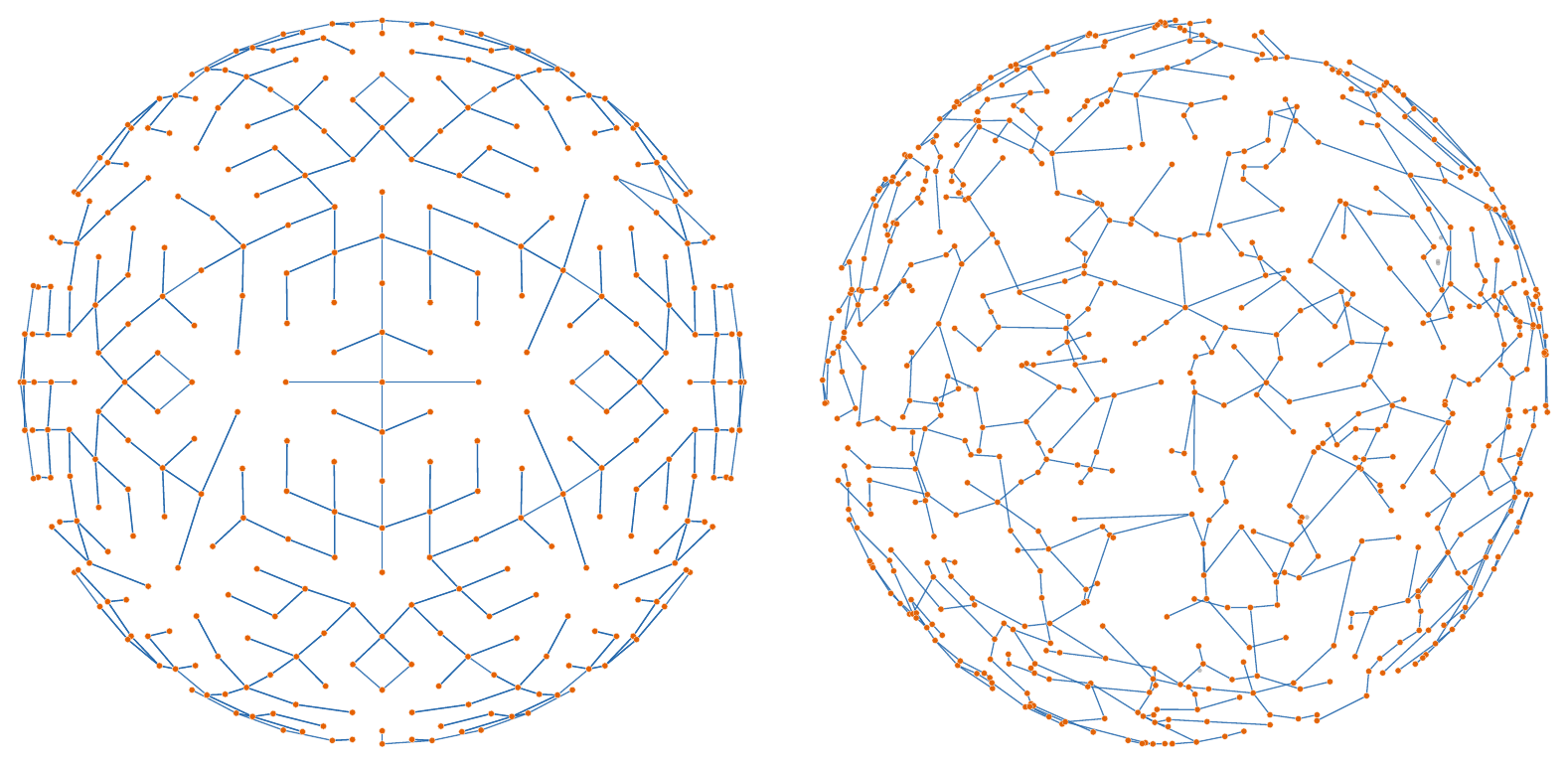}
  \caption{Maximally resolved level-one incidence spines for the icosahedral sphere sample (left) and the uniformly sampled sphere (right). Blue segments are incidence edges, orange points are data points used as spine vertices, and gray points, when present, are data points not used by the spine. Every point of the icosahedral sample is a spine vertex, while $632$ of the $642$ points of the uniform sample are spine vertices. Both spines are shown under the same orthogonal projection onto the first two coordinate axes; the constructions themselves use the distance matrices. Under this projection, distinct three-dimensional vertices may overlap or appear very close together, and unrelated edges may cross. Consequently, apparent cycles in the planar drawings are projection artifacts; the underlying incidence spines are forests.}
  \label{fig:sphere-incidence-spines}
\end{figure}

\Cref{fig:sphere-incidence-spines} illustrates that the construction does not depend on an ambient Euclidean dimension. Both filtrations and incidence spines use the distance matrices; the common planar projection is used only for display, and any apparent cycles created by that projection are not cycles in the underlying incidence forests.

The maximally resolved carrier also supplies distinguished paths relevant to degree-one homology. For vertices $a$ and $b$ in the same tree of $\mathcal I_1$, let $\gamma_{\mathcal I}(a,b)$ denote their unique incidence-spine path. If their images later become equal or comparable in a $T_0$-poset $P_i$, the transported path can be closed in $\mathcal K(P_i)$ by the endpoint identification or comparability edge. Transporting the resulting cycles through the filtration yields the incidence-generated persistence submodule defined and studied in \Cref{sec:incidence-obstruction}.

The obstruction analysis asks a separate question of the same frozen carrier. It begins with metrically close endpoints separated by a sufficiently long incidence-spine path and examines the point-level geometry of the resulting close-return cycle. A nonzero $H_1$-class is not required by the obstruction criterion, while the exhaustive $H_1$-association procedure imposes no close-return condition. Neither analysis modifies the incidence spine.

The accompanying repository records these structures through numerical and graphical reports. When requested, the levelwise $\beta_0$ panels display the $T_0$-classes, connected components, core-class blocks, and retrospective realizations $\mathcal I_i$, while the remaining reports include barcodes, merge trees, the evolution poset, incidence--$H_1$ associations, and their underlying numerical records. The adjustable \nolinkurl{run_*.sage} files write the selected outputs to \nolinkurl{pdf/}, \nolinkurl{svg/}, and \nolinkurl{csv/} subdirectories. The corresponding configuration fields and complete report organization are described in the accompanying guide.

% ====================================================================
\section{Stability}
\label{sec:stability}
% ====================================================================

The stability argument begins with a comparison that is common to the rule-generated and anchor-based A-stages. Local refinement controls the minimal opens of covered points; to continue the comparison after points become uncovered, one needs the uncovered policy to preserve an already established nesting of minimal opens. This gives a general A-stage interleaving criterion. For practical anchor filtrations, full B-memory supplies this stronger comparison for removal states of a fixed A-set dictionary, leading to a fixed-dictionary stability result. A complementary route transfers an interleaving from rule-generated references to Quillen-faithful anchor filtrations. The sample-dependent profile and adaptive constructions used in the experiments satisfy neither sufficient hypothesis automatically, so their completed filtrations are compared empirically under controlled metric perturbations.

% ====================================================================
\subsection{A-stage stability under local refinement}
\label{subsec:rule-refinement-stability}
% ====================================================================

The rule-generated and anchor-based constructions of \Cref{subsec:rule-based,subsec:anchor-based} both produce A-stage filtrations from nested finite families of subsets of the same labelled set $X$, together with faithful uncovered policies. The comparison below depends only on these data. We therefore formulate it for finite set families and specialize it afterward to the rule-generated and practical anchor settings. Let $\mathcal S,\mathcal S'\subseteq\mathcal P(X)$ be two such families.

\begin{definition}[Local refinement]
\label{def:local-rule-refinement}
We say that $\mathcal S'$ \emph{locally refines} $\mathcal S$, and write
$$
\mathcal S'\preccurlyeq\mathcal S,
$$
if, for every $x\in X$ and every $R\in\mathcal S$ containing $x$, there exists $R'\in\mathcal S'$ such that
$$
x\in R'\subseteq R.
$$
\end{definition}

The inclusion comparison of \Cref{lem:anchor-covered-comparison} is a special case: if $\mathcal A_t(X)\subseteq\mathcal R_t(X)$, then $\mathcal R_t(X)\preccurlyeq\mathcal A_t(X)$. Local refinement permits the same minimal-open comparison without requiring either family to be contained in the other. For any finite family $\mathcal S$ and any point $x$ covered by $\mathcal S$, write
$$
M_{\mathcal S}(x)=\bigcap_{\substack{R\in\mathcal S\\x\in R}}R.
$$

\begin{lemma}[Covered minimal opens under refinement]
\label{lem:rule-refinement-minimal-open}
If $\mathcal S'\preccurlyeq\mathcal S$ and $x$ is covered by $\mathcal S$, then $x$ is covered by $\mathcal S'$ and
$$
M_{\mathcal S'}(x)\subseteq M_{\mathcal S}(x).
$$
\end{lemma}

\begin{proof}
For every $R\in\mathcal S$ containing $x$, choose $R'\in\mathcal S'$ with $x\in R'\subseteq R$. The intersection $M_{\mathcal S'}(x)$ is contained in each chosen $R'$ and hence in every $R$ entering $M_{\mathcal S}(x)$.
\end{proof}

\Cref{lem:rule-refinement-minimal-open} gives the desired containment for points covered by $\mathcal S$. For points uncovered in one or both states, however, faithfulness alone does not compare the assignments made in two different systems. We therefore isolate the additional property needed for stability.

At a noninitial A-stage removal state, write $V_{\mathcal S}(x)$ for the minimal open inherited from the immediately preceding state. The covered-point lemma shows what local refinement contributes at the next step. The remaining issue is whether the assignments made after coverage is lost preserve a containment that has already been established at the preceding states.

\begin{definition}[Refinement-monotone uncovered policy]
\label{def:refinement-monotone-continuation}
A faithful uncovered policy is \emph{refinement-monotone on a collection of A-stage states} if the following holds. Consider two states with active families $\mathcal S'$ and $\mathcal S$ such that
$$
\mathcal S'\preccurlyeq\mathcal S
\qquad\text{and}\qquad
V_{\mathcal S'}(x)\subseteq V_{\mathcal S}(x)
\quad(x\in X).
$$
After applying the uncovered policy at the two states, their resulting minimal opens satisfy
$$
U_{\mathcal S'}(x)\subseteq U_{\mathcal S}(x)
\qquad(x\in X).
$$
\end{definition}

\begin{proposition}[A-stage stability under mutual local refinement]
\label{prop:rule-refinement-stability}
Let
$$
\mathbb X_{A}^{\mathcal S},\mathbb X_{A}^{\mathcal S'}\colon\mathbb R\longrightarrow\mathbf{FinTop}
$$
be A-stage filtrations determined by nested finite set families $\{\mathcal S_t\}_{t\in\mathbb R}$ and $\{\mathcal S'_t\}_{t\in\mathbb R}$ and the same faithful uncovered policy, assumed to be refinement-monotone on the collection of their removal states. Each may be a rule-generated A-stage as in \Cref{subsec:rule-based} or an anchor-based A-stage as in \Cref{subsec:anchor-based}. The filtrations and their active families are extended over $\mathbb R$ by keeping the initial states fixed before the first critical value and the exhausted terminal states fixed afterward. Write $U_t(x)$ and $U'_t(x)$ for their minimal opens, and assume that the active families cover $X$ on the initial constant region. Suppose that there exists $\varepsilon\geq0$ such that
$$
\mathcal S'_t\preccurlyeq\mathcal S_{t+\varepsilon}
\qquad\text{and}\qquad
\mathcal S_t\preccurlyeq\mathcal S'_{t+\varepsilon}
$$
for every $t\in\mathbb R$. Then
$$
U'_t(x)\subseteq U_{t+\varepsilon}(x)
\qquad\text{and}\qquad
U_t(x)\subseteq U'_{t+\varepsilon}(x)
$$
for every $x\in X$ and $t\in\mathbb R$. After applying the $T_0$-quotient, poset, order-complex, and homology functors, the resulting persistence modules are $\varepsilon$-interleaved in every degree. Consequently, their interleaving distance is at most $\varepsilon$, and the same bound holds for the bottleneck distance between their persistence diagrams.
\end{proposition}

\begin{proof}
The active families have only finitely many critical values, so it is enough to argue successively through a common finite refinement of their critical values and $\varepsilon$-shifts. On the initial constant region, both systems are covered, and the required containments follow from \Cref{lem:rule-refinement-minimal-open}. Suppose that they hold at the preceding pair of states. These minimal opens are the inherited minimal opens at the next pair of states. The local-refinement hypotheses and the inductively established containments of the inherited minimal opens therefore meet the hypotheses of \Cref{def:refinement-monotone-continuation} in both directions, which gives the required containments at the next pair of states.

The pointwise containments make the corresponding identity maps on $X$ continuous. Since all maps in the two filtrations are induced by $\id_X$, the comparison maps commute with the filtration maps, and their composites are the identity-induced maps from level $t$ to level $t+2\varepsilon$. The functorial construction of \Cref{subsec:abstract-fts}, together with the functoriality of homology recalled in \Cref{subsec:filtrations-background}, therefore gives an $\varepsilon$-interleaving of the induced persistence modules in the sense of \Cref{subsec:filtrations-background}. The bottleneck bound follows from \Cref{thm:persistence-isometry}.
\end{proof}

The proposition applies directly to rule-generated and anchor-based A-stages whenever their active families can be compared. It does not include the anchor-specific B-completion or assert that the practical B-set policy is refinement-monotone. 

% ====================================================================
\subsection{B-memory and full-memory monotonicity}
\label{subsec:eta-refinement-memory}
% ====================================================================

We now pass from the general A-stage criterion to the practical anchor filtrations of \Cref{subsec:anchor-based,sec:practical-constructions}. Their active families are surviving subfamilies of a selected finite A-set dictionary, and \Cref{subsec:b-sets} supplies minimal opens for points that become uncovered. \Cref{prop:practical-b-faithful} already shows that, for every permitted $\eta_0$, this policy propagates minimal opens monotonically through the active propagation states of a single practical filtration. Stability requires a different comparison: two removal schedules must be compared state by state, and \Cref{prop:rule-refinement-stability} applies only when the uncovered assignments preserve the nesting inherited from the preceding pair of states.

The memory parameter $\eta_0$ belongs only to the practical anchor construction and has no counterpart in the rule-generated reference. For finite memory, \Cref{subsec:b-sets} retains only a moving part of the removed A-set history. We do not assume that this truncated history is refinement-monotone across two removal schedules. At the full-memory endpoint $\eta_0=m$, however, every removed A-set remains available, and the required cross-state comparison can be proved. The first step is to separate the monotonicity of the B-set closure operation from the question of whether its initial prescriptions are nested.

At any practical A-stage state, let $C$ be the covered set and let $L^{(0)}(x)$ denote the initial prescription: the active-family intersection for a covered point and the B-set seed for an uncovered point. Covered prescriptions are held fixed, while uncovered prescriptions are enlarged until the finite-topology condition is satisfied.

\begin{lemma}[Monotonicity of B-set closure]
\label{lem:b-seed-monotonicity}
Consider two A-stage states with covered sets $C$ and $C'$. Suppose
$$
C\subseteq C'\qquad \text{ and }\qquad (L')^{(0)}(x)\subseteq L^{(0)}(x)\qquad(x\in X).
$$
Then their least fixed points satisfy
$$
U'(x)\subseteq U(x)
\qquad(x\in X).
$$
\end{lemma}

\begin{proof}
At the initial step the assumed pointwise containment $(L')^{(0)}(x)\subseteq L^{(0)}(x)$ holds for every $x$. Covered prescriptions are fixed throughout the closure iteration. For uncovered points, each update enlarges the current prescription only by unions of prescriptions indexed by points already present. Hence pointwise containment of the two families of prescriptions is preserved inductively at every update, including the case $x\in C'\setminus C$, where the left-hand prescription is fixed and the right-hand one may enlarge. Since $X$ is finite, both iterations stabilize, and the containment holds for their least fixed points.
\end{proof}

For a full-memory state with fixed dictionary $\mathcal A$ and active family $\mathcal S\subseteq\mathcal A$, write
$$
Q_{\mathcal S}^{(m)}(x)=\bigcup_{\substack{A\in\mathcal A\setminus\mathcal S\\x\in A}}A\qquad \text{ and }\qquad S_{\mathcal S}^{(m)}(x)=V_{\mathcal S}(x)\cup\bigcup_{z\in Q_{\mathcal S}^{(m)}(x)}V_{\mathcal S}(z).
$$
These are the remembered-support union and the corresponding B-set seed, written without reference to a particular level index.

\begin{proposition}[Full-memory monotonicity for a fixed dictionary]
\label{prop:full-memory-refinement-monotone}
Let $\mathcal A$ be a fixed finite A-set dictionary, and consider two of its removal states with active families
$$
\mathcal S\subseteq\mathcal S'\subseteq\mathcal A.
$$
Suppose that their inherited minimal opens satisfy
$$
V_{\mathcal S'}(x)\subseteq V_{\mathcal S}(x)
\qquad(x\in X).
$$
If both states use the full-memory policy $\eta_0=m$, then
$$
U_{\mathcal S'}(x)\subseteq U_{\mathcal S}(x)
\qquad(x\in X).
$$
\end{proposition}

\begin{proof}
Let $C$ and $C'$ be the covered sets of $\mathcal S$ and $\mathcal S'$, respectively. Since $\mathcal S\subseteq\mathcal S'$, one has $C\subseteq C'$. If $x\in C$, then $x$ is covered in both states and $M_{\mathcal S'}(x)\subseteq M_{\mathcal S}(x)$. If $x\in C'\setminus C$, choose $A\in\mathcal S'$ containing $x$. Since $x$ is uncovered by $\mathcal S$, the set $A$ does not belong to $\mathcal S$ and is therefore present in the full removal history of the $\mathcal S$-state. Consequently,
$$
M_{\mathcal S'}(x)
\subseteq A
\subseteq Q_{\mathcal S}^{(m)}(x)
\subseteq S_{\mathcal S}^{(m)}(x).
$$
Finally, if $x\notin C'$, then $x$ is uncovered in both states. The inclusion $\mathcal A\setminus\mathcal S'\subseteq \mathcal A\setminus\mathcal S$ gives $Q_{\mathcal S'}^{(m)}(x)
\subseteq Q_{\mathcal S}^{(m)}(x)$. Together with the assumed containment of the inherited minimal opens, the definitions give $S_{\mathcal S'}^{(m)}(x) \subseteq S_{\mathcal S}^{(m)}(x)$. Thus the initial prescriptions are nested at every point, and \Cref{lem:b-seed-monotonicity} gives the result.
\end{proof}

The proposition supplies exactly the additional hypothesis needed in \Cref{prop:rule-refinement-stability}: for a fixed dictionary, inclusion of active families together with nesting inherited from the preceding states is preserved after the full-memory B-assignment. We can therefore turn closeness of two removal schedules into an interleaving by expressing each schedule through the times at which its A-sets disappear.

For an exhausting removal schedule $\Sigma$ on a fixed dictionary $\mathcal A$, define the scheduled lifetime function $\lambda_{\Sigma}\colon\mathcal A\longrightarrow\mathbb R$ by letting $\lambda_{\Sigma}(A)$ be the first displayed parameter at which $A$ is no longer active. After extending the scheduled filtration piecewise constantly over $\mathbb R$, its active family is
$$
\mathcal A_t^{\Sigma}=\{A\in\mathcal A\colon\lambda_{\Sigma}(A)>t\}.
$$
The primary removal score orders the dictionary, while the schedule determines the removal time relevant to the stability comparison. Every schedule mode is represented here by its induced lifetime function.

\begin{corollary}[Fixed-dictionary stability under scheduled removal]
\label{cor:fixed-rule-score-stability}
Let $\mathcal A$ be a fixed finite A-set dictionary covering $X$, and let $\Sigma$ and $\Sigma'$ be two exhausting removal schedules with the same number $m$ of displayed levels on $\mathcal A$, expressed in a common real-valued level coordinate. Suppose that their scheduled lifetime functions satisfy
$$
\|\lambda_{\Sigma}-\lambda_{\Sigma'}\|_\infty\leq\varepsilon.
$$
If both filtrations use the full-memory policy $\eta_0=m$, then their A-stages, which in this case are already the completed FTS filtrations, induce $\varepsilon$-interleaved persistence modules in every degree. Consequently, the bottleneck distance between their persistence diagrams is at most $\varepsilon$ in every degree.
\end{corollary}

\begin{proof}
The lifetime bound gives
$$
\mathcal A_{t+\varepsilon}^{\Sigma}
\subseteq
\mathcal A_t^{\Sigma'}
\subseteq
\mathcal A_{t-\varepsilon}^{\Sigma}
$$
and the analogous inclusions with $\Sigma$ and $\Sigma'$ interchanged. These inclusions give the mutual shifted local refinements required by \Cref{prop:rule-refinement-stability}. \Cref{prop:full-memory-refinement-monotone} shows that the full-memory policy is refinement-monotone on the collection of shifted states occurring here. \Cref{prop:rule-refinement-stability} therefore gives an $\varepsilon$-interleaving of the induced persistence modules, and the bottleneck bound follows from \Cref{thm:persistence-isometry}. Since $\eta_0=m$, the complete removal history is already present during the A-stage and no separate B-completion levels remain.
\end{proof}

The corollary applies to any schedule mode through its induced lifetime function; it does not assert that two uses of the same mode automatically satisfy the displayed bound. Its fixed-dictionary hypothesis is essential: a metric perturbation may change the candidates selected by profile or adaptive mode, so the lifetime comparison no longer takes place on one common dictionary. The proof also singles out full memory because the complete removal history makes the cross-schedule nesting argument available. For $\eta_0<m$, only a moving window of removed A-sets is retained, and no analogous refinement-monotonicity theorem is asserted here. Full memory is therefore an exact but coarse reference endpoint for fixed-dictionary removal: its horizon is $L=m$, whereas smaller values retain more A-stage resolution and append $m-\eta_0$ B-completion levels.

% ====================================================================
\subsection{Reference transfer and practical design}
\label{subsec:quillen-stability-transfer}
% ====================================================================

The fixed-dictionary argument controls variation of removal times once the practical dictionary itself has been fixed. When the dictionary changes, a different comparison is available if the anchor filtration can instead be related to a rule-generated reference through the Quillen-faithfulness criterion of \Cref{subsec:anchor-based}.

\Cref{prop:rule-refinement-stability} gives the conditional stability result for rule-generated A-stages. Its hypotheses are particularly natural at the rule-generated level because the full admissible rule family provides many possible witnesses for local refinement: a set need not persist literally under perturbation if a sufficiently local admissible set survives inside it after the parameter shift. Thus the abundance of the rule-generated family makes the required containment relations plausible without making them automatic. Mutual $\varepsilon$-shifted local refinement of the active rule families, together with a refinement-monotone faithful uncovered policy, then yields an $\varepsilon$-interleaving of the induced persistence modules. The next result does not establish rule-generated stability anew; it transfers this interleaving to Quillen-faithful anchor-generated A-stages.

\begin{corollary}[Stability transfer through Quillen-faithful anchors]
\label{cor:quillen-stability-transfer}
For $j=0,1$, let $\mathbb X_{A,j}^{\mathcal R}$ be a rule-generated A-stage filtration and let $\mathbb X_{A,j}^{\mathcal A}$ be a Quillen-faithful anchor-generated A-stage filtration for it. If the two rule-generated A-stages satisfy \Cref{prop:rule-refinement-stability} with comparison constant $\varepsilon$, then the A-stage persistence modules of the two anchor filtrations are $\varepsilon$-interleaved in every degree.
\end{corollary}

\begin{proof}
The rule-generated modules are $\varepsilon$-interleaved, and \Cref{thm:quillen-faithful-persistence} gives a natural isomorphism from each of them to its anchor-generated counterpart. Transporting the interleaving through these isomorphisms proves the claim.
\end{proof}

The practical profile and adaptive modes used below are not covered automatically by either route. Their selected dictionaries may change under perturbation, and the experiments do not construct rule-generated references and verify the corresponding Quillen fibres. The empirical study therefore varies the ingredients that determine the practical evolution itself: how many local scales are retained in the A-set construction, how the resulting candidates are selected and removed, and how much removed history is carried by B-memory. The tested scores are
$$
\frac{s}{\delta},\qquad
\frac1\delta,\qquad
\frac1s,\qquad
\frac{\delta}{s}.
$$
Support per distance is the theory-motivated dense-support priority, while distance per support is its aligned reversal. Retaining two local scales may preserve intersections and delay premature identifications, but neither this choice nor coherent candidate preparation certifies Quillen faithfulness. B-memory changes a different part of the construction: it determines how much removed support is available when uncovered points are propagated. These choices are therefore compared jointly in the practical regime left open by the two sufficient stability results.

% ====================================================================
\subsection{Empirical stability of the practical filtrations}
\label{subsec:stability-guided-calibration}
% ====================================================================

The empirical study compares practical FTS filtrations under label-preserving metric perturbations. For each reference cloud $X\subseteq\mathbb R^d$, vectors $v_x$ are chosen independently and uniformly in the unit ball and
$$
X^\varepsilon=\{x+\varepsilon v_x\colon x\in X\}.
$$
The perturbation size is expressed relative to the median nearest-neighbour distance $d_{\mathrm{NN,med}}$ of the reference cloud:
$$
\varepsilon_{\mathrm{rel}}=\frac{\varepsilon}{d_{\mathrm{NN,med}}}
\in\{0.1,0.2,0.4,0.8\}.
$$
For a fixed reference sample and perturbation seed, the same vectors are used at all four magnitudes. The data sets and the varied and fixed construction parameters are summarized in \Cref{tab:stability-protocol}. We write profile2 and adaptive2 for the profile and adaptive constructions retaining at most two A-sets per centre.

\begin{table}[!ht]
\centering
\small
\caption{Design of the empirical stability experiment. Each score is used for both A-set selection and A-set removal. The construction parameters in the final row are fixed throughout this experiment.}
\label{tab:stability-protocol}
\begin{tabular}{lp{0.70\textwidth}}
\toprule
dimension & values \\
\midrule
data sets & letters $A,B,C$; mixed bumps; mixed holes; icosahedral and uniform spheres; icosahedral and uniform ellipsoids \\
A-set selection & one-set profile, two-set profile, and two-set adaptive; the latter two retain at most two A-sets per centre \\
candidate representation & raw and coherent \\
scores & aligned selection and removal with $s/\delta$, $1/\delta$, $1/s$, or $\delta/s$ \\
initial B-memory & $\eta_0\in\{0,1,2,m\}$ \\
relative perturbation & $\varepsilon_{\mathrm{rel}}\in\{0.1,0.2,0.4,0.8\}$ \\
seeds & reference seeds $0,1,2$ and perturbation seeds $1000,1001,1002,1003$; the two icosahedral data sets use only the deterministic reference seed $0$ \\
fixed parameters & $m=30$, quantile schedule, forward removal, and major-level propagation \\
\bottomrule
\end{tabular}
\end{table}

The three A-set selections, two candidate representations, four aligned score choices, and four memory depths give $3\cdot2\cdot4\cdot4=96$ practical configurations for each data set. For each fixed data set, configuration, and perturbation magnitude, the four perturbations associated with a reference sample are first averaged. Seven data sets have three independently generated reference samples, so the three resulting base-sample means are then averaged, summarizing $3\cdot4=12$ perturbed-reference comparisons. The reported standard deviation is the sample standard deviation of the three base-sample means. The icosahedral sphere and icosahedral ellipsoid each have one deterministic reference sample, so their four perturbations are averaged directly and no between-base standard deviation is reported. The complete panel contains $35\,328$ perturbed-reference comparisons and $2\,208$ reference runs. Seeds contribute to the estimated discrepancies but are not counted as separate data sets or as separate wins in the design comparisons.

The label correspondence identifies the points of a reference sample with those of its perturbation. For corresponding minimal opens $U_i(x)$ and $U'_i(x)$ at displayed level $i$, define
$$
d_{\mathrm{FTS}}(i)=\frac1{|X|^2}\sum_{x\in X}|U_i(x)\mathbin\triangle U'_i(x)|.
$$
The averages of $d_{\mathrm{FTS}}(i)$ over the displayed A-stage levels, B-stage levels, and all levels are denoted by $D_{\mathrm{FTS}}^A$, $D_{\mathrm{FTS}}^B$, and $D_{\mathrm{FTS}}^{\mathrm{all}}$, respectively. They measure labelled finite-topological discrepancy and are not interleaving distances. Homological change is measured by the bottleneck distances $d_B(H_0)$ and $d_B(H_1)$ of the completed persistence modules in displayed-level coordinates. The main tables report $d_B(H_k)/L$, where $L=2m-\eta_0$ is the completed horizon; the unnormalized bottleneck distances and $D_{\mathrm{FTS}}^B$ are retained in the complete repository tables.

\Cref{tab:stability-absolute} first gives the scale of the observed discrepancies in the principal raw two-set $s/\delta$ panel. It presents the median of the nine data-set means for each mode, memory depth, and perturbation magnitude.

\begin{table}[!ht]
\centering
\scriptsize
\caption{Observed discrepancies for the raw two-set $s/\delta$ configurations. Each entry is the median of the nine data-set means computed as described in the text. The bottleneck distances are normalized by the completed horizon $L$. Every bottleneck distance in this selected panel is finite.}
\label{tab:stability-absolute}
\begin{tabular}{llcrrrr}
\toprule
mode & $\eta_0$ & $\varepsilon_{\mathrm{rel}}$ & $D_{\mathrm{FTS}}^A$ & $D_{\mathrm{FTS}}^{\mathrm{all}}$ & $d_B(H_0)/L$ & $d_B(H_1)/L$ \\
\midrule
profile2 & $0$ & 0.1 & 0.004 & 0.007 & 0.056 & 0.087 \\
profile2 & $0$ & 0.2 & 0.005 & 0.010 & 0.058 & 0.085 \\
profile2 & $0$ & 0.4 & 0.007 & 0.015 & 0.060 & 0.092 \\
profile2 & $0$ & 0.8 & 0.009 & 0.020 & 0.101 & 0.097 \\
profile2 & $1$ & 0.1 & 0.007 & 0.009 & 0.025 & 0.071 \\
profile2 & $1$ & 0.2 & 0.010 & 0.012 & 0.031 & 0.077 \\
profile2 & $1$ & 0.4 & 0.009 & 0.017 & 0.042 & 0.080 \\
profile2 & $1$ & 0.8 & 0.013 & 0.024 & 0.042 & 0.082 \\
profile2 & $2$ & 0.1 & 0.010 & 0.013 & 0.027 & 0.062 \\
profile2 & $2$ & 0.2 & 0.013 & 0.020 & 0.032 & 0.057 \\
profile2 & $2$ & 0.4 & 0.017 & 0.023 & 0.033 & 0.054 \\
profile2 & $2$ & 0.8 & 0.023 & 0.030 & 0.040 & 0.073 \\
profile2 & $m$ & 0.1 & 0.006 & 0.006 & 0.047 & 0.088 \\
profile2 & $m$ & 0.2 & 0.010 & 0.010 & 0.061 & 0.096 \\
profile2 & $m$ & 0.4 & 0.015 & 0.015 & 0.072 & 0.100 \\
profile2 & $m$ & 0.8 & 0.021 & 0.021 & 0.081 & 0.092 \\
adaptive2 & $0$ & 0.1 & 0.004 & 0.006 & 0.071 & 0.071 \\
adaptive2 & $0$ & 0.2 & 0.004 & 0.009 & 0.079 & 0.079 \\
adaptive2 & $0$ & 0.4 & 0.006 & 0.013 & 0.092 & 0.085 \\
adaptive2 & $0$ & 0.8 & 0.008 & 0.018 & 0.075 & 0.078 \\
adaptive2 & $1$ & 0.1 & 0.007 & 0.008 & 0.034 & 0.072 \\
adaptive2 & $1$ & 0.2 & 0.010 & 0.011 & 0.045 & 0.070 \\
adaptive2 & $1$ & 0.4 & 0.009 & 0.015 & 0.047 & 0.083 \\
adaptive2 & $1$ & 0.8 & 0.012 & 0.020 & 0.045 & 0.081 \\
adaptive2 & $2$ & 0.1 & 0.010 & 0.011 & 0.032 & 0.069 \\
adaptive2 & $2$ & 0.2 & 0.015 & 0.016 & 0.036 & 0.083 \\
adaptive2 & $2$ & 0.4 & 0.019 & 0.021 & 0.042 & 0.083 \\
adaptive2 & $2$ & 0.8 & 0.023 & 0.024 & 0.045 & 0.083 \\
adaptive2 & $m$ & 0.1 & 0.006 & 0.006 & 0.064 & 0.100 \\
adaptive2 & $m$ & 0.2 & 0.009 & 0.009 & 0.069 & 0.096 \\
adaptive2 & $m$ & 0.4 & 0.014 & 0.014 & 0.081 & 0.107 \\
adaptive2 & $m$ & 0.8 & 0.019 & 0.019 & 0.086 & 0.092 \\
\bottomrule
\end{tabular}
\end{table}

Throughout this panel, $D_{\mathrm{FTS}}^A\leq0.023$ and $D_{\mathrm{FTS}}^{\mathrm{all}}\leq0.030$, and every bottleneck distance is finite. The bottleneck summaries respond less regularly to $\varepsilon_{\mathrm{rel}}$ than the labelled finite-topological discrepancies, reflecting discrete changes in the induced persistence modules.

The cross-data-set medians conceal substantial heterogeneity. \Cref{tab:stability-by-dataset} therefore gives the same four discrepancies separately for the nine data sets at the largest tested perturbation, using $\eta_0=0$.

\begin{table}[!ht]
\centering
\scriptsize
\caption{Data-set-specific discrepancies for the raw two-set $s/\delta$ configurations with $\eta_0=0$ at $\varepsilon_{\mathrm{rel}}=0.8$. For the seven data sets with three reference samples, each entry averages the three base-sample means, with their sample standard deviation in parentheses. For each deterministic icosahedral data set, the entry averages the four perturbations of its single reference sample and has no between-base standard deviation.}
\label{tab:stability-by-dataset}
\begin{tabular}{lrrrr}
\toprule
data & $D_{\mathrm{FTS}}^A$ & $D_{\mathrm{FTS}}^{\mathrm{all}}$ & $d_B(H_0)/L$ & $d_B(H_1)/L$ \\
\midrule
\multicolumn{5}{l}{\emph{profile2}}\\
letter A & 0.009 (0.001) & 0.023 (0.007) & 0.044 (0.019) & 0.119 (0.050) \\
letter B & 0.009 (0.001) & 0.020 (0.003) & 0.064 (0.057) & 0.097 (0.007) \\
letter C & 0.009 (0.001) & 0.020 (0.004) & 0.101 (0.078) & 0.128 (0.019) \\
mixed-bumps & 0.013 (0.002) & 0.020 (0.002) & 0.035 (0.009) & 0.208 (0.047) \\
mixed-holes & 0.007 (0.001) & 0.019 (0.001) & 0.118 (0.075) & 0.084 (0.014) \\
icosahedral sphere & 0.005 & 0.006 & 0.183 & 0.046 \\
uniform sphere & 0.011 (0.001) & 0.021 (0.003) & 0.137 (0.029) & 0.070 (0.009) \\
icosahedral ellipsoid & 0.018 & 0.037 & 0.312 & 0.096 \\
uniform ellipsoid & 0.010 (0.002) & 0.023 (0.004) & 0.098 (0.043) & 0.108 (0.055) \\
\addlinespace
\multicolumn{5}{l}{\emph{adaptive2}}\\
letter A & 0.008 (0.000) & 0.019 (0.004) & 0.036 (0.002) & 0.107 (0.017) \\
letter B & 0.008 (0.001) & 0.017 (0.002) & 0.033 (0.001) & 0.083 (0.019) \\
letter C & 0.008 (0.001) & 0.016 (0.003) & 0.085 (0.054) & 0.114 (0.020) \\
mixed-bumps & 0.013 (0.001) & 0.018 (0.001) & 0.065 (0.055) & 0.143 (0.030) \\
mixed-holes & 0.006 (0.001) & 0.016 (0.002) & 0.075 (0.025) & 0.078 (0.007) \\
icosahedral sphere & 0.004 & 0.006 & 0.183 & 0.050 \\
uniform sphere & 0.009 (0.001) & 0.020 (0.003) & 0.147 (0.027) & 0.048 (0.004) \\
icosahedral ellipsoid & 0.016 & 0.032 & 0.258 & 0.071 \\
uniform ellipsoid & 0.009 (0.001) & 0.020 (0.002) & 0.068 (0.029) & 0.065 (0.019) \\
\bottomrule
\end{tabular}
\end{table}

The most conspicuous effects occur in different degrees on different data sets: the icosahedral ellipsoid has the largest normalized $H_0$ distance, while mixed bumps has the largest normalized $H_1$ distance. The reported standard deviations also show that the variation among the three reference samples can be appreciable for the bottleneck distances even when the finite-topological discrepancies vary little.

The immediate comparisons among the principal design choices are collected in \Cref{tab:stability-design-contrasts}. For each data set and perturbation magnitude, the discrepancy difference (first setting minus second setting) is computed for every pair of configurations matched in the other three design dimensions, and these differences are averaged. An entry $a/b/c$ counts the data sets for which this mean difference favours the first setting, is tied, or favours the second setting. Lower discrepancy is preferred.

\begin{table}[!ht]
\centering
\scriptsize
\caption{Matched comparisons of the principal design choices at each perturbation magnitude. Within each data set, differences (first setting minus second setting) are averaged over all configuration pairs matched in the other three design dimensions. An entry $a/b/c$ counts data sets for which this mean difference favours the first setting, is tied, or favours the second setting.}
\label{tab:stability-design-contrasts}
\begin{tabular}{lccccc}
\toprule
comparison & $\varepsilon_{\mathrm{rel}}$ & $D_{\mathrm{FTS}}^A$ & $D_{\mathrm{FTS}}^{\mathrm{all}}$ & $d_B(H_0)/L$ & $d_B(H_1)/L$ \\
\midrule
profile2 versus one-set & 0.1 & 9/0/0 & 9/0/0 & 9/0/0 & 9/0/0 \\
profile2 versus one-set & 0.2 & 9/0/0 & 9/0/0 & 9/0/0 & 9/0/0 \\
profile2 versus one-set & 0.4 & 9/0/0 & 9/0/0 & 8/0/1 & 8/0/1 \\
profile2 versus one-set & 0.8 & 9/0/0 & 9/0/0 & 9/0/0 & 9/0/0 \\
\addlinespace
adaptive2 versus one-set & 0.1 & 9/0/0 & 9/0/0 & 9/0/0 & 9/0/0 \\
adaptive2 versus one-set & 0.2 & 9/0/0 & 9/0/0 & 9/0/0 & 8/0/1 \\
adaptive2 versus one-set & 0.4 & 9/0/0 & 9/0/0 & 9/0/0 & 8/0/1 \\
adaptive2 versus one-set & 0.8 & 9/0/0 & 9/0/0 & 9/0/0 & 9/0/0 \\
\addlinespace
adaptive2 versus profile2 & 0.1 & 8/0/1 & 8/0/1 & 8/0/1 & 5/0/4 \\
adaptive2 versus profile2 & 0.2 & 8/0/1 & 8/0/1 & 8/0/1 & 7/0/2 \\
adaptive2 versus profile2 & 0.4 & 9/0/0 & 9/0/0 & 9/0/0 & 9/0/0 \\
adaptive2 versus profile2 & 0.8 & 9/0/0 & 9/0/0 & 9/0/0 & 9/0/0 \\
\addlinespace
raw versus coherent & 0.1 & 9/0/0 & 9/0/0 & 9/0/0 & 9/0/0 \\
raw versus coherent & 0.2 & 9/0/0 & 9/0/0 & 9/0/0 & 9/0/0 \\
raw versus coherent & 0.4 & 9/0/0 & 9/0/0 & 9/0/0 & 9/0/0 \\
raw versus coherent & 0.8 & 8/0/1 & 9/0/0 & 9/0/0 & 9/0/0 \\
\addlinespace
$s/\delta$ versus $1/\delta$ & 0.1 & 2/0/7 & 8/0/1 & 9/0/0 & 8/0/1 \\
$s/\delta$ versus $1/\delta$ & 0.2 & 3/0/6 & 8/0/1 & 9/0/0 & 8/0/1 \\
$s/\delta$ versus $1/\delta$ & 0.4 & 3/0/6 & 9/0/0 & 9/0/0 & 8/0/1 \\
$s/\delta$ versus $1/\delta$ & 0.8 & 3/0/6 & 9/0/0 & 9/0/0 & 7/0/2 \\
\addlinespace
$s/\delta$ versus $1/s$ & 0.1 & 4/0/5 & 9/0/0 & 9/0/0 & 7/0/2 \\
$s/\delta$ versus $1/s$ & 0.2 & 4/0/5 & 9/0/0 & 9/0/0 & 6/0/3 \\
$s/\delta$ versus $1/s$ & 0.4 & 4/0/5 & 9/0/0 & 9/0/0 & 8/0/1 \\
$s/\delta$ versus $1/s$ & 0.8 & 3/0/6 & 8/0/1 & 9/0/0 & 8/0/1 \\
\addlinespace
$s/\delta$ versus $\delta/s$ & 0.1 & 5/0/4 & 9/0/0 & 8/0/1 & 8/0/1 \\
$s/\delta$ versus $\delta/s$ & 0.2 & 6/0/3 & 9/0/0 & 8/0/1 & 7/0/2 \\
$s/\delta$ versus $\delta/s$ & 0.4 & 5/0/4 & 9/0/0 & 8/0/1 & 7/0/2 \\
$s/\delta$ versus $\delta/s$ & 0.8 & 4/0/5 & 9/0/0 & 8/0/1 & 7/0/2 \\
\addlinespace
\bottomrule
\end{tabular}
\end{table}

Both two-set constructions are preferred to the one-set construction almost uniformly, and raw candidates are preferred to coherent candidates throughout the completed and homological comparisons. Adaptive2 and profile2 are closer at the two smaller perturbation magnitudes, especially for $H_1$, whereas adaptive2 is preferred on all nine data sets for all four discrepancies at $\varepsilon_{\mathrm{rel}}=0.4$ and $0.8$. The score comparison exhibits a consistent division. The score $1/\delta$ is usually preferred to $s/\delta$ for $D_{\mathrm{FTS}}^A$, but $s/\delta$ is preferred for $D_{\mathrm{FTS}}^{\mathrm{all}}$ and for both bottleneck distances on most data sets. The comparisons with $1/s$ and $\delta/s$ show the same distinction less uniformly at the A-stage. Thus the score that most closely preserves the A-stage minimal opens need not best preserve the completed topology or its homology.

Finally, \Cref{tab:stability-memory-resolution} compares the four values of $\eta_0$ in the raw two-set $s/\delta$ panel at $\varepsilon_{\mathrm{rel}}=0.8$. The stability quantities are accompanied by the completed horizon, the mean A-stage $T_0$ resolution, and the number $|P_m|$ of $T_0$-classes at A-set exhaustion.

\begin{table}[!ht]
\centering
\scriptsize
\caption{B-memory, observed stability, and finite-topological resolution for the raw two-set $s/\delta$ configurations at $\varepsilon_{\mathrm{rel}}=0.8$. Here $L=2m-\eta_0$, A-stage $T_0$ resolution is the average of $|P_i|/|X|$ over the A-stage levels, and $|P_m|$ is the number of $T_0$-classes at A-set exhaustion. Entries are medians of the nine data-set summaries described in the text.}
\label{tab:stability-memory-resolution}
\resizebox{0.95\textwidth}{!}{%
\begin{tabular}{llrrrrrrr}
\toprule
mode & $\eta_0$ & $L$ & A-stage $T_0$ resolution & $|P_m|$ & $D_{\mathrm{FTS}}^A$ & $D_{\mathrm{FTS}}^{\mathrm{all}}$ & $d_B(H_0)/L$ & $d_B(H_1)/L$ \\
\midrule
profile2 & $0$ & 60 & 0.836 & 241.0 & 0.009 & 0.020 & 0.101 & 0.097 \\
profile2 & $1$ & 59 & 0.820 & 64.3 & 0.013 & 0.024 & 0.042 & 0.082 \\
profile2 & $2$ & 58 & 0.813 & 17.7 & 0.023 & 0.030 & 0.040 & 0.073 \\
profile2 & $m$ & 30 & 0.804 & 1.0 & 0.021 & 0.021 & 0.081 & 0.092 \\
adaptive2 & $0$ & 60 & 0.859 & 256.3 & 0.008 & 0.018 & 0.075 & 0.078 \\
adaptive2 & $1$ & 59 & 0.841 & 57.0 & 0.012 & 0.020 & 0.045 & 0.081 \\
adaptive2 & $2$ & 58 & 0.834 & 11.3 & 0.023 & 0.024 & 0.045 & 0.083 \\
adaptive2 & $m$ & 30 & 0.830 & 1.0 & 0.019 & 0.019 & 0.086 & 0.092 \\
\bottomrule
\end{tabular}%
}
\end{table}

For profile2, $\eta_0=1$ and $2$ reduce both normalized bottleneck distances relative to $\eta_0=0$, with $\eta_0=2$ giving the smallest displayed $H_0$ and $H_1$ values. For adaptive2, the same memory depths reduce the $H_0$ distance but do not improve $H_1$. In both modes, the improvement in some homological comparisons is accompanied by larger $D_{\mathrm{FTS}}^A$ and a sharp decrease in $|P_m|$. Full memory $\eta_0=m$ gives $|P_m|=1$ in the cross-data-set median and is not the empirical optimum for the normalized bottleneck distances. It is therefore a theoretically distinguished fixed-dictionary endpoint, as established in \Cref{subsec:eta-refinement-memory}, rather than a universally preferable practical setting.

The theoretical results give conditional interleaving statements for rule-generated A-stages, fixed-dictionary full-memory anchor filtrations, and Quillen-faithful transfers, but not for the sample-dependent profile and adaptive constructions. Within the tested perturbations, the principal raw two-set $s/\delta$ configurations have small labelled finite-topological discrepancies, while the normalized bottleneck distances remain data- and degree-dependent at the reported scales. The comparisons favour two local scales, show that $s/\delta$ balances completed topology and homology despite a different A-stage preference, and show that memory effects are accompanied by a substantial resolution trade-off rather than by a uniform empirical improvement.

% ====================================================================
\section{Finite-Topological Structure and Clustering}
\label{sec:clustering}
% ====================================================================

We study how cluster structure and relations among clusters are represented in an FTS filtration by connected components, $T_0$-classes, and minimal opens. We begin with \texttt{mixed\_bumps}, a synthetic data set with six prescribed regions of different shapes, sizes, and densities. Since these regions are known in advance but are not used in constructing the filtration, they provide a controlled way to examine what the finite topology detects. We first study $T_0$-classes and minimal opens during the B-stage and compare them with the prescribed regions. We then return to the A-stage and remove observations according to their uncovered history, asking whether the remaining observations separate into connected components corresponding to the same regions. Next, we compare restriction with recomputation to determine how much of this separation is already present in the induced topology and how much depends on rebuilding the neighbourhood system. We then compare the finite-topological structures on \texttt{mixed\_bumps} with ToMATo, first at the level of localized regions and then through separating supports obtained from selected corridors and their minimal opens. Finally, we consider the Paul15 single-cell data, where there are no prescribed geometric regions and the evolving finite topology is instead used to examine cellular states, their relations and assembly, and the effect of removing observations connecting them; the published cell annotations are introduced only afterward for interpretation.

% ====================================================================
\subsection{Minimal opens and $T_0$-classes on \texttt{mixed\_bumps}}
\label{subsec:mixed-bumps-clustering}
% ====================================================================

\subsubsection*{Data and evaluation}

The seed-0 \texttt{mixed\_bumps} sample is shown in \Cref{fig:mixed-bumps-reference-regions}.  It contains 420 observations from a warped heterogeneous background, 40 uniform observations, and six localized sources.  We enumerate them as $G_1$: round dense, $G_2$: large ellipse, $G_3$: medium round, $G_4$: line-like, $G_5$: curved filament, and $G_6$: small weak ellipse. Their planted cardinalities are $70,60,45,55,50$, and $35$.  Seeds 1 and 2 regenerate the complete sample with the same parameters.

\begin{figure}[!ht]
\centering
\includegraphics[width=\textwidth]{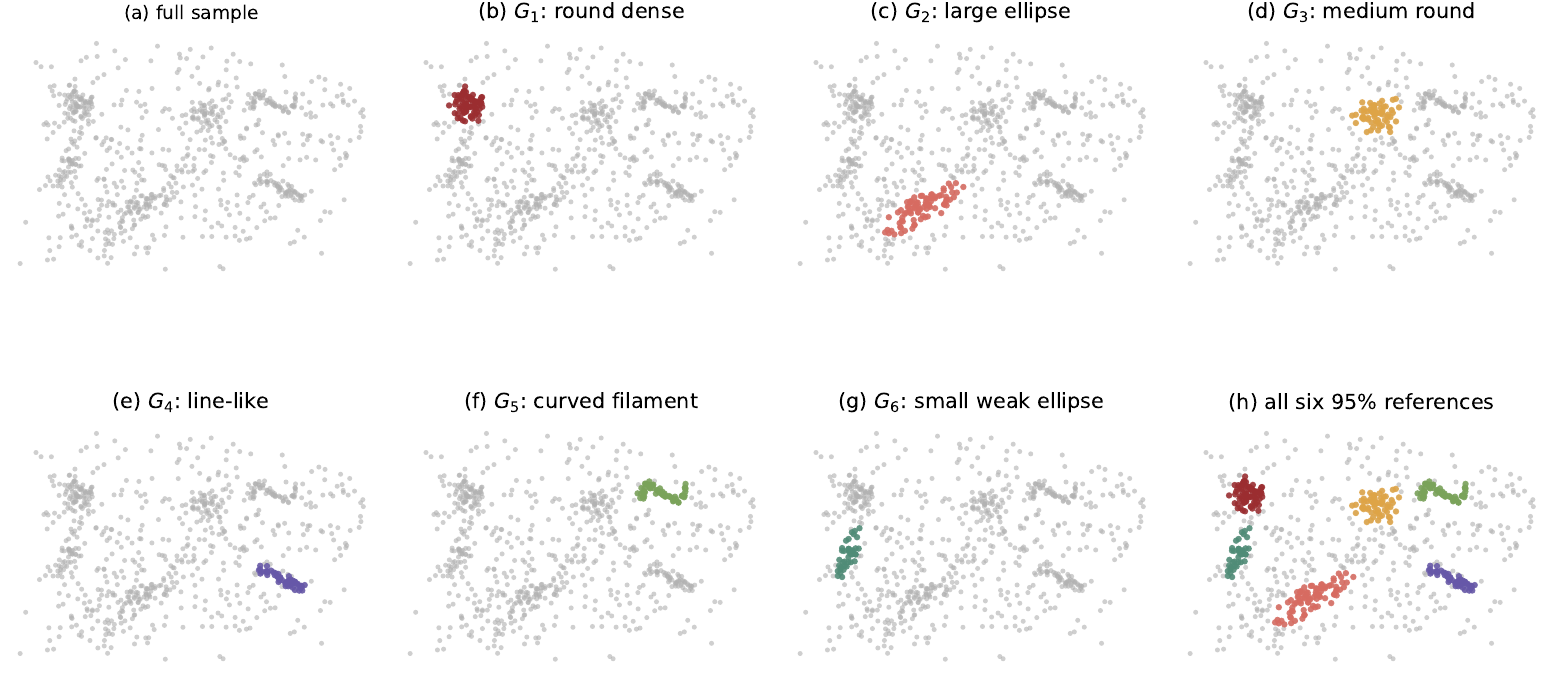}
\caption{\texttt{mixed\_bumps}, seed 0.  Panel (a) shows the full sample. Panels (b)--(g) show, respectively, the round dense, large ellipse, medium round, line-like, curved-filament, and small weak-ellipse $95\%$ reference regions; panel (h) shows all six together.  A highlighted set contains every observed point in the corresponding geometric region, including background points.  The regions are used only for posthoc evaluation.}
\label{fig:mixed-bumps-reference-regions}
\end{figure}

For source $r$, let $G_r^{(0.95)}$ and $G_r^{(0.90)}$ denote the observations inside its prescribed $95\%$ and $90\%$ geometric regions. For a round or elliptical Gaussian source, $G_r^{(0.95)}$ is the set of observations lying in the ellipse that contains $95\%$ of the corresponding generating Gaussian distribution; $G_r^{(0.90)}$ is defined analogously. The $95\%$ ellipse has principal semiaxes approximately $2.45$ times the corresponding standard deviations. The line-like source is sampled uniformly along its generating segment with Gaussian displacement in the normal direction, so its $95\%$ reference is the band extending approximately $1.96$ normal standard deviations on either side of the segment. The curved filament is generated around a prescribed sinusoidal curve with Gaussian transverse displacement and a much smaller longitudinal jitter; its reference is defined analogously as a band around the generating curve. The resulting seed-0 reference sizes are reported in \Cref{tab:mixed-bumps-reference-regions}.

We also use a label-dependent sensitivity reference $M_r^{(2)}$. For each $x\in S_r$, let $d_r(x)$ be its distance to the nearest other observation in the planted source $S_r$, and let $q_r$ be the $90$th percentile of these nearest-neighbour distances. Thus $q_r$ is a single source-specific local scale, not a point-dependent threshold. The reference $M_r^{(2)}$ consists of $S_r$ together with every background observation whose distance to $S_r$ is at most $2q_r$. This reference is used only as a sensitivity check.

\begin{table}[!ht]
\centering
\small
\caption{Posthoc reference sizes for the seed-0 \texttt{mixed\_bumps} realization; all references are used only after the finite-topological supports have been constructed.}
\label{tab:mixed-bumps-reference-regions}
\begin{tabular}{@{}clrrr@{}}
\toprule
region & source geometry & $|G^{(0.95)}|$ & $|G^{(0.90)}|$ & $|M^{(2)}|$ \\
\midrule
1 & round dense & 77 & 73 & 85 \\
2 & large ellipse & 86 & 78 & 112 \\
3 & medium round & 62 & 57 & 73 \\
4 & line-like & 51 & 48 & 60 \\
5 & curved filament & 57 & 54 & 60 \\
6 & small weak ellipse & 41 & 38 & 53 \\
\bottomrule
\end{tabular}
\end{table}

The three reference families in the table are used only for evaluation after the finite-topological supports have been constructed. For finite $S,T\subseteq X$, write
$$
 J(S,T)=\frac{|S\cap T|}{|S\cup T|}.
$$
The primary score is $J(S,G_r^{(0.95)})$; the other two references provide sensitivity checks. Because the three seeds generate different point clouds, there is no pointwise cross-seed Jaccard comparison. Instead, the seeds test whether the same procedure produces comparable matches on independently generated samples.

All experiments use 30 A-levels followed by B-completion with $\eta_0=0$. We compare profile and adaptive dictionaries, raw and coherent A-sets, and the scores $s/\delta$ and $1/\delta$, giving eight configurations. At level $i$, let $U_i(x)$ be the minimal open of $x$ and
$$
 C_i(x)=\{y\in X:U_i(y)=U_i(x)\}
$$
its $T_0$-class.  Points are in the same $T_0$-class precisely when they have the same minimal open. For a class $C$, a containing minimal open $U$ is called \emph{tight} when $|C|/|U|\geq .90$; such an open may add a small number of observations to the class and may overlap the corresponding opens of other
classes.

\subsubsection*{Availability and target-blind selection}

The first question is whether the B-stage contains localized supports of the right kind at all. For each prescribed region, the availability diagnostic chooses the distance-matrix medoid of the observations in $G_r^{(0.90)}$, follows its $T_0$-class through the B-levels, and considers the tight minimal opens containing that class. The primary $95\%$ reference is then used only to record the best posthoc level on the class path and on the minimal-open path. \Cref{tab:mixed-bumps-central-results} reports these matches: the mean Jaccards are $.828$--$.915$ for the central $T_0$-classes and $.848$--$.937$ for their containing minimal opens. The point of this calculation is existence rather than selection: all eight
filtrations contain supports close to the prescribed regions, so a label-free rule has meaningful structure from which to select.

\begin{table}[!ht]
\centering
\small
\caption{Target-guided availability of central $T_0$-classes and their tight containing minimal opens. Entries are $J_{95}$ summaries averaged over the three independent realizations.}
\label{tab:mixed-bumps-central-results}
\setlength{\tabcolsep}{4.5pt}
\begin{tabular}{@{}lrrrr@{}}
\toprule
configuration & \shortstack{class\\mean $J_{95}$} & \shortstack{class\\min $J_{95}$} & \shortstack{open\\mean $J_{95}$} & \shortstack{open\\min $J_{95}$} \\
\midrule
profile raw $s/\delta$ & 0.878 & 0.778 & 0.896 & 0.800 \\
profile coherent $s/\delta$ & 0.880 & 0.788 & 0.899 & 0.813 \\
adaptive raw $s/\delta$ & 0.876 & 0.781 & 0.894 & 0.816 \\
adaptive coherent $s/\delta$ & 0.855 & 0.744 & 0.881 & 0.802 \\
profile raw $1/\delta$ & 0.892 & 0.760 & 0.914 & 0.815 \\
profile coherent $1/\delta$ & 0.885 & 0.769 & 0.898 & 0.784 \\
adaptive raw $1/\delta$ & 0.889 & 0.771 & 0.916 & 0.815 \\
adaptive coherent $1/\delta$ & 0.888 & 0.792 & 0.919 & 0.839 \\
\bottomrule
\end{tabular}
\end{table}

Existence by itself does not provide a clustering rule. We therefore discard the reference regions at the selection stage and construct stable $T_0$-class corridors from the filtration alone. At each B-level the rule retains classes $C$ with $|C|\geq10$ whose own minimal open $F_i(C)$ satisfies $|C|/|F_i(C)|\geq .95$. Equal supports are collapsed, consecutive unequal supports are linked when their Jaccard overlap is at least $.80$, and a corridor must occupy at least three B-levels. If $w(D)$ is the number of levels represented by support $D$ in corridor $K$, the representative class is the level-weighted Jaccard medoid
$$
 C^*(K)=\operatorname*{argmax}_{C\in K}
 \frac{\sum_{D\in K}w(D)J(C,D)}{\sum_{D\in K}w(D)}.
$$
Ties are resolved deterministically. After $C^*(K)$ has been fixed, we consider all tight minimal opens $U\supseteq C^*(K)$. For each such $U$, we compute the average distance from
an observation of $U$ to its ten nearest neighbours within $U$, and report the open with the smallest average. Thus the corridor evolution chooses the $T_0$-class, while this final compactness criterion chooses one of its possible containing minimal opens.

Depending on the configuration and seed, the target-blind rule produces 10--19 corridors and therefore 10--19 representative classes. Only after these representatives and their selected opens have been frozen do we match six of them one-to-one to the six reference regions by maximum total Jaccard. The reference regions therefore determine neither the corridors nor their representatives. \Cref{tab:mixed-bumps-t0-selection} reports the evaluation. The selected minimal opens have overall mean $J_{95}=.839$, with configuration--seed means between $.766$ and $.896$.

\begin{table}[!ht]
\centering
\small
\caption{Target-blind stable-corridor selection on \texttt{mixed\_bumps}, averaged over the three independent realizations. Entries after the corridor count are posthoc mean $J_{95}$ scores.}
\label{tab:mixed-bumps-t0-selection}
\setlength{\tabcolsep}{3.2pt}
\begin{tabular}{@{}lrrrrr@{}}
\toprule
configuration & corridors & \shortstack{representative\\$T_0$-class} & \shortstack{selected\\tight open} & \shortstack{all corridor\\$T_0$-classes} & \shortstack{all tight opens\\of representative} \\
\midrule
profile raw $s/\delta$ & 11.7 & 0.832 & 0.832 & 0.878 & 0.850 \\
profile coherent $s/\delta$ & 12.7 & 0.831 & 0.829 & 0.880 & 0.856 \\
adaptive raw $s/\delta$ & 10.3 & 0.821 & 0.820 & 0.876 & 0.833 \\
adaptive coherent $s/\delta$ & 10.0 & 0.825 & 0.830 & 0.855 & 0.854 \\
profile raw $1/\delta$ & 13.3 & 0.859 & 0.861 & 0.892 & 0.881 \\
profile coherent $1/\delta$ & 13.7 & 0.840 & 0.843 & 0.885 & 0.858 \\
adaptive raw $1/\delta$ & 13.3 & 0.857 & 0.857 & 0.889 & 0.880 \\
adaptive coherent $1/\delta$ & 17.0 & 0.842 & 0.842 & 0.888 & 0.871 \\
\bottomrule
\end{tabular}
\end{table}

The selected supports should also be read against the richer objects from which they were compressed. The last two columns of \Cref{tab:mixed-bumps-t0-selection} do exactly this.  Allowing every $T_0$-class occurring in the retained corridors raises the mean to $.880$. After one representative $T_0$-class has been chosen from each corridor, allowing all of its tight containing minimal opens gives $.860$, compared with $.839$ after selecting only the compact open. The loss therefore occurs mainly when an evolution corridor is summarized by a small number of representatives, not because the corresponding localized supports are absent from the filtration. The same fixed target-blind rule is retained for all modes and seeds rather than being adjusted to close this posthoc gap.

\begin{figure}[!ht]
\centering
\includegraphics[width=\textwidth]{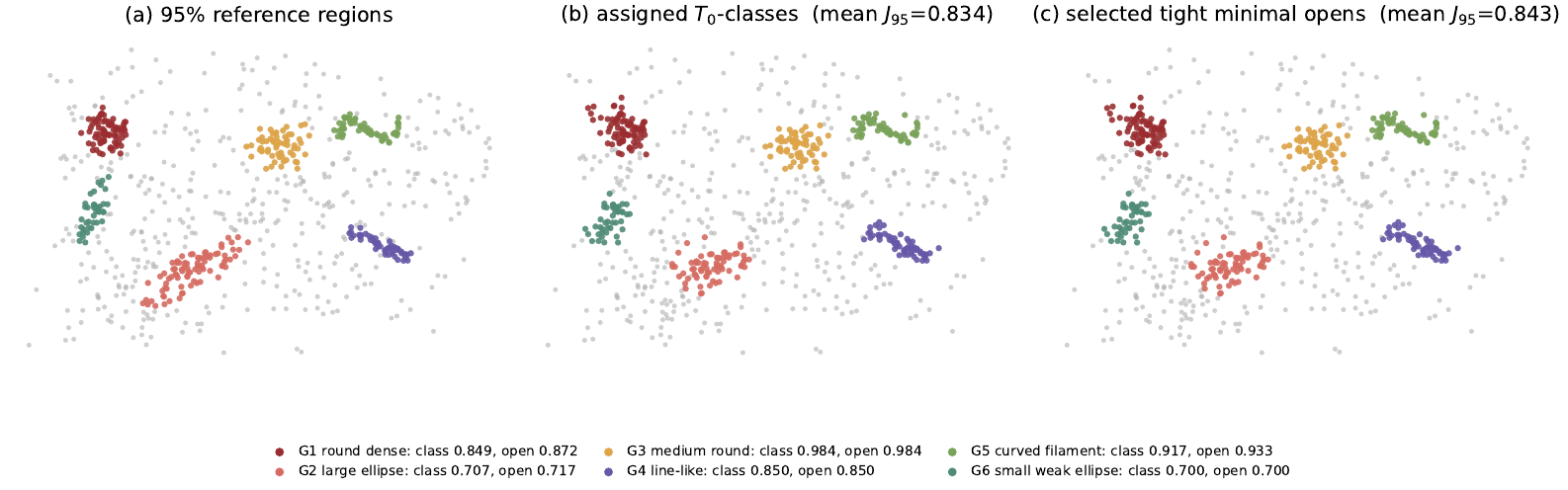}
\caption{An illustrative target-blind recovery on \texttt{mixed\_bumps} (seed 0, adaptive coherent $s/\delta$, $\eta_0=0$), chosen from the fixed experimental panel to display the relation between representative $T_0$-classes and their selected tight minimal opens. Panel (a) shows the six $95\%$ reference regions. Panel (b) shows the six $T_0$-classes selected as corridor representatives by the weighted-medoid rule and then assigned posthoc to the six references. Panel (c) shows the selected tight minimal opens containing those classes. For some regions the selected open coincides with its $T_0$-class, while for others it adds a small number of observations. The legend reports, for each region, the posthoc $J_{95}$ of the displayed $T_0$-class and minimal open, respectively. The references are used only for this final assignment and evaluation.}
\label{fig:mixed-bumps-target-blind}
\end{figure}

\Cref{fig:mixed-bumps-target-blind} makes the relation between the two support types visible. The six displayed $T_0$-classes already recover the localized regions well, as quantified directly in the figure. For three regions the selected tight minimal open coincides with the representative class, while for the other three it adds a small number of observations. Thus a $T_0$-class may itself form a minimal open or may sit inside a slightly larger one, and these containing opens may overlap. This additional relation among localized supports is information that a flat partition does not retain.

% ====================================================================
\subsection{Removal, restriction, and recomputation}
\label{subsec:restriction-versus-recomputation}
% ====================================================================

The B-stage analysis above shows that localized regions can appear as equality fibres and as slightly larger incidence neighbourhoods. It does not determine whether those regions can also be exposed as genuine connected pieces once bridge-like observations are removed. For that purpose we return to the finer A-stage and use uncovered history as a target-blind removal schedule. Let
$$
\tau(x)=\min\{j\colon x\text{ is uncovered at A-level }j\},\qquad R_i=\{x\colon \tau(x)\leq i\},\qquad Y_i=X\setminus R_i,
$$
with $\tau(x)=\infty$ if $x$ is never uncovered. For each distinct nonempty removal set with $2\leq i\leq25$, restriction uses the subspace filtration
$$
\mathbb I_i={(Y_i,\mathcal T_j|{Y_i})}{j=1}^{30},\qquad U_j^{Y_i}(x)=U_j^X(x)\cap Y_i,
$$
whereas recomputation, denoted $\mathbb Q_i$, rebuilds the A-stage on the restricted distance matrix with the same configuration. Restriction asks which pieces were already implicit after bridge deletion; recomputation asks how the retained observations rebuild their neighbourhoods.

\begin{proposition}[Restriction refines component traces]
\label{prop:restriction-refines-traces}
For every cutoff $i$, level $j$, and $x\in Y_i$, the restricted component $I_{i,j}(x)$ is contained in the trace $O_{i,j}(x)=\operatorname{Comp}_{(X,\mathcal T_j)}(x)\cap Y_i$. Equality holds for a trace exactly when that trace remains connected after deletion.
\end{proposition}

\begin{proof}
The inclusion $Y_i\hookrightarrow X$ is continuous. Hence a connected subset of the subspace $Y_i$ lies in one connected component of $X$.
\end{proof}

Components are selected independently in every restricted and recomputed filtration. At each A-level we retain $25\leq|C|<|Y_i|$, collapse repeated supports, and link successive components when their Jaccard overlap is at least $.80$. A component corridor must occupy at least three A-levels and is represented by the same weighted-medoid rule used in \Cref{subsec:mixed-bumps-clustering}. For each seed, configuration, and construction separately, these medoids are then pooled over the removal cutoffs and duplicate supports are removed. This produces one target-blind pool of restricted supports and one target-blind pool of recomputed supports. Only after these pools have been fixed are six supports from each pool assigned one-to-one to the six reference regions by maximum total Jaccard. Thus each configuration and construction gives six assigned supports for each of the three seeds, hence $3\times6=18$ seed–region assignments.

Across the three seeds, this procedure compresses 25,162 qualifying corridor nodes to 7,012 distinct medoid supports, a $72.1\%$ reduction. Each of the eight configurations yields six posthoc assignments for each seed under restriction and another six under recomputation. The complete primary comparison therefore contains $3$ seeds $\times\,8$ configurations $\times\,2$ constructions $\times\,6$ regions $=288$ assigned supports. Across these assignments, the corridor medoids have mean $J_{95}=.834$ and median $.847$. As a check on the loss caused by representing each corridor by a single medoid, allowing all component supports occurring along the retained corridors raises the mean only to $.846$. \Cref{tab:mixed-bumps-component-corridors} gives the primary results separately by configuration and construction, with each row summarizing the $18$ seed–region assignments.

\begin{table}[!ht]
\centering
\small
\caption{Assigned component-corridor representatives after target-blind removal, restriction, or recomputation. Each row summarizes the $18=3$ seeds $\times\,6$ planted regions assigned for the indicated configuration and construction. Corridors and candidate-support pools are formed separately for restriction and recomputation and independently for each seed; the planted references enter only in the final six-region assignment.}
\label{tab:mixed-bumps-component-corridors}
\setlength{\tabcolsep}{5.0pt}
\begin{tabular}{@{}llrrr@{}}
\toprule
configuration & construction & mean $J_{95}$ & median $J_{95}$ & $J_{95}\geq .70$ \\
\midrule
profile raw $s/\delta$ & restricted & 0.869 & 0.868 & 18 \\
profile raw $s/\delta$ & recomputed & 0.869 & 0.868 & 18 \\
profile coherent $s/\delta$ & restricted & 0.879 & 0.877 & 18 \\
profile coherent $s/\delta$ & recomputed & 0.869 & 0.877 & 17 \\
adaptive raw $s/\delta$ & restricted & 0.855 & 0.837 & 18 \\
adaptive raw $s/\delta$ & recomputed & 0.855 & 0.837 & 18 \\
adaptive coherent $s/\delta$ & restricted & 0.826 & 0.838 & 15 \\
adaptive coherent $s/\delta$ & recomputed & 0.841 & 0.853 & 17 \\
profile raw $1/\delta$ & restricted & 0.813 & 0.834 & 17 \\
profile raw $1/\delta$ & recomputed & 0.867 & 0.896 & 18 \\
profile coherent $1/\delta$ & restricted & 0.756 & 0.862 & 13 \\
profile coherent $1/\delta$ & recomputed & 0.799 & 0.821 & 15 \\
adaptive raw $1/\delta$ & restricted & 0.787 & 0.821 & 16 \\
adaptive raw $1/\delta$ & recomputed & 0.875 & 0.888 & 18 \\
adaptive coherent $1/\delta$ & restricted & 0.808 & 0.821 & 16 \\
adaptive coherent $1/\delta$ & recomputed & 0.771 & 0.825 & 12 \\
\bottomrule
\end{tabular}
\end{table}

Thus the summaries combine the six region assignments from each of the three independently generated samples, while the target-blind support pools themselves are constructed separately within each seed. The reported $J_{95}$ values enter only after those pools have been fixed. The small difference between the medoid mean $.834$ and the all-corridor-support mean $.846$ also shows that replacing each corridor by its weighted medoid retains most of the available region-level agreement.

Recovery quality and construction agreement are different questions. For each fixed seed, configuration, and planted region, we therefore pair the support assigned from the restricted pool with the support assigned from the recomputed pool and compute their Jaccard overlap. Thus each configuration contributes $3\times6=18$ restricted–recomputed support pairs. \Cref{tab:mixed-bumps-restrict-recompute-agreement} summarizes these paired comparisons. Exact agreement in the raw $s/\delta$ modes means that recomputation returns the same subsets already exposed by deletion; lower agreement, particularly in coherent inverse-distance modes, means that rebuilding the neighbourhood system changes the recovered support even when both constructions still match the planted region well.

\begin{table}[!ht]
\centering
\small
\caption{Agreement between restricted and recomputed component supports assigned to the same planted region. Each row summarizes $18=3$ seeds $\times,6$ regions paired within the indicated configuration. The mean and minimum columns report the Jaccard overlap between the two assigned supports themselves and therefore measure construction agreement, not recovery against the planted reference.}
\label{tab:mixed-bumps-restrict-recompute-agreement}
\setlength{\tabcolsep}{5.0pt}
\begin{tabular}{@{}lrrr@{}}
\toprule
configuration & \shortstack{exact\\supports} & \shortstack{mean support\\Jaccard} & \shortstack{minimum support\\Jaccard} \\
\midrule
profile raw $s/\delta$ & 18 & 1.000 & 1.000 \\
profile coherent $s/\delta$ & 13 & 0.974 & 0.724 \\
adaptive raw $s/\delta$ & 18 & 1.000 & 1.000 \\
adaptive coherent $s/\delta$ & 13 & 0.973 & 0.750 \\
profile raw $1/\delta$ & 6 & 0.912 & 0.530 \\
profile coherent $1/\delta$ & 0 & 0.722 & 0.253 \\
adaptive raw $1/\delta$ & 7 & 0.890 & 0.500 \\
adaptive coherent $1/\delta$ & 1 & 0.734 & 0.403 \\
\bottomrule
\end{tabular}
\end{table}

The two tables therefore answer different questions. \Cref{tab:mixed-bumps-component-corridors} asks how well the target-blind component representatives from each construction match the planted regions after assignment, whereas \Cref{tab:mixed-bumps-restrict-recompute-agreement} asks whether restriction and recomputation assign the same observations to a given region. In particular, high recovery under both constructions need not imply high agreement between their supports.

The two mixed-bumps constructions therefore provide complementary readings of the same A-stage removal schedule. Restriction asks whether deleting bridge-like observations exposes connected pieces already latent in the original topology, whereas recomputation asks how the retained sample reorganizes when its neighbourhood system is rebuilt. Their agreement indicates structure that survives this rebuilding, while their disagreement records the extent to which the recovered connected pieces depend on recomputing neighbourhood relations after removal.

% ====================================================================
\subsection{Comparison with ToMATo}
\label{subsec:mixed-bumps-tomato}
% ====================================================================

We next compare the structures found by the FTS clustering analysis with ToMATo, a persistence-based clustering method based on density modes \cite{CGOS13}. The comparison has two stages. We first compare the representative $T_0$-classes selected above with localized regions obtained from ToMATo. We then use the minimal opens around the selected $T_0$-class corridors to construct larger FTS supports and compare these with the full ToMATo clusters. In both stages, the planted regions are used only for posthoc evaluation.

\subsubsection*{Representative $T_0$-classes and localized ToMATo regions}

For each of the three realizations we run ToMATo at the predeclared neighbourhood scales $k\in\{10,20,30\}$. We use a symmetrized $k$-nearest-neighbour graph and the \texttt{logDTM} density estimator, with the same $k$ used for the DTM neighbourhood parameter. ToMATo identifies density modes and follows the merging of their basins as the density level decreases. If a mode has density $b$ and merges into an older mode at density $d<b$, its prominence is $b-d$. We choose the prominence threshold without using the planted regions: the finite prominences are sorted, the largest consecutive gap is located, and the midpoint of this gap is used as the threshold. The number of surviving ToMATo clusters is therefore determined by this rule rather than fixed at six.

The full ToMATo clusters are larger objects than the representative FTS $T_0$-classes. For the first comparison we therefore localize the ToMATo clusters using their persistence information. Let $d_{\max}$ be the largest finite merge level among the surviving modes. If $B_i$ is a surviving ToMATo cluster, define
$$
B_i^{\mathrm{res}} = \{x\in B_i:\widetilde f(x)>d_{\max}\},
$$
where $\widetilde f$ is the ToMATo density function. Thus the same density cutoff is applied to every surviving cluster. These restricted regions are introduced only for this comparison and are not the ToMATo clustering itself. We can now compare localized structures on the two sides. The FTS corridor procedure selects representative $T_0$-classes, while the sets $B_i^{\mathrm{res}}$ give localized parts of the surviving ToMATo clusters. Neither construction uses the planted regions. Only after these objects have been obtained do we associate them with the six $95\%$ reference regions by a one-to-one assignment maximizing total Jaccard overlap.

\Cref{fig:mixed-bumps-tomato-localized} shows one such comparison, for seed~1. Here ToMATo uses $k=30$, while FTS uses the profile dictionary, raw A-sets, $1/\delta$ score, and $\eta_0=0$. The selected tight minimal opens are also shown, so that the representative $T_0$-classes can be viewed together with the controlled enlargements considered earlier in this section.

\begin{figure}[!ht]
\centering
\includegraphics[width=\textwidth]{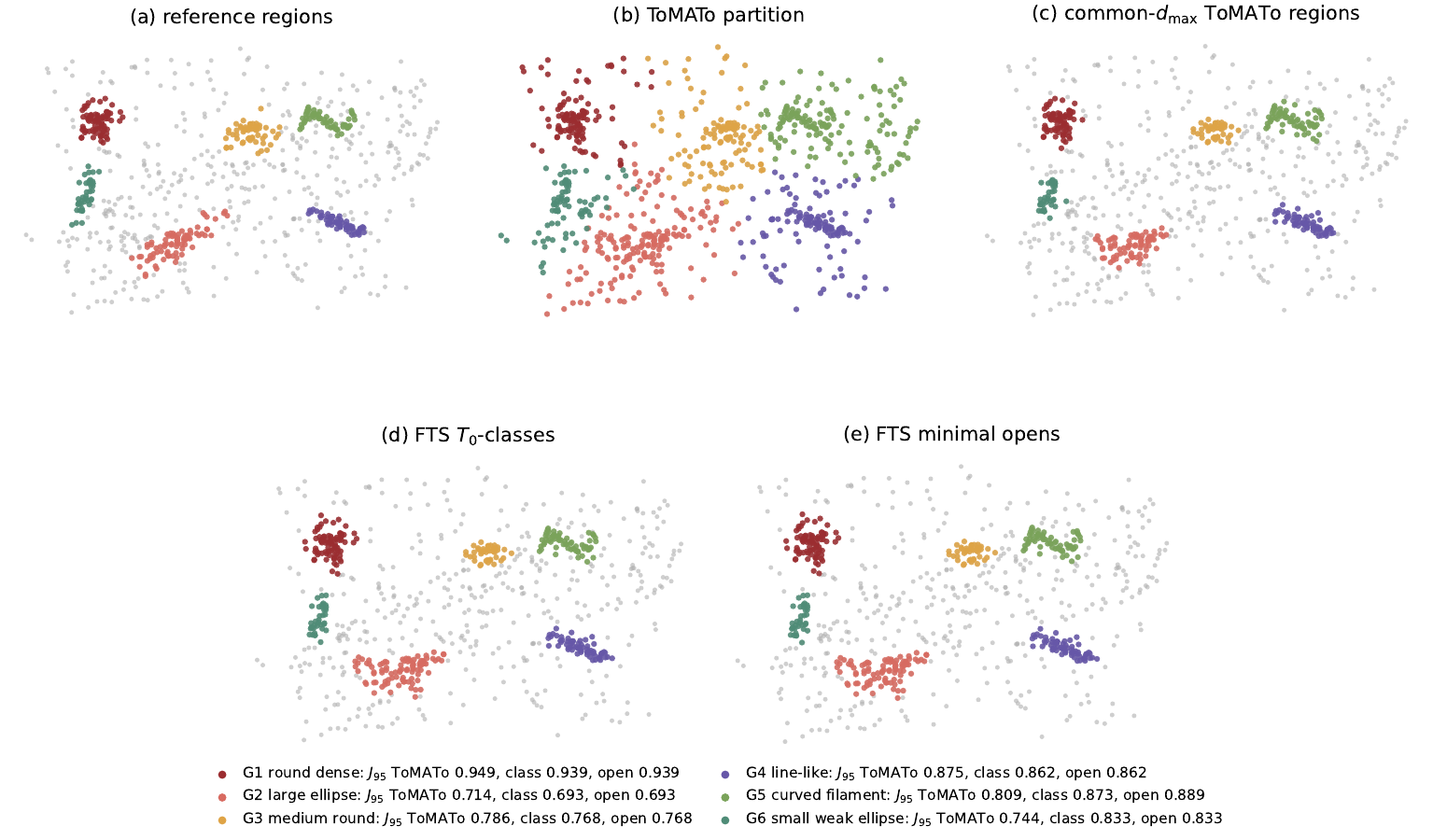}
\caption{Localized FTS--ToMATo comparison on seed~1. From left to right: the $95\%$ reference regions; the ToMATo clustering at $k=30$; the common-$d_{\max}$ persistence-restricted ToMATo regions; the representative FTS $T_0$-classes; and their target-blind selected tight minimal opens for the profile dictionary, raw A-sets, $1/\delta$ score, and $\eta_0=0$. The reference regions enter only in the posthoc association indicated by the common colors; they are not used to construct or select any of the displayed FTS or ToMATo structures.}
\label{fig:mixed-bumps-tomato-localized}
\end{figure}

The displayed example is one case from the full comparison. For each of the three seeds, we evaluate all eight FTS settings and all three predeclared ToMATo scales $k\in\{10,20,30\}$. The results are summarized in \Cref{tab:mixed-bumps-tomato-summary}. The target-blind ToMATo prominence threshold gives six clusters in eight of the nine seed--$k$ runs; for seed~0 at $k=30$, it gives five. No FTS setting or ToMATo scale is selected using the reference regions.

\begin{table}[!ht]
\centering
\footnotesize
\caption{Target-blind summary over the declared comparison. For FTS, the two range columns give, over the eight topology settings for each seed, the ranges of the mean and minimum regionwise $J_{95}$ values of the representative $T_0$-classes. Each ToMATo entry gives mean/minimum $J_{95}$ for the common-$d_{\max}$ restricted regions at the indicated predeclared $k$. Seed~0 at $k=30$ has only five surviving ToMATo clusters. The reference regions are used only for posthoc evaluation, not to select an FTS setting, $k$, or the ToMATo prominence threshold.}
\label{tab:mixed-bumps-tomato-summary}
\setlength{\tabcolsep}{3.8pt}
\begin{tabular}{@{}crrrrr@{}}
\toprule
seed & FTS mean range & FTS min range & ToMATo $k=10$ & ToMATo $k=20$ & ToMATo $k=30$ \\
\midrule
0 & 0.830--0.860 & 0.680--0.766 & 0.826/0.718 & 0.869/0.787 & 0.810/0.579 \\
1 & 0.766--0.837 & 0.646--0.693 & 0.794/0.667 & 0.792/0.673 & 0.813/0.714 \\
2 & 0.843--0.896 & 0.738--0.837 & 0.838/0.703 & 0.855/0.773 & 0.873/0.756 \\
\bottomrule
\end{tabular}
\end{table}

Across the $24$ seed--configuration FTS runs, the representative $T_0$-classes retain a nonzero match to each of the six reference regions; the smallest regionwise $J_{95}$ is $.646$. In the seed~0, $k=30$ ToMATo run, only five clusters survive the target-blind prominence threshold, so one of the six reference regions has no corresponding cluster.

\subsubsection*{Separating supports and ToMATo clusters}

The first comparison concerns localized structures. It does not yet use all of the minimal-open information surrounding a selected $T_0$-class. The minimal opens also tell us how far such a class can be enlarged while
remaining separated from the other selected classes. We use this information to construct a larger FTS support.

For a selected corridor $i$, let $L_i$ be the last filtration level at which its $T_0$-class remains distinct from the classes of other selected corridors, and let $C_i=C_i(L_i)$ be its $T_0$-class at that level. Thus, after level $L_i$, the class associated with this corridor merges with the class of at least one of the other selected corridors. At level
$L_i$, consider the minimal opens containing $C_i$ that avoid all the other selected classes:
$$
\mathcal U_i^{\mathrm{sep}} = \left\{U\colon C_i\subseteq U,\quad U\cap C_j(L_i)=\varnothing \text{ for every }j\neq i
\right\}.
$$
We define the corresponding \emph{separating support} by
$$
S_i = \bigcup_{U\in\mathcal U_i^{\mathrm{sep}}}U.
$$
Thus $S_i$ consists of the points reached from $C_i$ by minimal opens that still avoid the competing selected classes. Neither the planted regions nor the ToMATo clusters are used in this construction. The levels $L_i$ need not be the same for the six selected corridors. Consequently the supports $S_i$ are not required to form a partition: a point may belong to one selected support, to none of them, or to more than one. We compare these supports with the full ToMATo clusters, rather than with the restricted regions used in the first comparison.

\Cref{fig:mixed-bumps-tomato-supports} shows the result for the same seed~1 FTS configuration and the ToMATo clustering at $k=30$. The association between the six FTS supports and the six ToMATo clusters is again made only after both constructions are complete.

\begin{figure}[!ht]
\centering
\includegraphics[width=\textwidth]{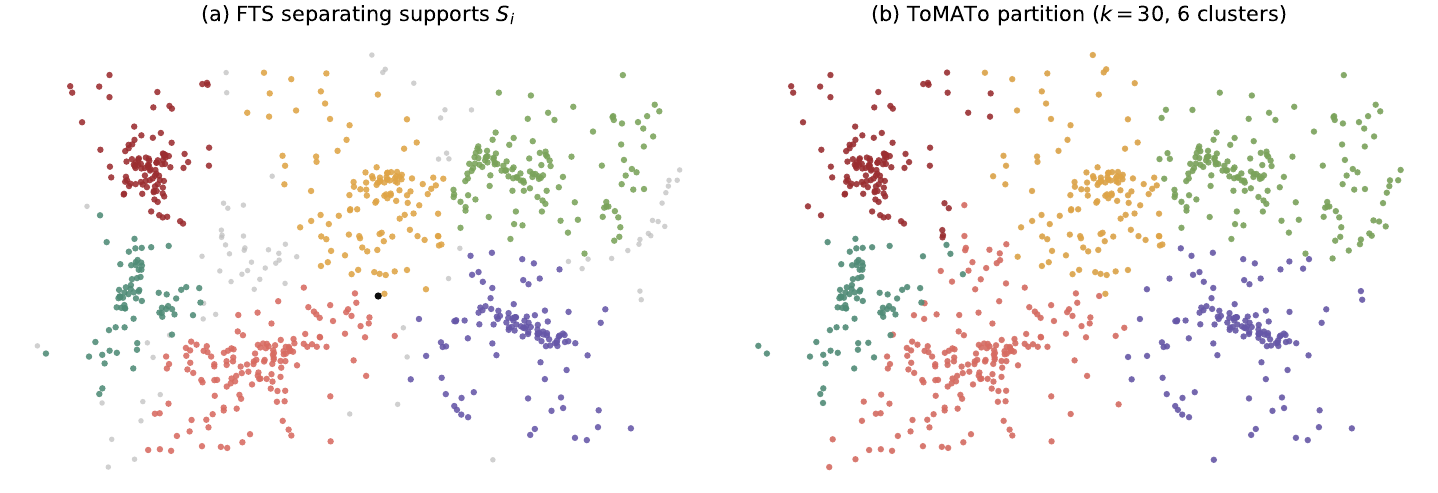}
\caption{Separating supports and ToMATo clusters for seed~1. Left: the six FTS separating supports $S_i$ for the profile dictionary, raw A-sets, $1/\delta$ score, and $\eta_0=0$. Gray points belong to none of the six supports, while the single black point belongs to two supports. Right: the ToMATo clustering at $k=30$. Colors indicate the posthoc one-to-one association between the FTS supports and ToMATo clusters.}
\label{fig:mixed-bumps-tomato-supports}
\end{figure}

The corresponding support--cluster comparisons are given in \Cref{tab:fts-separating-support-tomato-selected}. Besides the Jaccard overlap, we report 
$$
\frac{|S_i\cap B_i|}{|S_i|}\qquad\text{and}\qquad\frac{|S_i\cap B_i|}{|B_i|},
$$
where $B_i$ is the matched ToMATo cluster. The first quantity measures how much of the FTS support lies in the matched ToMATo cluster, while the second measures how much of that cluster is covered by the FTS support.

\begin{table}[!ht]
\centering
\small
\caption{Separating supports and matched ToMATo clusters for the displayed seed-1 comparison: profile dictionary, raw A-sets, $1/\delta$, $\eta_0=0$, and ToMATo $k=30$. The matching is posthoc. The directional columns report the fraction of each FTS support contained in its matched ToMATo cluster, $|S_i\cap B_i|/|S_i|$, and the fraction of that ToMATo cluster covered by the FTS support, $|S_i\cap B_i|/|B_i|$.}
\label{tab:fts-separating-support-tomato-selected}
\setlength{\tabcolsep}{2.3pt}
\begin{tabular}{@{}clrrrrrrr@{}}
\toprule
support & geometry & $L_i$ & $|S_i|$ & $|B_i|$ & $|S_i\cap B_i|$ & $|S_i\cap B_i|/|S_i|$ & $|S_i\cap B_i|/|B_i|$ & Jaccard \\
\midrule
$S_1$ & round dense & 47 & 102 & 114 & 102 & 1.000 & 0.895 & 0.895 \\
$S_2$ & large ellipse & 40 & 150 & 181 & 150 & 1.000 & 0.829 & 0.829 \\
$S_3$ & medium round & 40 & 119 & 119 & 113 & 0.950 & 0.950 & 0.904 \\
$S_4$ & line-like & 45 & 116 & 126 & 116 & 1.000 & 0.921 & 0.921 \\
$S_5$ & curved filament & 40 & 133 & 155 & 133 & 1.000 & 0.858 & 0.858 \\
$S_6$ & small weak ellipse & 40 & 74 & 80 & 71 & 0.959 & 0.887 & 0.855 \\
\multicolumn{8}{r}{\textbf{mean Jaccard}} & \textbf{0.877} \\
\multicolumn{8}{r}{\textbf{minimum Jaccard}} & \textbf{0.829} \\
\bottomrule
\end{tabular}
\end{table} 

The six support--cluster Jaccards have mean $.877$ and minimum $.829$. Four of the six FTS supports are entirely contained in their matched ToMATo clusters; for the other two, $|S_i\cap B_i|/|S_i|$ is $.950$ and $.959$. The union of the six FTS supports contains $693$ of the $775$ observations. Thus $82$ observations are left unassigned, while one observation belongs to two supports. An unassigned observation is not declared noise; it is simply not included in a separating support associated with any of the six selected corridors.

The two comparisons use different parts of the same FTS evolution. Representative $T_0$-classes provide localized structures that can be compared with persistence-restricted ToMATo regions. Following the selected corridors to the last levels at which their classes remain distinct and using the corresponding minimal opens then produces the larger separating supports $S_i$, which can be compared with the full ToMATo clusters. In the displayed realization these supports closely follow the corresponding ToMATo clusters, while they are not required to exhaust the data or to form a partition. Complete results over the three seeds, eight FTS settings, and three ToMATo neighbourhood scales, together with further diagnostics, are provided in the additional materials.

% ====================================================================
\subsection{Structural formation in Paul15}
\label{subsec:paul15}
% ====================================================================

\subsubsection*{Data, annotations, and scope}

Paul15 contains 2,730 mouse bone-marrow progenitor cells, represented here by a PCA-30 distance matrix \cite{Paul2015}. We analyze the full data and three label-stratified samples of 1,000 cells, using the same eight configurations as for \texttt{mixed\_bumps}. The annotations $C1,\ldots,C19$ are not used to construct the filtrations. They are introduced afterward to interpret structures already obtained from the distance data.

The controlled example showed that an FTS filtration can be read at several levels. Connected components describe coarse separation, $T_0$-classes collect observations with the same minimal open, and the minimal opens themselves describe neighbourhood relations that can overlap. B-completion then shows how the finer A-stage structure is assembled as the topology coarsens. On Paul15 we use these objects to study cellular organization without treating the 19 annotations as a target partition. Core reduction provides a compressed view of the same finite topology, while observation removal tests what remains after cells that become uncovered early are deleted and whether rebuilding the neighbourhood system changes the resulting structure.

Throughout this subsection, filtration level is a resolution parameter rather than developmental time. In particular, the merge trees below describe how $T_0$-classes combine across the filtration; they are not inferred lineage trees.

\subsubsection*{Fine $T_0$-classes and B-stage assembly}

The difference between connected components and $T_0$-classes is already substantial at the end of the A-stage. On the full data, the four $s/\delta$ filtrations have only $1$–$5$ connected components at $A30$, while the four inverse-distance filtrations have $13$–$22$. At the same level, however, the eight configurations have $1{,}727$–$2{,}419$ distinct $T_0$-classes, and their largest $T_0$-class contains only $6$–$30$ cells. Thus the data can be almost connected while still having thousands of distinct minimal-open types.

The first B-level changes this picture sharply. \Cref{tab:paul15-b31} compares the $T_0$-classes immediately before and after the $A30$–$B31$ boundary. For the largest $B31$ class, the last two columns record how many $A30$ classes contribute to it and the fraction contributed by its largest predecessor.

\begin{table}[!ht]
\centering
\footnotesize
\caption{Full-data Paul15 $T_0$ structure across the $A30$--$B31$ boundary. The final two columns decompose the largest $B31$ class into its $A30$ predecessors: the predecessor count gives the number of contributing $A30$ classes, and the final column gives the fraction of the $B31$ class contributed by its largest predecessor.}
\label{tab:paul15-b31}
\setlength{\tabcolsep}{2.8pt}
\begin{tabular}{@{}lrrrrrr@{}}
\toprule
configuration & \shortstack{$\#T_0$-classes\\$A30$} & \shortstack{largest\\$A30$} & \shortstack{$\#T_0$-classes\\$B31$} & \shortstack{largest\\$B31$} & \shortstack{$A30$\\predecessors} & \shortstack{largest pred.\\share} \\
\midrule
profile raw $s/\delta$ & 1948 & 30 & 1374 & 574 & 502 & 0.014 \\
profile coherent $s/\delta$ & 1753 & 13 & 1085 & 542 & 451 & 0.020 \\
adaptive raw $s/\delta$ & 1958 & 30 & 1384 & 574 & 502 & 0.014 \\
adaptive coherent $s/\delta$ & 1727 & 20 & 1059 & 542 & 451 & 0.020 \\
profile raw $1/\delta$ & 2134 & 11 & 1769 & 372 & 331 & 0.016 \\
profile coherent $1/\delta$ & 2419 & 6 & 2243 & 131 & 124 & 0.015 \\
adaptive raw $1/\delta$ & 2124 & 11 & 1765 & 366 & 324 & 0.016 \\
adaptive coherent $1/\delta$ & 2400 & 7 & 2198 & 141 & 135 & 0.014 \\
\bottomrule
\end{tabular}
\end{table}

The large $B31$ classes are assembled from many small $A30$ classes rather than obtained by enlarging one already dominant class. For profile raw $s/\delta$, for example, the largest class has $574$ cells at $B31$ and receives cells from $502$ different $A30$ classes; no single predecessor contributes more than $.014$ of the resulting class. The same many-to-one assembly occurs in all eight configurations, although its scale differs. The $s/\delta$ modes already have largest $B31$ classes of $542$–$574$ cells, whereas the coherent inverse-distance modes have largest classes of only $131$–$141$ cells. B-completion therefore combines many fine $T_0$-classes into larger classes, but the rate of this assembly depends on the construction. The two full-data merge trees in \Cref{fig:paul15-tree-raw-sd,fig:paul15-tree-coh-invd} show two contrasting versions of this process.

\begin{figure}[!ht]
\centering
\includegraphics[width=\textwidth]{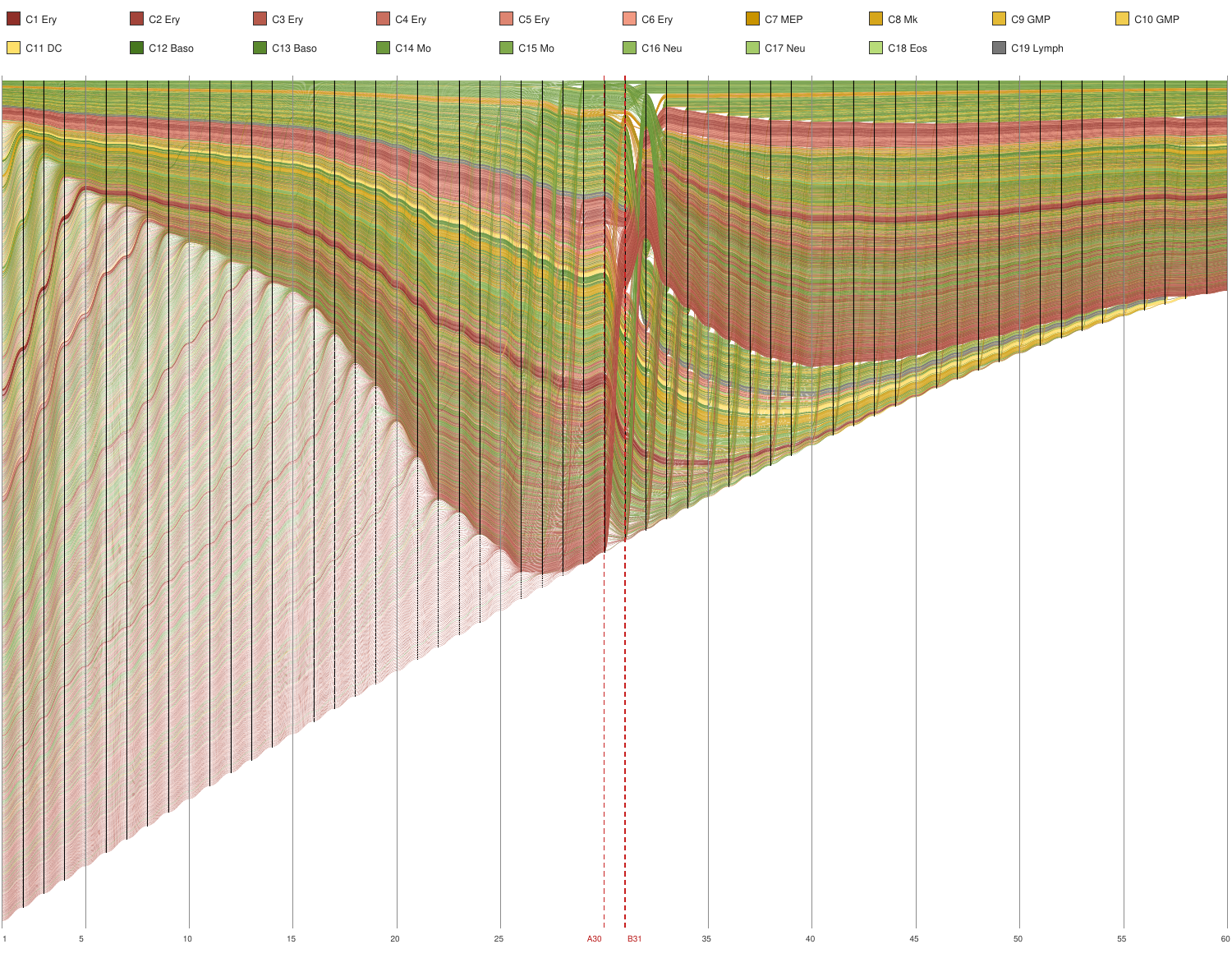}
\caption{Full-data Paul15 $T_0$-class merge tree for the profile raw $s/\delta$ configuration with $\eta_0=0$. The legend above the tree gives the Paul15 annotation colors; mixed bands retain the contributing colors. The red guides mark $A30$ and $B31$. A large-scale assembly occurs immediately across the A/B boundary.}
\label{fig:paul15-tree-raw-sd}
\end{figure}

\begin{figure}[!ht]
\centering
\includegraphics[width=\textwidth]{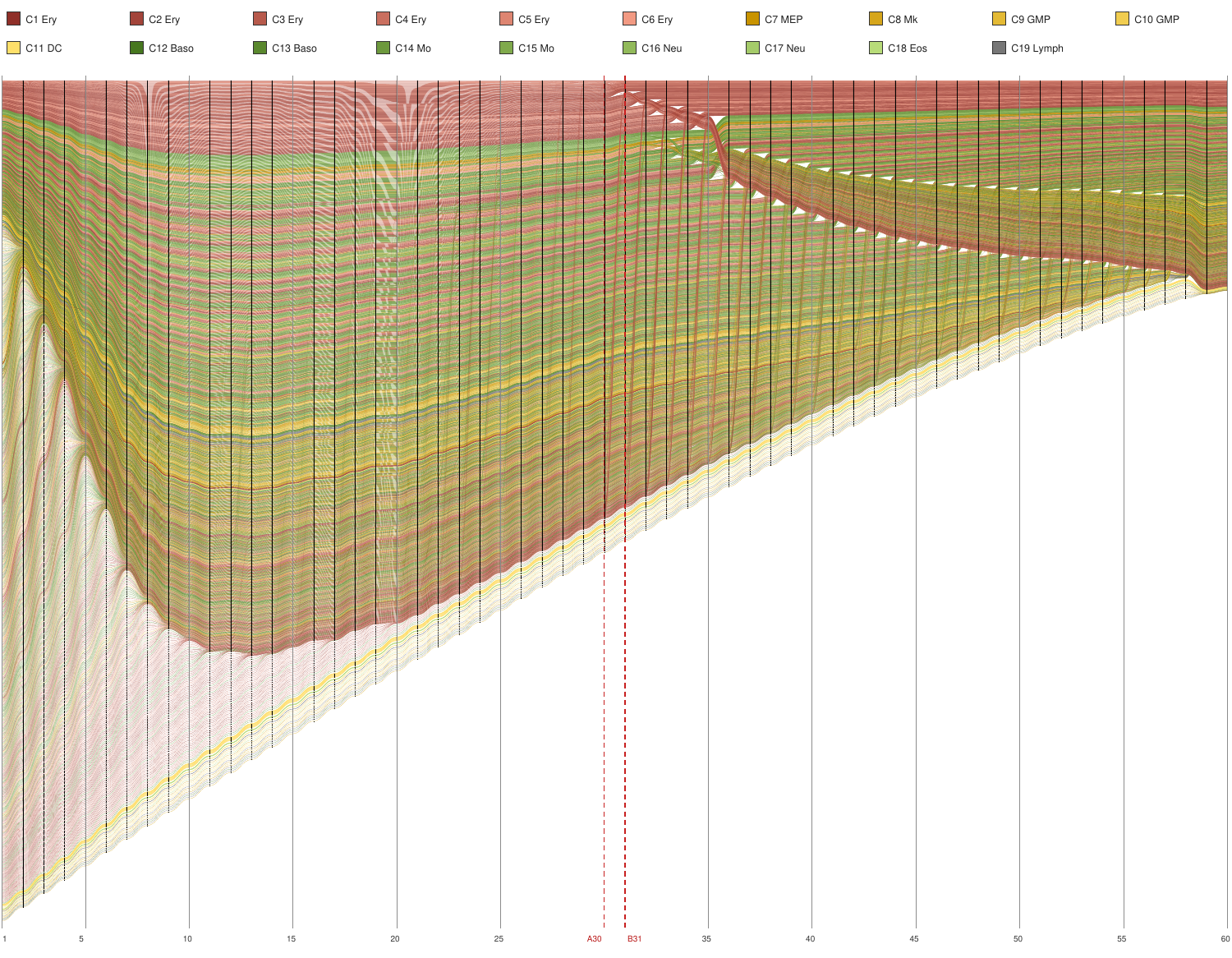}
\caption{Full-data Paul15 $T_0$-class merge tree for the profile coherent $1/\delta$ configuration with $\eta_0=0$. The same annotation colors and $A30$/$B31$ guides are used. Compared with \Cref{fig:paul15-tree-raw-sd}, the B-stage begins with smaller leading classes and consolidation proceeds more gradually.}
\label{fig:paul15-tree-coh-invd}
\end{figure}

The annotation colors also make clear that the $T_0$-classes are not simply reproducing the 19 annotated populations. Small classes can lie largely within one annotation, while larger B-stage classes combine cells carrying several annotations. The merge trees therefore show both the fine structure present near the end of the A-stage and how that structure is assembled at coarser levels.

\subsubsection*{Minimal opens and annotated relations}

A $T_0$-class tells us which cells have the same minimal open, but it does not by itself describe the full neighbourhood represented by that open. The minimal open may contain several $T_0$-classes and can therefore record relations that are not visible from the classes alone. We use the annotations posthoc to examine several such relations.

Let $X_a$ be the set of cells carrying annotation $C_a$. For a minimal-open support $S$, define
$$
\operatorname{cov}_a(S)=\frac{|S\cap X_a|}{|X_a|}.
$$
For a pair $(a,b)$, let
$$
m_{a,b}(S)=\min\{\operatorname{cov}_a(S),\operatorname{cov}_b(S)\}
$$
and
$$
q_{a,b}(S)=\frac{|S\cap X_a|+|S\cap X_b|}{|S|}.
$$
The first quantity requires the open to cover substantial portions of both annotated populations, while the second measures how much of the open belongs to those two populations. These requirements are complementary: a very large open can have high coverage but low concentration, while a small open can be highly concentrated without containing much of either population. We combine them as
$$
\rho_{a,b}(S)=\sqrt{m_{a,b}(S)q_{a,b}(S)}.
$$
The score $\rho$ is used only for posthoc interpretation. It does not enter the filtration or the selection of minimal opens.

For each pair and each of the $32=8$ configurations $\times\,4$ data sets, we evaluate $\rho$ over all minimal opens in the complete A+B filtration and retain the maximizing open. \Cref{tab:paul15-original-relations} summarizes the resulting scores.

\begin{table}[!ht]
\centering
\small
\caption{Posthoc annotated relations carried by minimal opens in the original Paul15 filtrations. For each configuration--data combination and label pair, the score $\rho$ is evaluated on every minimal open over the complete A+B filtration and the single open attaining the largest score is retained. The table reports medians over the $32=8$ configurations $\times\,4$ data sets. Here $m$ is the smaller of the two population coverages, $q$ is the fraction of the support carrying either annotation, and $\rho=\sqrt{mq}$ balances simultaneous coverage with pair concentration.}
\label{tab:paul15-original-relations}
\setlength{\tabcolsep}{4.5pt}
\begin{tabular}{@{}lrrrr@{}}
\toprule
label pair & median $\rho$ & median $m$ & median $q$ & median $|S|$ \\
\midrule
$C7$--$C8$ & 0.483 & 0.545 & 0.447 & 132.0 \\
$C13$--$C14$ & 0.682 & 0.601 & 0.770 & 211.5 \\
$C14$--$C15$ & 0.598 & 0.664 & 0.590 & 260.0 \\
$C15$--$C16$ & 0.638 & 0.531 & 0.808 & 106.0 \\
$C12$--$C13$ & 0.373 & 0.984 & 0.141 & 946.0 \\
\bottomrule
\end{tabular}
\end{table}

The separate roles of $m$ and $q$ are visible in the table. For $C13$–$C14$, for example, the median maximizing open has $m=.601$ and $q=.770$, so substantial portions of both populations occur in a support that remains comparatively concentrated in the pair. The $C15$-$C16$ relation is similarly concentrated, with $q=.808$. In contrast, $C12$–$C13$ has very high simultaneous coverage, $m=.984$, but this is achieved in a much broader support: its median size is $946$ cells and its pair concentration is only $.141$. The smaller value $\rho=.373$ therefore distinguishes occurrence together in a broad neighbourhood from a more localized relation.

The maximizing opens need not occur in the same part of the filtration. \Cref{tab:paul15-original-relation-locations} records whether each optimum occurs in the A- or B-stage and the corresponding median level.

\begin{table}[!ht]
\centering
\small
\caption{Location of the score-maximizing minimal opens used in \Cref{tab:paul15-original-relations}. Each label pair has one maximizing open in each of the $32=8$ configurations $\times\,4$ data sets. The A and B columns count how often that maximum occurs in the corresponding stage and give the median maximizing level conditional on that stage. Thus this table records where the reported original-filtration relation scores are attained; it does not introduce an additional selection.}
\label{tab:paul15-original-relation-locations}
\setlength{\tabcolsep}{4.5pt}
\begin{tabular}{@{}lrrrr@{}}
\toprule
label pair & best in A & median A level & best in B & median B level \\
\midrule
$C7$--$C8$ & 31 & 18.0 & 1 & 32.0 \\
$C13$--$C14$ & 23 & 26.0 & 9 & 38.0 \\
$C14$--$C15$ & 13 & 27.0 & 19 & 34.0 \\
$C15$--$C16$ & 28 & 22.0 & 4 & 32.0 \\
$C12$--$C13$ & 14 & 24.5 & 18 & 58.0 \\
\bottomrule
\end{tabular}
\end{table}

Some relations are predominantly visible during the A-stage. The maximizing open for $C7$–$C8$ occurs in A in $31$ of the $32$ cases, and that for $C15$–$C16$ in $28$. The $C14$–$C15$ relation behaves differently: its maximum occurs in B in $19$ of the $32$ cases. The $C12$–$C13$ maxima are divided more evenly between the two stages, and the B-stage maxima occur relatively late. Thus different annotated relations are most clearly represented at different resolutions of the filtration. This is one reason to retain the evolving topology rather than reduce the analysis to a single partition or a single selected level.

These relations should be interpreted geometrically. The annotations identify which cell populations occur in minimal opens already constructed from the PCA-30 distances; they do not determine those opens. Likewise, a strong value of $\rho$ records a localized finite-topological neighbourhood relation, not by itself a developmental transition.

\subsubsection*{Core reduction}
\label{subsec:paul15-core}

The large numbers of $T_0$-classes raise a natural question: how much of the corresponding poset is topologically redundant? We address this by applying the deterministic levelwise core reduction of \Cref{subsec:core-reduction}. No cells are removed. When a quotient vertex is eliminated, its original cells are transferred to a surviving core vertex, so the reduced poset can still be interpreted on the original data. \Cref{tab:paul15-core-compression} gives the fraction of $T_0$-poset vertices eliminated at four representative levels.

\begin{table}[!ht]
\centering
\small
\caption{Fraction of full-data Paul15 $T_0$-poset vertices eliminated by deterministic levelwise core reduction at selected levels. The entries measure topological compression; no observations are deleted from the data, since eliminated quotient vertices are transferred to surviving core vertices.}
\label{tab:paul15-core-compression}
\begin{tabular}{@{}lrrrr@{}}
\toprule
configuration & $A30$ & $B31$ & $B33$ & $B45$ \\
\midrule
profile raw $s/\delta$ & 0.393 & 0.528 & 0.802 & 0.968 \\
profile coherent $s/\delta$ & 0.339 & 0.676 & 0.837 & 0.968 \\
adaptive raw $s/\delta$ & 0.404 & 0.543 & 0.834 & 0.982 \\
adaptive coherent $s/\delta$ & 0.328 & 0.669 & 0.823 & 0.939 \\
profile raw $1/\delta$ & 0.326 & 0.460 & 0.600 & 0.930 \\
profile coherent $1/\delta$ & 0.384 & 0.497 & 0.620 & 0.928 \\
adaptive raw $1/\delta$ & 0.347 & 0.496 & 0.682 & 0.923 \\
adaptive coherent $1/\delta$ & 0.375 & 0.465 & 0.670 & 0.935 \\
\bottomrule
\end{tabular}
\end{table}

Substantial compression is already possible at $A30$, where the reduction eliminates about $.33$–$.40$ of the vertices. The effect becomes much stronger during the B-stage. By $B45$, more than $.92$ of the vertices are eliminated in every configuration. The increasingly coarse topology therefore contains considerable poset redundancy.

The compression percentages alone do not show what the reduced topology looks like. \Cref{tab:paul15-core-structure} summarizes the levelwise cores in another way. The transferred core-point blocks partition the original cells, while transferred core opens give the corresponding reduced minimal-open supports. Annotation purity is included only as a posthoc description of the transferred point blocks.

\begin{table}[!ht]
\centering
\small
\caption{Levelwise structure of the full-data Paul15 cores at selected levels, summarized by medians over the eight configurations. Core vertices are computed independently at each level. Compression is the fraction of $T_0$-poset vertices eliminated by beat-point reduction. Core-point purity is the size-weighted posthoc purity of the transferred core-point blocks with respect to the Paul15 annotations. The final two columns give the largest transferred core-point block and largest transferred core-open support as fractions of the 2,730 cells. Thus the table describes what remains visible after levelwise core compression without introducing cross-level core tracks.}
\label{tab:paul15-core-structure}
\setlength{\tabcolsep}{3.8pt}
\begin{tabular}{@{}lrrrrr@{}}
\toprule
level & \shortstack{median\\core vertices} & \shortstack{median\\compression} & \shortstack{core-point\\purity} & \shortstack{largest point\\fraction} & \shortstack{largest open\\fraction} \\
\midrule
$A30$ & 1284 & 0.361 & 0.785 & 0.014 & 0.466 \\
$B31$ & 769 & 0.513 & 0.639 & 0.243 & 0.612 \\
$B33$ & 327 & 0.742 & 0.376 & 0.618 & 0.752 \\
$B45$ & 26 & 0.937 & 0.169 & 0.965 & 0.975 \\
\bottomrule
\end{tabular}
\end{table}

At $A30$, the median core over the eight configurations has $1{,}284$ vertices. Its largest transferred point block contains only $.014$ of the cells, and the size-weighted annotation purity of the transferred point blocks is $.785$. The transferred minimal opens can already be much broader: the largest has median size $.466$ of the data. As B-completion proceeds, the core becomes much smaller and its transferred blocks become much larger. At $B33$, the median core has $327$ vertices and its largest transferred point block contains $.618$ of the cells; by $B45$, the corresponding figures are $26$ vertices and $.965$. The largest transferred core open then contains $.975$ of the data.

The core therefore gives a compressed view of the same local-to-global change seen in the $T_0$-class assembly. Near the end of the A-stage, many small transferred blocks remain, while the B-stage progressively concentrates the data into a much smaller reduced poset with large transferred supports. The simultaneous decline in posthoc point-block purity, from $.785$ at $A30$ to $.169$ at $B45$, is consistent with the same change: the coarser transferred blocks increasingly combine cells carrying different annotations. Since the core is computed independently at each level, we use it here only as a levelwise reduction and do not assign cross-level identities to its vertices.

\subsubsection*{Removal, restriction, and recomputation}
\label{subsec:paul15-rerun}

Core reduction leaves the observations unchanged. We next ask a different question: what happens when cells that become uncovered early in the A-stage are removed? Let
$$
\tau(x)=\min\{i\colon x\text{ is uncovered at A-level }i\},
$$
with $\tau(x)=\infty$ if $x$ is never uncovered, and define
$$
R_c=\{x\colon\tau(x)\leq c\},
\qquad
Y_c=X\setminus R_c.
$$
For a fixed cutoff $c$, restriction keeps the original neighbourhood structure and restricts it to $Y_c$, whereas recomputation rebuilds the A+B filtration on the same retained cells. This separates structure exposed by deletion from structure that depends on rebuilding the neighbourhood system.

Unlike \texttt{mixed\_bumps}, Paul15 has no prescribed geometric regions against which the resulting partitions can be matched. We therefore compare restriction and recomputation directly on the same retained cells using normalized mutual information. For partitions $\mathcal P$ and $\mathcal Q$,
$$
\operatorname{NMI}(\mathcal P,\mathcal Q)
=
\frac{2I(\mathcal P;\mathcal Q)}
{H(\mathcal P)+H(\mathcal Q)}.
$$
Here NMI is only an agreement measure. A-component NMI compares the connected-component partitions under restriction and recomputation, while A-$T_0$ NMI compares their $T_0$-class partitions. At the first B-level we separately compare the $T_0$ partitions and the size of their largest class. \Cref{tab:paul15-removal-modes} reports these quantities at cutoff $25$.

\begin{table}[!ht]
\centering
\small
\caption{Paul15 restriction--recomputation comparison at removal cutoff 25. The removed fraction is measured on the full data. A-component NMI and A-$T_0$ NMI are medians of the NMI between the restricted and recomputed partitions over all A-levels and over the full data and three 1,000-cell samples. $B31$ $T_0$ NMI compares the restricted and recomputed $T_0$-partitions specifically at the first B-level, again taking the median over the four data sets. The final column gives the median largest-$T_0$-class fraction at $B31$ under restriction/recomputation ($I/Q$).}
\label{tab:paul15-removal-modes}
\setlength{\tabcolsep}{3.2pt}
\begin{tabular}{@{}lrrrrr@{}}
\toprule
configuration & \shortstack{removed\\fraction} & \shortstack{A-component\\NMI} & \shortstack{A-$T_0$\\NMI} & \shortstack{$B31$ $T_0$\\NMI} & \shortstack{largest $B31$\\$I/Q$} \\
\midrule
profile raw $s/\delta$ & 0.325 & 0.917 & 1.000 & 0.931 & 0.313/0.270 \\
profile coherent $s/\delta$ & 0.334 & 0.541 & 1.000 & 0.881 & 0.300/0.258 \\
adaptive raw $s/\delta$ & 0.325 & 0.935 & 1.000 & 0.931 & 0.312/0.269 \\
adaptive coherent $s/\delta$ & 0.335 & 0.537 & 1.000 & 0.878 & 0.300/0.255 \\
profile raw $1/\delta$ & 0.560 & 0.723 & 0.997 & 0.926 & 0.297/0.208 \\
profile coherent $1/\delta$ & 0.699 & 0.389 & 0.990 & 0.921 & 0.117/0.024 \\
adaptive raw $1/\delta$ & 0.571 & 0.713 & 0.997 & 0.920 & 0.298/0.201 \\
adaptive coherent $1/\delta$ & 0.633 & 0.541 & 0.994 & 0.933 & 0.098/0.028 \\
\bottomrule
\end{tabular}
\end{table}

A common cutoff does not remove the same fraction of cells in every construction. On the full data, cutoff $25$ removes about $.325$–$.335$ in the $s/\delta$ modes, $.560$–$.571$ in the raw inverse-distance modes, and $.633$–$.699$ in the coherent inverse-distance modes. The rows should therefore not be interpreted as a comparison at equal deletion rates. Rather, each row compares restriction with recomputation after applying the same filtration-defined cutoff within that construction.

The A-stage comparison separates two aspects of the topology particularly clearly. The $T_0$-class partitions are almost unchanged by recomputation: their median NMI lies between $.990$ and $1.000$ in all eight configurations. Connected components can be much more sensitive. A-component NMI ranges from $.935$ for adaptive raw $s/\delta$ to $.389$ for profile coherent $1/\delta$. Thus recomputation can leave the fine $T_0$-classes almost unchanged while altering how those classes are connected.

The first B-level shows a related effect at a coarser scale. The restricted and recomputed $B31$ $T_0$ partitions still have high NMI, between $.878$ and $.933$, but their largest classes can differ substantially. In the $s/\delta$ modes, the largest restricted class contains about $.30$–$.31$ of the retained cells, compared with about $.26$–$.27$ after recomputation. The difference is much larger for the coherent inverse-distance modes: the largest-class fractions change from $.117$ to $.024$ and from $.098$ to $.028$. Hence strong overall agreement between the two $T_0$ partitions does not imply that their largest assembled class has the same scale. Rebuilding the neighbourhood system can preserve most $T_0$-class relations while changing the formation of a dominant B-stage class.

Finally, we return to the annotated minimal-open relations, now after removal. The maximization is performed separately in the A- and B-stages so that relations visible at different resolutions are not pooled.

\begin{table}[!ht]
\centering
\small
\caption{Posthoc annotated minimal-open relations after removal cutoff 25. For each configuration--data combination, construction, and stage separately, $\rho$ is evaluated on every minimal open in that stage and the maximizing open is retained; the entries are medians of those maxima over the $32=8$ configurations $\times\,4$ data sets. Hence A and B are never pooled in the maximization or in the reported medians. $I$ denotes restriction and $Q$ recomputation.}
\label{tab:paul15-relations-after-removal}
\setlength{\tabcolsep}{4.5pt}
\begin{tabular}{@{}lrrrr@{}}
\toprule
label pair & A-stage $I$ & A-stage $Q$ & B-stage $I$ & B-stage $Q$ \\
\midrule
$C7$--$C8$ & 0.599 & 0.600 & 0.598 & 0.626 \\
$C13$--$C14$ & 0.680 & 0.716 & 0.740 & 0.759 \\
$C14$--$C15$ & 0.551 & 0.592 & 0.700 & 0.711 \\
$C15$--$C16$ & 0.622 & 0.650 & 0.551 & 0.563 \\
$C12$--$C13$ & 0.258 & 0.265 & 0.323 & 0.363 \\
\bottomrule
\end{tabular}
\end{table}

Restriction and recomputation give broadly similar relation scores within each stage, but the A/B distinction remains informative. For $C14$–$C15$, the median score is larger in B than in A under both restriction and recomputation, whereas $C15$–$C16$ is more strongly represented in A. The $C7$–$C8$ relation is also stronger after removal than in the original-filtration summary, while $C12$–$C13$ remains comparatively weak. Thus deleting early-uncovered cells does not simply erase the neighbourhood relations seen in the original filtration. Several remain visible, and their strongest representation can still depend on whether one reads the A- or B-stage.

Taken together, the Paul15 results show several forms of organization that coexist in the same FTS filtration. Near the end of the A-stage, the data can be nearly connected while retaining thousands of $T_0$-classes. Minimal opens relate these classes through overlapping neighbourhoods, and different annotated relations are most clearly represented at different filtration levels. B-completion assembles many fine $T_0$-classes into larger classes, with different configurations producing more abrupt or more gradual assembly. Levelwise core reduction compresses much of the resulting poset while retaining the same progression from many small blocks to a few large supports. After early-uncovered cells are removed, recomputation leaves the A-stage $T_0$-classes almost unchanged but can alter their connected components and the scale of the first B-stage assembly. These are complementary descriptions of the geometry represented by the evolving finite topology; the annotations make them interpretable but do not define them. The next section turns from this organization by components, $T_0$-classes, and minimal opens to incidence and cyclic structure.

% ====================================================================
\section{Incidence paths, $H_1$-association, and obstruction analysis}
\label{sec:incidence-obstruction}
% ====================================================================

The incidence spine introduced in \Cref{subsec:incidence-spines} is a sparse record of how connected components are assembled through minimal-open incidences. Its maximally resolved realization $\mathcal I_1$ is a forest, so it does not itself contain cycles. It does, however, provide distinguished paths through the merger history. This section asks two different questions about those paths. First, when a path can be closed later in the filtration, does the resulting cycle represent a class in ordinary persistent $H_1$? Second, when two endpoints of a long spine path return close to one another in the distance matrix, does the resulting point-level cycle carry geometric evidence of an obstruction?

These questions use the same incidence spine but are deliberately separated. The $H_1$-association is exhaustive and algebraic: every within-tree path is considered, with no close-return or geometric admission rule.  The obstruction analysis is metric and selective: it begins with close-return paths and applies a distance-only geometric criterion, without using homology. The first construction therefore identifies the part of ordinary persistent $H_1$ that has an incidence-spine explanation; the second records a different kind of evidence, namely that a close-return path surrounds a poorly occupied region in the metric data. Detailed algorithmic conventions, thresholds, checkpoint structure, and report fields are given in the accompanying implementation guide.

% ====================================================================
\subsection{The incidence-generated persistence submodule}
\label{subsec:incidence-generated-submodule}
% ====================================================================

Let
$$
P\colon I\longrightarrow \mathbf{Poset}
$$
be the $T_0$-poset filtration and write $K_t=\mathcal K(P_t)$. For $s\leq t$, let
$$
f_{s,t}\colon K_s\longrightarrow K_t
$$
be the simplicial map induced by the continuous identity map of the FTS filtration. Recall that $\mathcal I_1$ denotes the maximally resolved level-one incidence spine. If $a$ and $b$ lie in the same tree of $\mathcal I_1$, write
$$
\gamma_{\mathcal I}(a,b)
$$
for the unique incidence-spine path joining them. The edges of $\gamma_{\mathcal I}(a,b)$ carry more information than their point-level endpoints. Each edge was selected at a merger witnessed by a minimal open $U$, and the construction records distinguished predecessor minimal opens $V^*$ and $W^*$ on the two sides of that merger. Thus the edge records the local incidence relation
$$
V^*\subsetneq U\supsetneq W^*.
$$
At a later level $t$, these three recorded occurrences determine $T_0$-classes in $P_t$ and hence a local comparability bridge
$$
[V^*]_t\leq [U]_t\geq [W^*]_t,
$$
with equality allowed after classes have merged. Reading these bridges in the order in which they occur along $\gamma_{\mathcal I}(a,b)$ gives an abstract incidence template for the whole path.

The template need not be realizable immediately. Let $\tau(a,b)$ be the first filtration level at which all of its pieces can be realized simultaneously and the two endpoint classes can be closed, either because they have become equal or because they are comparable in $P_{\tau(a,b)}$. Consecutive pieces of the template are joined by deterministic shortest paths in the comparability graph at this first realizable level. These connectors are then fixed. After cancelling degenerate edges over the working field $\Bbbk$, the closed realization gives a cycle
$$
c_{a,b,\tau(a,b)}\in Z_1\bigl(K_{\tau(a,b)};\Bbbk\bigr).
$$
For $t\geq\tau(a,b)$ we transport this cycle only by simplicial push-forward,
$$
c_{a,b,t}=(f_{\tau(a,b),t})_\#c_{a,b,\tau(a,b)}.
$$
Consequently, for $s\leq t$ after the cycle has appeared,
$$
(f_{s,t})_\#c_{a,b,s}=c_{a,b,t}.
$$
This compatibility is checked at chain level in the implementation. To collect these cycles before passing to homology, introduce one source generator $e_{a,b}$ at the level $\tau(a,b)$ and retain it thereafter. Thus
$$
E_t^{\rm inc} = \bigoplus_{\{a,b\}\colon\tau(a,b)\leq t}\Bbbk e_{a,b},
$$
and for $s\leq t$ the persistence map $E_s^{\rm inc}\to E_t^{\rm inc}$ sends every generator already present at level $s$ to the same generator at level $t$. Define
$$
\Psi_t\colon E_t^{\rm inc}\longrightarrow Z_1(K_t;\Bbbk),\qquad \Psi_t(e_{a,b})=c_{a,b,t}.
$$
The chain-level compatibility above makes these maps commute with persistence. Composing with the quotient from cycles to homology gives
$$
\Phi_t\colon E_t^{\rm inc}\longrightarrow H_1(K_t;\Bbbk).
$$

\begin{proposition}[Incidence-generated persistence submodule]
\label{prop:incidence-generated-submodule}
The levelwise images
$$
F_t^{\rm inc}=\operatorname{im}\Phi_t\subseteq H_1(K_t;\Bbbk)
$$
form a persistence submodule $F^{\rm inc}$ of the ordinary persistent-homology module $H_1(K;\Bbbk)$.
\end{proposition}

\begin{proof}
Let $s\leq t$ and let $z\in F_s^{\rm inc}$. Write $z=\Phi_s(v)$ for some $v\in E_s^{\rm inc}$. Compatibility of the transported cycles gives
$$
H_1(f_{s,t})(z) = H_1(f_{s,t})(\Phi_s(v)) = \Phi_t(v_t),
$$
where $v_t$ is the image of $v$ in $E_t^{\rm inc}$.  Hence $H_1(f_{s,t})(z)\in F_t^{\rm inc}$.
\end{proof}

The source module $E^{\rm inc}$ and its image $F^{\rm inc}$ play different roles. The source remembers which incidence paths produced cycles, while the image records what those cycles contribute to ordinary persistent homology. Distinct paths may represent the same homology class, and a transported cycle may later become a boundary. For comparison with ordinary persistence, $F^{\rm inc}$ is therefore the relevant object.

% ====================================================================
\subsection{Exhaustive $H_1$-association}
\label{subsec:exhaustive-h1-association}
% ====================================================================

The preceding construction also gives a concrete association procedure.  Every unordered pair of distinct vertices in each tree of $\mathcal I_1$ is evaluated. For each pair, the implementation realizes the incidence template when possible, transports the resulting cycle through the filtration, and expresses its homology class in the ordinary persistent-$H_1$ coordinates. The resulting coordinate sequence is the persistent trace of the incidence path. We call the path $H_1$-associated when this trace is nonzero.

This association is intentionally exhaustive. Endpoint distance, spine-path length, close-return thresholds, shortcut cycles, and the obstruction criterion do not enter the test.  Paths with the same complete persistent-coordinate trace are grouped into one association class. Only after this grouping is the distance matrix used to choose a concise representative for display: among the paths in one class, the primary choice minimizes
$$
d_{\rm close}(a,b)=D(a,b),
$$
with deterministic tie-breaking. This reporting choice does not change the submodule $F^{\rm inc}$. At a level with $\beta_1(t)>0$, the quantity
$$
\dim F_t^{\rm inc}
$$
measures how much of ordinary $H_1(K_t;\Bbbk)$ is spanned by incidence-derived cycles. For a complete run we summarize this by the weighted coverage
$$
\rho_{\rm wt}^{\rm inc} = \frac{\sum_{t\colon\beta_1(t)>0}\dim F_t^{\rm inc}}{\sum_{t\colon\beta_1(t)>0}\beta_1(t)}.
$$
Thus $\rho_{\rm wt}^{\rm inc}=1$ means that incidence-derived cycles span ordinary $H_1$ at every $H_1$-bearing level, while smaller values indicate that part of ordinary persistent homology has no representative arising from the incidence paths considered here. This is a rank-coverage statistic, not a count of geometric holes.

% ====================================================================
\subsection{Running example: \texttt{mixed\_holes} $H_1$-association}
\label{subsec:mixed-holes-h1-association}
% ====================================================================

The \texttt{mixed\_holes} data provide a controlled setting in which both the incidence geometry and the ordinary persistent classes can be inspected directly. We use three seeds and the eight combinations of profile/adaptive scheduling, raw/coherent $A$-sets, and the scores $s/\delta$ and $1/\delta$. The association procedure is run on the completed filtration for each configuration. Across the resulting 24 runs, all $3{,}871{,}882$ within-tree pairs in $\mathcal I_1$ are evaluated. The computation produces $1{,}191{,}366$ nonzero persistent traces, which collapse to $1{,}524$ association classes after grouping by complete persistent coordinates. The reported representatives pass the chain-level transport audit.

\Cref{fig:mixed-holes-h1-association} shows one representative association for the adaptive raw inverse-distance filtration. Panel (a) places the selected path in the full level-one incidence spine and shows the closing pair.  Panel (b) shows the incidence-generated cycle when it first becomes realizable in $\mathcal K(P_{32})$. Panel (c) shows the ordinary persistent-$H_1$ class with which that cycle is associated. The three panels separate the three objects involved in the construction: a path in the incidence spine, its closed realization in an order complex, and the ordinary persistent class detected by the homology calculation. The association is established algebraically by the persistent-coordinate trace; the planar similarity in the figure is explanatory rather than part of the test.

\begin{figure}[!ht]
  \centering
  \includegraphics[
  width=\textwidth,
  height=0.34\textheight,
  keepaspectratio
  ]{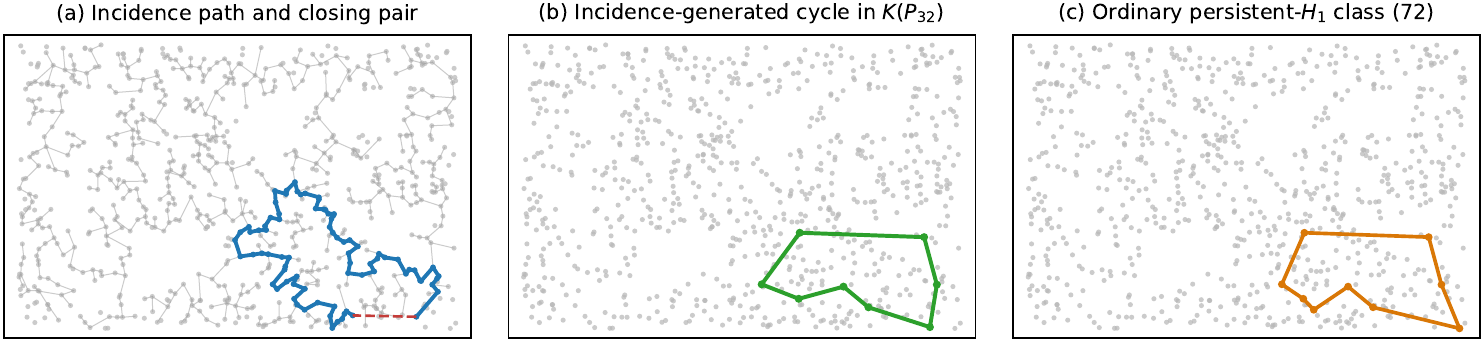}
  \caption{A representative incidence--$H_1$ association in \texttt{mixed\_holes} for the adaptive raw inverse-distance filtration. (a) The selected path in the maximally resolved incidence spine, together with its closing pair. (b) The resulting incidence-generated cycle at its first realization level, $t=32$. (c) The associated class in the ordinary persistent-$H_1$ module. Gray points and edges provide the surrounding data and incidence context; display coordinates are not used by the association procedure.}
  \label{fig:mixed-holes-h1-association}
\end{figure}

\Cref{tab:mixed-holes-h1-association} summarizes the three-seed means. The incidence-generated submodule is substantial in several configurations, but its size depends on how the practical filtration is built. The two raw inverse-distance settings have the largest weighted coverage, $0.888$ for adaptive scheduling and $0.872$ for profile scheduling. The raw $s/\delta$ settings are also high, at $0.786$ and $0.780$.  Coherent settings give smaller coverage here, particularly profile coherent inverse distance. The number of association classes does not track coverage monotonically: many incidence paths can collapse to the same ordinary classes, while a smaller collection of classes can span a comparatively large part of $H_1$.

Across all 24 runs, the mean weighted coverage is $0.676$. The incidence-generated subspace equals ordinary $H_1$ at only $316$ of the $973$ levels for which $\beta_1>0$, and no run is fully covered at every $H_1$-bearing level. No configuration should therefore be read as making incidence paths synonymous with ordinary $H_1$. The point of the construction is more selective: it identifies the portion of ordinary persistent homology that can be traced directly to the merger incidences retained by $\mathcal I_1$.

\begin{table}[!ht]
\centering
\caption{Mean incidence--$H_1$ association results over the three \texttt{mixed\_holes} seeds. ``Assoc. classes'' counts distinct persistent-coordinate traces after grouping.  ``Full levels'' counts levels at which the incidence-generated subspace equals ordinary $H_1$; ``$H_1$ levels'' counts levels with $\beta_1>0$.}
\label{tab:mixed-holes-h1-association}
\begin{tabular}{llrrrr}
\toprule
Schedule & $A$-sets, score & Assoc. classes & Weighted coverage & Full levels & $H_1$ levels\\
\midrule
Adaptive & raw, $1/\delta$ & 70.3 & 0.888 & 17.3 & 36.3\\
Adaptive & raw, $s/\delta$ & 79.0 & 0.786 & 16.7 & 35.7\\
Adaptive & coherent, $1/\delta$ & 34.3 & 0.557 & 4.0 & 42.7\\
Adaptive & coherent, $s/\delta$ & 77.0 & 0.574 & 15.0 & 41.0\\
Profile & raw, $1/\delta$ & 63.3 & 0.872 & 20.3 & 36.3\\
Profile & raw, $s/\delta$ & 81.0 & 0.780 & 13.7 & 37.0\\
Profile & coherent, $1/\delta$ & 23.3 & 0.357 & 1.3 & 47.0\\
Profile & coherent, $s/\delta$ & 79.7 & 0.596 & 17.0 & 48.3\\
\bottomrule
\end{tabular}
\end{table}

% ====================================================================
\subsection{A geometric obstruction criterion}
\label{subsec:geometric-obstruction-criterion}
% ====================================================================

Persistent $H_1$ records a failure of cycles to bound in the order complexes, but it does not by itself say that a cycle surrounds a poorly occupied region in the metric data. The obstruction analysis addresses this second question directly. It starts again from $\mathcal I_1$, but it does not use the incidence--$H_1$ association, ordinary homology, or persistent coordinates.

The basic candidate is a \emph{close return}: two vertices of the same incidence tree that are near in the distance matrix but separated by a sufficiently long spine path. The unique path $\gamma_{\mathcal I}(a,b)$ then gives a long route between endpoints that are metrically close. Closing this route by the short return edge produces a cycle-shaped metric object. For \texttt{mixed\_holes}, small detours in the spine path are first removed by a distance-only shortcut procedure; for COIL20 the complete spine path is used. In either case the purpose of this stage is only to obtain a stable candidate path. A shortcut is not itself evidence of an obstruction.

The geometric test asks whether the closed candidate admits two well-separated transverse chords whose interiors contain certified gaps relative to the complete data set $X$. Long cycle edges are first refined virtually, using only distances inherited from the original distance matrix. Candidate chords are then tested against the two complementary cycle arcs. The shape test requires the arcs to remain sufficiently separated from the chord through substantial middle portions, to meet it transversely near the ends, and to exhibit coherent progression across the chord. Two accepted chords must also cut the cycle into four nondegenerate sectors. Finally, each retained chord is tested against all points of $X$ for a sufficiently long empty middle segment. A close-return path is called \emph{obstruction-indicating} when an admissible pair of shape-valid chords passes this global empty-gap test.

This criterion is intentionally operational. It uses only the distance matrix and a frozen scale profile; the planar coordinates in the figures are used only for display. It does not assert that the data lie on a manifold, nor does it identify a unique geometric hole. The exact close-return scale, shortcut bounds, chord-profile thresholds, virtual subdivision, and empty-gap parameters are recorded in the implementation guide and the experiment outputs. Keeping those details outside the main exposition is useful here because the conceptual distinction is simple: $F_t^{\rm inc}\neq 0$ is a homological statement, whereas the existence of an obstruction-indicating path is a distance-geometric statement. Neither is used to define the other.

% ====================================================================
\subsection{Running example: \texttt{mixed\_holes} obstruction results}
\label{subsec:mixed-holes-obstruction}
% ====================================================================

\Cref{fig:mixed-holes-obstruction} shows the obstruction construction before the aggregate counts. In panel (a), a close-return path is highlighted inside the full incidence spine and closed after the permitted distance-only simplification. Panel (b) shows the two retained transverse chords. The emphasized middle portions are the segments certified to have sufficiently large gaps relative to the complete data set. The figure therefore displays the two stages that matter for interpretation: the incidence spine supplies the candidate route, while the obstruction test asks whether the metric data leave transverse passages through the resulting cycle.

\begin{figure}[!ht]
  \centering
  \includegraphics[
  width=0.92\textwidth,
  height=0.43\textheight,
  keepaspectratio
  ]{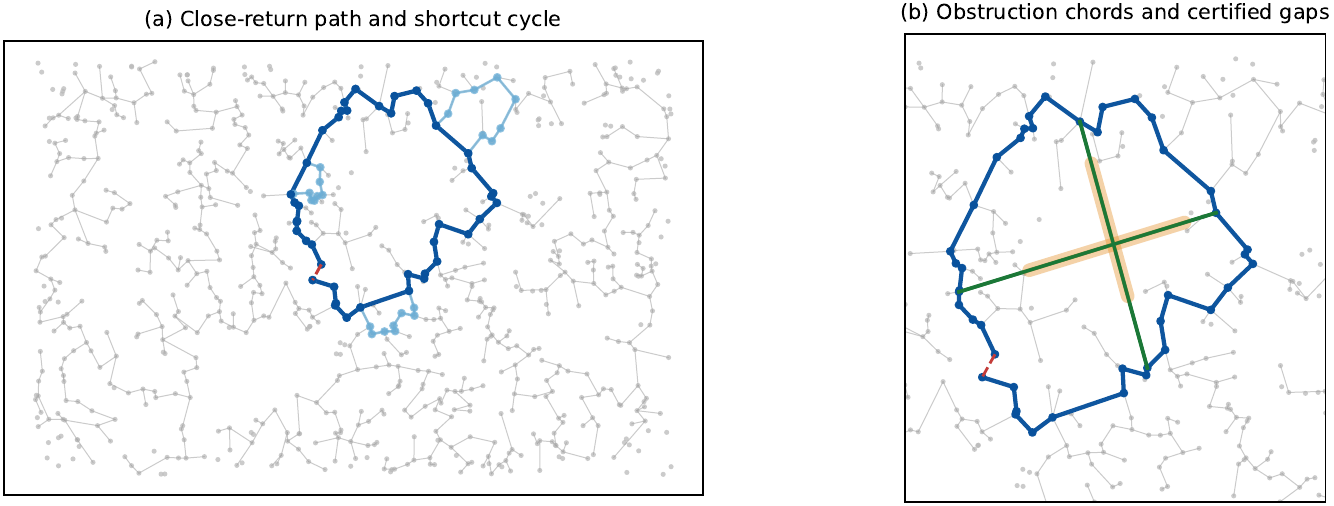}
  \caption{A representative obstruction-indicating close return in \texttt{mixed\_holes} for the adaptive raw inverse-distance filtration. (a) The close-return incidence path and its closed shortcut cycle, shown in the context of the full incidence spine. (b) The two retained obstruction chords; the emphasized portions are the certified empty gaps. The acceptance test uses the distance matrix and the complete blocker population $X$, not the planar display coordinates or an $H_1$ label.}
  \label{fig:mixed-holes-obstruction}
\end{figure}

The same frozen criterion is applied to all 24 \texttt{mixed\_holes} runs. In total there are $6{,}359$ close-return candidates. Of these, $2{,}555$ pass the two-chord shape stage and $1{,}617$ also pass the global empty-gap test. Thus the gap test removes a substantial part of the shape-valid population rather than merely confirming every geometric candidate. Every run nevertheless contains at least one obstruction-indicating path.

\Cref{tab:mixed-holes-obstruction} shows the three-seed means. Raw $A$-sets produce more close-return candidates and more accepted paths in absolute terms. This is consistent with their more finely extended level-one incidence spines. The coherent constructions reduce the candidate population substantially, while still retaining accepted obstruction-indicating paths. These counts should not be interpreted as estimates of the number of planted holes. Several accepted paths may describe the same empty region, and the analysis imposes no one-to-one matching between paths and the visible holes. The role of \texttt{mixed\_holes} is instead to show that the distance-only criterion can reject many plausible cycle shapes while retaining paths that visibly traverse the intended empty-region geometry.

\begin{table}[!ht]
\centering
\caption{Mean obstruction counts over the three \texttt{mixed\_holes} seeds. ``Candidates'' counts close-return paths, ``Shape-valid'' counts candidates admitting an admissible two-chord shape certificate, and ``Obstruction-indicating'' counts those that also pass the global empty-gap test.}
\label{tab:mixed-holes-obstruction}
\begin{tabular}{llrrr}
\toprule
Schedule & $A$-sets, score & Candidates & Shape-valid & Obstruction-indicating\\
\midrule
Adaptive & coherent, $1/\delta$ & 193.7 & 86.7 & 57.3\\
Adaptive & coherent, $s/\delta$ & 232.7 & 105.0 & 67.0\\
Adaptive & raw, $1/\delta$ & 358.0 & 129.0 & 86.0\\
Adaptive & raw, $s/\delta$ & 367.3 & 128.0 & 80.0\\
Profile & coherent, $1/\delta$ & 87.3 & 42.7 & 28.7\\
Profile & coherent, $s/\delta$ & 174.7 & 86.7 & 46.3\\
Profile & raw, $1/\delta$ & 374.3 & 148.0 & 99.3\\
Profile & raw, $s/\delta$ & 331.7 & 125.7 & 74.3\\
\bottomrule
\end{tabular}
\end{table}

% ====================================================================
\subsection{Case study: COIL20}
\label{subsec:coil20-incidence-obstruction}
% ====================================================================

The controlled planar example makes the constructions easy to see, but it does not show why the distinction between incidence organization, persistent $H_1$, and geometric obstruction matters for nonplanar data. COIL20 provides that test \cite{coil20}. The data contain 20 objects, each observed in 72 views at equally spaced acquisition angles.  \Cref{fig:coil20-objects} displays one reference view of each object and fixes the numbering used below.

Each image is converted to grayscale, resized to $32\times32$, flattened to a vector in $[0,1]^{1024}$, and compared by Euclidean distance. Each object is then analyzed separately.  The filtration and every downstream construction in this section receive only the resulting $72\times72$ distance matrix. Acquisition angle is used only to arrange the 72 observations around a circle in the incidence-spine display. In particular, cyclic acquisition order is not supplied to the filtration or used as a target that the incidence spine is expected to recover.

\begin{figure}[!ht]
  \centering
  \includegraphics[
  width=\textwidth,
  height=0.45\textheight,
  keepaspectratio
  ]{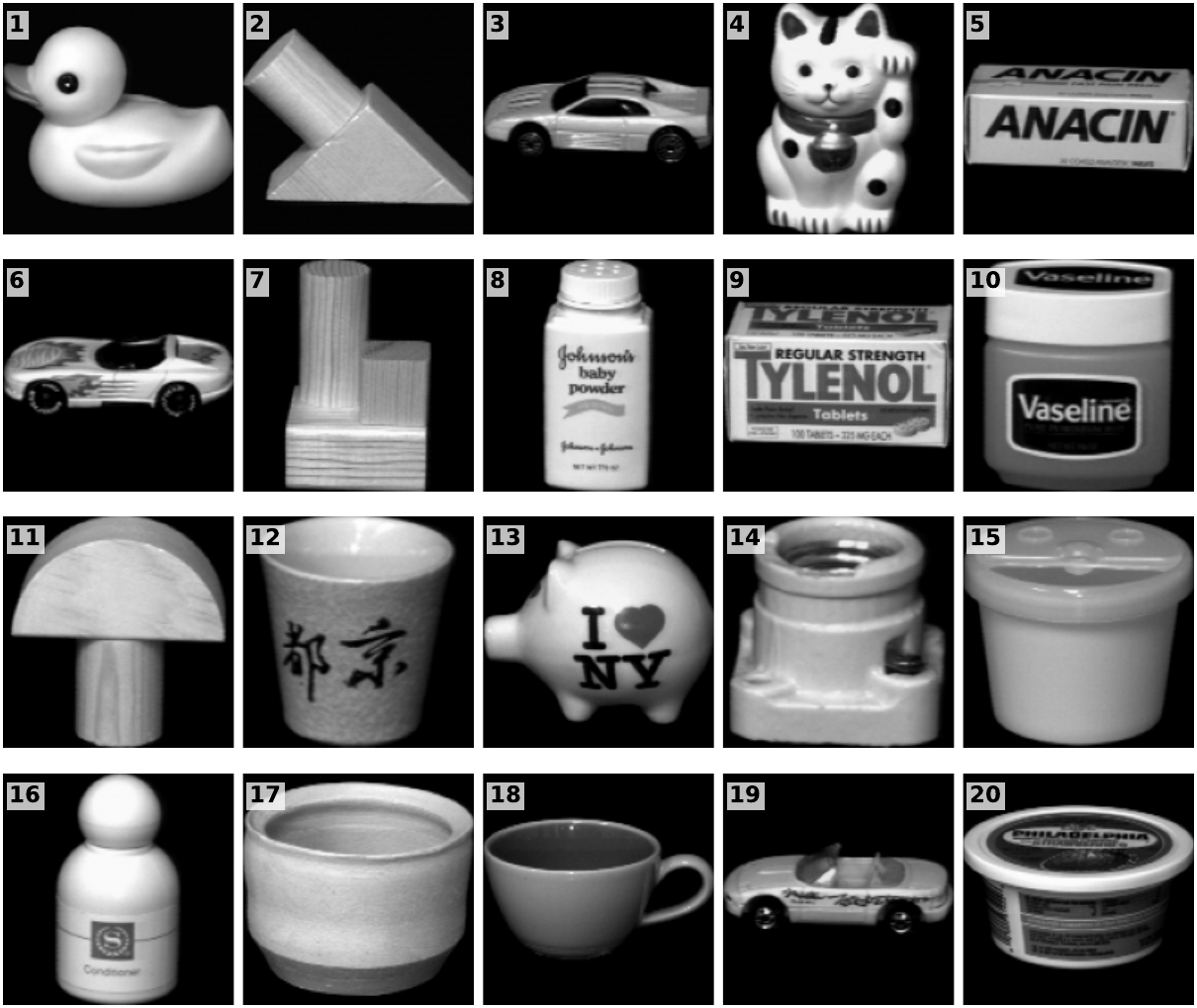}
  \caption{Frontal views of the 20 objects in the COIL20 data set \cite{coil20}. The numbering is retained in \Cref{fig:coil20-incidence-spines}. Each object contributes 72 views to its own distance matrix and is analyzed independently.}
  \label{fig:coil20-objects}
\end{figure}

Because there are only 72 observations per object, the adaptive construction uses $k\in[6,22]$, while the profile construction uses the two windows $[6,14]$ and $[14,22]$. Raw and coherent $A$-sets are compared with the scores $s/\delta$ and $1/\delta$, giving eight configurations for each object. The filtration has 20 A-removal levels followed by 20 B-completion levels, hence 160 runs in total.

\paragraph{Incidence-spine morphology.}
\Cref{fig:coil20-incidence-spines} shows the maximally resolved incidence spines for the adaptive raw inverse-distance setting. The vertices are placed in acquisition order around the circle, so a short edge near the circumference connects nearby viewing angles while a long chord connects views that are far apart in acquisition order. This display makes the incidence organization directly comparable with the object atlas, but the circular placement is not used in constructing the spine.

The 20 objects show markedly different patterns. Some spines remain close to an acquisition-ordered path, whereas others contain repeated long-range connections or strong branching. Objects 1 and 4 illustrate the sparse end of this range, while objects such as 9, 11, 14, and 19 exhibit much denser long-range incidence. These differences are not a statement that one object is intrinsically more or less circular than another. They record how image similarity and the finite-topological merger history reorganize the 72 views. Thus the incidence spine provides information that is not visible from an ordinary barcode alone: it says which view groups are being connected when the filtration assembles the object.

\begin{figure}[!ht]
  \centering
  \includegraphics[
  width=\textwidth,
  height=0.52\textheight,
  keepaspectratio
  ]{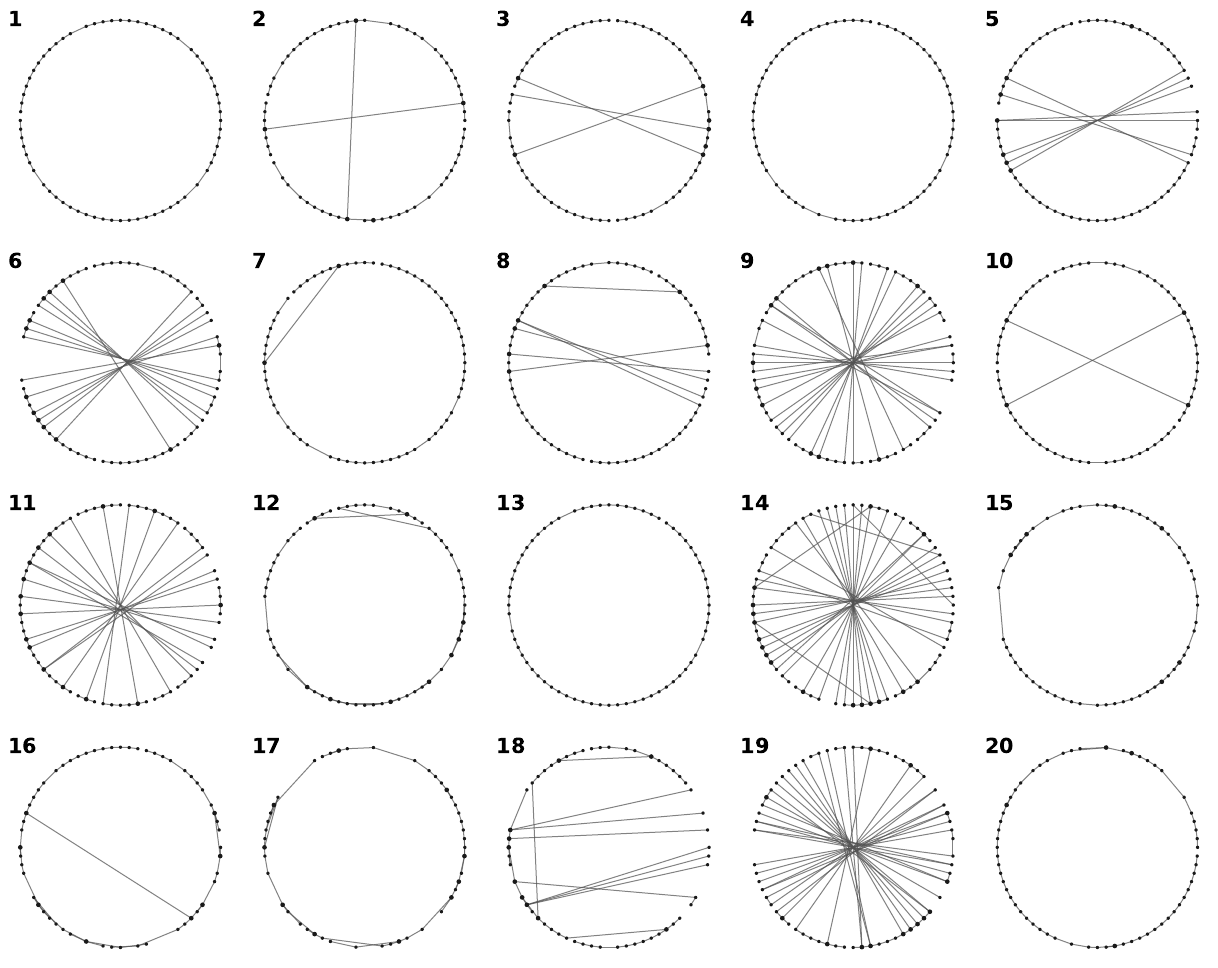}
  \caption{Maximally resolved incidence spines for the 20 COIL20 objects under the adaptive raw inverse-distance configuration. For each object, the 72 views are placed in acquisition order around a circle and the incidence edges are drawn between the corresponding positions. The circular placement is used only for interpretation; the incidence spine is computed from the image-distance filtration.  The numbering agrees with \Cref{fig:coil20-objects}.}
  \label{fig:coil20-incidence-spines}
\end{figure}

\paragraph{$H_1$ association and obstruction across objects.}
\Cref{tab:coil20-joint} compares the three downstream summaries over all 20 objects. Ordinary $H_1$ is present in $129$ of the 160 runs, and at least one incidence association is found in $121$. Thus incidence paths account for persistent degree-one information in many objects and configurations, but not automatically in every $H_1$-positive run. This is the same distinction seen in \texttt{mixed\_holes}: $F^{\rm inc}$ is a data-derived submodule of ordinary persistent $H_1$, not a reformulation of the full module.

The obstruction analysis gives a different picture. Across the eight configurations, the number of objects with at least one accepted obstruction ranges from 3 to 9, for a total of 47 obstruction-positive runs.  These counts are much smaller than the corresponding $H_1$ and association counts.  Moreover, the three summaries vary differently with the topology setting. For example, adaptive raw $s/\delta$ gives ordinary $H_1$ and an incidence association for 17 of the 20 objects but an obstruction for only 3, whereas adaptive raw inverse distance gives 13 objects with $H_1$, all 13 with an incidence association, and 9 with an obstruction. The descriptors therefore answer genuinely different questions.

The obstruction calculation also exposes an important limitation of the metric certificate. Across all COIL20 runs, 131 candidates pass the two-chord shape stage, and all 131 also pass the empty-gap stage. With only 72 points in the $1024$-dimensional image space, an empty tube around an already shape-valid chord is much easier to obtain than in the 800-point planar example. For COIL20 the accepted obstruction should therefore be read primarily as evidence from the frozen shape profile; the gap test is not strongly discriminating at this sampling density. We retain the same criterion rather than retuning it after inspecting the object identities.

\begin{table}[!ht]
\centering
\caption{COIL20 comparison over the 20 objects. ``Path-like'' counts objects whose connected level-one incidence spine has no branch vertex and maximum degree at most two; ``Mean branches'' is the mean number of vertices of degree at least three. ``$H_1$'', ``Assoc.'', and ``Obs.'' count objects with nonzero ordinary $H_1$, at least one incidence--$H_1$ association, and at least one obstruction-indicating path, respectively.}
\label{tab:coil20-joint}
\begin{tabular}{llrrrrr}
\toprule
Schedule & $A$-sets, score & Path-like & Mean branches & $H_1$ & Assoc. & Obs.\\
\midrule
Adaptive & raw, $s/\delta$ & 0 & 11.05 & 17 & 17 & 3\\
Adaptive & coherent, $s/\delta$ & 0 & 7.85 & 19 & 18 & 8\\
Adaptive & raw, $1/\delta$ & 3 & 7.45 & 13 & 13 & 9\\
Adaptive & coherent, $1/\delta$ & 0 & 6.95 & 14 & 12 & 5\\
Profile & raw, $s/\delta$ & 0 & 9.60 & 18 & 17 & 4\\
Profile & coherent, $s/\delta$ & 0 & 6.50 & 19 & 18 & 4\\
Profile & raw, $1/\delta$ & 2 & 7.35 & 14 & 13 & 8\\
Profile & coherent, $1/\delta$ & 2 & 5.85 & 15 & 13 & 6\\
\bottomrule
\end{tabular}
\end{table}

Taken together, \Crefrange{fig:coil20-objects}{fig:coil20-incidence-spines} and \Cref{tab:coil20-joint} clarify the role of the incidence lens. The acquisition sequence supplies an external ordering that is useful for interpretation, while the FTS filtration supplies a different organization determined by image distances. The incidence spine records how that organization is assembled; $F^{\rm inc}$ records which ordinary persistent-$H_1$ classes can be generated from its paths; and the obstruction analysis asks whether selected close returns carry a particular distance-geometric certificate. None of these descriptors is required to agree with the others. Their disagreement is informative because it separates hierarchical assembly, homological persistence, and metric evidence for a poorly occupied region within one common finite-topological filtration.

% ====================================================================
\section{Conclusion}
\label{sec:conclusion}
% ====================================================================

We have introduced finite-topological-space filtrations as a framework for topological data analysis. Starting from a nested family of finite topologies, the construction retains the points of the data set together with their evolving minimal opens and $T_0$-classes. At the same time, $T_0$-reduction and order complexes provide a canonical route to persistent
homology.  Thus persistence is available within the framework, but it is not the only information carried by the filtration.

The examples developed in this paper illustrate this broader viewpoint. Minimal opens, $T_0$-classes, connected components, and their evolution give direct access to grouping and transitional structure. Because these objects remain attached to the original observations, they can also be combined, restricted, traced, and compared to define further descriptors. On \texttt{mixed\_bumps}, comparison with ToMATo shows that the selected $T_0$-classes correspond closely to localized persistence-based clusters, while minimal-open information extends them to separating supports that closely follow the full ToMATo clusters without being required to form a partition. The incidence construction provides another example: it records how finite-topological organization changes across the filtration and leads both to a distinguished submodule of persistent $H_1$ and, independently, to geometric obstruction analysis.

The main point is therefore not to replace existing persistent-homology constructions by another route to barcodes, but to retain a richer evolving topological object from which barcodes and other data-dependent features can be derived. The constructions studied here give several first examples of such features rather than an exhaustive list. Developing further descriptors from evolving minimal opens, $T_0$-classes, and their incidence structure, together with appropriate notions of stability for them, is a natural direction for future work.

\section*{Declaration of generative AI and AI-assisted technologies in the manuscript preparation process}

During the preparation of this work, the author used ChatGPT to assist with coding, clarification of mathematical concepts, discussion and testing of proposed constructions, preparation of scripts for generating computational tables, and improvements to the language, exposition, and organization of the manuscript. The author reviewed and revised the AI-assisted material as appropriate and takes full responsibility for the content of the manuscript.

% --------------------------------------------------------------------
% Bibliography
% --------------------------------------------------------------------

% \bibliographystyle{amsplain}
% \bibliography{references}

\end{document}